\documentclass[11pt]{article}
\usepackage{amsmath}
\usepackage{amsthm}
\usepackage{amssymb}
\usepackage{graphicx}
\usepackage{mathtools}
\usepackage[margin=2.5cm]{geometry}
\usepackage{setspace}
\usepackage{empheq}
\usepackage{algorithm}
\usepackage{algorithmicx}
\usepackage{algpseudocode}
\usepackage{multirow}
\usepackage{url}
\usepackage[title]{appendix}
\usepackage{enumitem}
\usepackage[svgnames]{xcolor}
\usepackage{bbm}
 \usepackage{diagbox}
\usepackage{subcaption}
\usepackage[open]{bookmark}

\makeatletter

\theoremstyle{plain}
\newtheorem{thm}{\protect\theoremname}
  \theoremstyle{definition}
  \theoremstyle{plain}
  \newtheorem{lem}[thm]{\protect\lemmaname}
  \newtheorem{prop}[thm]{\protect\propname}
  \newtheorem{cor}[thm]{\protect\corollaryname}

  \providecommand{\propname}{Proposition}
  \providecommand{\lemmaname}{Lemma}
  \providecommand{\corollaryname}{Corollary}
\providecommand{\theoremname}{Theorem}

\makeatother

\usepackage{endnotes}
\usepackage{color}
\usepackage{xparse}

\newtheorem{remark}{Remark}
\newtheorem{assump}{Assumption}
\newtheorem{prob}{Problem}
\newtheorem{example}{Example}
\newtheorem{defn}{Definition}

\begin{document}

\title{Local Two-Sample Testing over Graphs and Point-Clouds \\ by Random-Walk Distributions}

\author{ Boris Landa${^{1,*}}$,~~Rihao Qu${^{2,4,5}}$,~~Joseph Chang${^{3}}$,~~and~~Yuval Kluger${^{1,2,4}}$\\
\small{${^1}$Program in Applied Mathematics, Yale University}\\
\small{${^2}$Department of Pathology, Yale University School of Medicine}\\
\small{${^3}$Department of Statistics and Data Science, Yale University}\\
\small{${^4}$Interdepartmental Program in Computational Biology and Bioinformatics, Yale University}\\
\small{${^5}$Department of Immunology, Yale School of Medicine}\\
\small{${^*}$Corresponding author. Email: boris.landa@yale.edu}
}

\maketitle

\begin{abstract}
Rejecting the null hypothesis in two-sample testing is a fundamental tool for scientific discovery. Yet, aside from concluding that two samples do not come from the same probability distribution, it is often of interest to characterize how the two distributions differ. Given samples from two densities $f_1$ and $f_0$, we consider the task of localizing occurrences of the inequality $f_1 > f_0$. 
To avoid the challenges associated with high-dimensional space, we propose a general hypothesis testing framework where hypotheses are formulated adaptively to the data by conditioning on the combined sample from the two densities.
We then investigate a special case of this framework where the notion of locality is captured by a random walk on a weighted graph constructed over this combined sample.
We derive a tractable testing procedure for this case employing a type of scan statistic, 
and provide non-asymptotic lower bounds on the power and accuracy of our test to detect whether $f_1>f_0$ in a local sense. Furthermore, we characterize the test's consistency according to a certain problem-hardness parameter, and show that our test achieves the minimax detection rate for this parameter.
We conduct numerical experiments to validate our method, and demonstrate our approach on two real-world applications: detecting and localizing arsenic well contamination across the United States, and analyzing two-sample single-cell RNA sequencing data from melanoma patients. 
\end{abstract}

\section{Introduction} \label{sec:inroduction}
A prototypical situation native to two-sample testing is an experiment that produces two sets of observations, and the goal is to determine whether they were generated by the same underlying distribution. 
While standard two-sample tests can at most provide a negative answer, namely, to reject (with prescribed significance) the null hypothesis that the two distributions are identical, it is often of interest to provide more detailed information on the differences between the distributions. In particular, the difference between the distributions may be minute or nonexistent everywhere but in a few small regions of the sample space. In such a case it is of interest to identify these regions, and to describe the relation between the distributions in these regions (e.g., if one of the distributions is uniformly larger than the other). 
Indeed, local variants of two-sample testing are required in numerous scientific applications, such as single-cell RNA sequencing~\cite{zhao2020detection}, astronomy~\cite{freeman2017local}, and climate analysis~\cite{michel2020high}. 
Many applications of interest, including the above, involve high-dimensional data. 
We note, however, that density estimation in high dimension is prohibitive due to the curse of dimensionality. Hence, local analysis of the differences between two high-dimensional distributions is a particularly challenging task.

In this work, we consider the following setting. Let $f_1$ and $f_0$ be two probability density functions defined on a measurable space $\mathcal{X}$, and suppose that $X$ is a random variable that is generated by sampling from $f_1$ with probability $p\in(0,1)$ and sampling from $f_0$ with probability $1-p$.
In other words, given a \textit{class} random variable $Z\sim \text{Bernoulli}(p)$, where $p$ is the class prior, we have
\begin{equation}
    (X \mid Z=1) \sim f_1, \qquad \text{and} \qquad (X\mid Z=0) \sim  f_0. \label{eq:x_i model}
\end{equation}
We then take $z_1,\ldots,z_n$ and $x_1,\ldots,x_n$ as i.i.d. samples from the joint distribution of $Z$ and $X$, where $z_i$ as the class label of the sample $x_i$. 
Given the labels $z_1,\ldots,z_n$ and the samples $x_1,\ldots,x_n$, the goal in standard two-sample testing to determine whether the densities $f_1(x)$ and $f_0(x)$ are identical, i.e., to attempt to reject the null hypothesis $f_1(x)=f_0(x)$ for all $x\in\mathcal{X}$, with prescribed significance. Here, we are interested in a refined local variant of this task, which is, loosely speaking, to determine whether $f_1(x)>f_0(x)$ in each of possibly many regions of interest in the sample space, noting that determining the other direction of the inequality, i.e., $f_1(x)<f_0(x)$, can be considered equivalently by interchanging the roles of $f_1$ and $f_0$. We refer to this task as \textit{local two-sample testing}, and propose a precise problem formulation and hypothesis testing framework in Section~\ref{sec:approach and problem formulation}.  

While two-sample testing has a long history and has been extensively studied in the context of high dimension, general-purpose variants for local two-sample testing have been considered only recently~\cite{kim2019global,hediger2019use,he2019instance,cheng2019classification,cazais2015beyond}, mostly through the machinery of classification and regression. In particular, \cite{kim2019global} proposed a general framework for this task using regression, where the idea is to regress the probabilities of the labels $z_1,\ldots,z_n$ conditioned on $x_1,\ldots,x_n$, and to compute a pointwise statistic from the discrepancy between the class prior $p$ and the output of the regression model at a given point $x$. Then, each such statistic is compared to its distribution under the (global) null hypothesis: $f_1(x) = f_0(x)$ for all $x\in\mathcal{X}$, using a permutation test combined with a multiple testing procedure. This methodology allows one to reject the global null and accept a pointwise alternatives of the form $f_0(x)\neq f_1(x)$ for a specific $x$, and hence acts a local two-sample test. The performance of the resulting local test in~\cite{kim2019global} was analyzed in terms of both the (global) type I error and the (local) type II error, for various regressions models, and under certain smoothness and structural assumptions on $f_0$ and $f_1$. 

A major limitation of the approach mentioned above is that even if the test rejects the global null hypothesis and accepts a pointwise alternative at $x_i$, we cannot conclude that $f_1(x_i) \neq f_0(x_i)$ with finite-sample probabilistic guarantees.
Indeed, for a given finite sample, a null hypothesis of the form $f_0(x_i) = f_1(x_i)$ cannot be decisively rejected unless $f_0$ and $f_1$ are known to be sufficiently smooth, since under the pointwise null $f_0(x_i) = f_1(x_i)$, $z_i$ is $1$ with probability $p$ and is $0$ otherwise, regardless of the distributions of the other labels.
Clearly, when the test rejects the global null hypothesis and accepts an alternative at $x_i$, the rejection relies on labels from a certain neighborhood of $x_i$. However, it is challenging to determine what that neighborhood is, as it depends implicitly on the regression model used and the unknown densities $f_0$ and $f_1$, thereby limiting the interpretability of the test's outcome.

A different methodology that is closely related to local two-sample testing, and is able to address the above-mentioned limitation, is \textit{scan statistics}~\cite{glaz2012scan}, and particularly \textit{spatial} scan statistics~\cite{kulldorff1997spatial,kulldorff2006elliptic,duczmal2006evaluation}, where hypotheses are associated with explicit spatial neighborhoods.
Spatial scan statistics is a widely-used methodology for detecting regions, usually in one or two dimensions, in which observed values of random variables suggest departure from a null hypothesis. For instance, suppose that each random variable is a count of disease cases in a town. Then, a typical procedure in spatial scan statistics is to compute statistics summarizing the total number of disease cases in certain predefined geographic regions (each containing multiple towns), and to scan over the different regions and their summarizing statistics to identify regions of elevated illness~\cite{kulldorff1995spatial,kulldorff2005space}. 
In the context of our setting, one can detect that $f_1(x)>f_0(x)$ in a certain region of $\mathcal{X}$ if the corresponding local proportion of labels satisfying $z_i=1$ (with $x_i$ in that region) significantly exceeds the global proportion $p$.
In most cases of spatial scan statistics, the geographic regions are defined by simple shapes in the ambient space, such as intervals in one dimension, and rectangles, circles, or ellipses in two dimensions, all of which vary in their location (e.g., center) and scale (size). The use of different scales is fundamental, since larger scales correspond to larger regions that include more random variables, allowing for more power to detect weak deviations from the null (as the number of random variables associated with a geographic region is analogous to the sample size observed locally). 

Due to the curse of dimensionality, standard techniques in spatial scan statistics are restricted to low-dimensional Euclidean space. Indeed, covering the ambient space with simple predefined shapes (e.g., balls that vary in their location and scale) is prohibitive in high dimension, and the corresponding scan procedure becomes infeasible, both from a computational perspective and from that of multiple testing. Moreover, many real-world datasets are known to be intrinsically low-dimensional and admit a non-Euclidean geometry, i.e., reside on or near a low-dimensional Riemannian manifold embedded in the ambient space. Arguably, in such cases it is inappropriate to define the regions of interest according to the high-dimensional ambient space, as the underlying notion of locality can be fundamentally different. Therefore, in such cases the notion of locality should be inferred directly from the data, allowing for a data-driven approach to scientific discovery.

Aside from simple spatial domains, the methodology of scan statistics was also extended to graph domains~\cite{arias2008searching,arias2011detection}, which are widely-used in data analysis for representing high-dimensional datasets that are intrinsically low-dimensional. In the setup of graph scan statistics, a random variable is attached to each node in a graph, and the goal is to detect a community in the graph that behaves anomalously. In~\cite{arias2011detection}, it was proposed to define communities as arbitrary sub-graphs that admit certain connectivity constraints. In this setting, the null hypothesis is defined by having identically distributed random variables over the graph nodes (e.g., the variables are all normal with variance $1$ and mean $0$), while an alternative (the anomaly) is a piecewise constant model for the parameters of the random variables (e.g., the random variables are all normal with variance $1$ and mean $0$ outside the anomalous community, and some positive constant mean inside). Then, a statistic indicative of an anomaly for each community is computed (typically a generalized likelihood ratio), and the maximal statistic from all communities is compared to its distribution under the null. Since the number of communities defined via graph connectivity constraints is typically very large, scanning over all communities (or even a representative subset of them) is prohibitive, and various other approaches have been proposed; see~\cite{addario2010combinatorial,sharpnack2015detecting,sharpnack2013near,sharpnack2013changepoint,qian2014connected,sharpnack2013detecting,chen2017detecting,cadena2019near} and references therein for more details. Consequently, probabilistic guarantees in this line of work are generally limited to the inference task of determining whether at least one community in the network is anomalous, and do not address the localization of the anomalous communities.

In this work, we provide a general framework for local two-sample testing that is founded on the idea of conditioning on the given samples $x_1,\ldots,x_n$. In particular, the hypotheses are formulated as inequalities concerning the values of the densities $f_0(x)$ and $f_1(x)$ evaluated only for $x=x_1,\ldots,x_n$, rather than for all $x$ in the ambient space.
This idea allows us to avoid the curse of dimensionality by determining the subsets of the samples $x_1,\ldots,x_n$ that are employed for testing $f_1(x)>f_0(x)$ adaptively (e.g., via clustering or near-neighbour graphs), instead of using predefined regions in the ambient space. 
More generally, our framework allows for testing whether certain weighted averages of $\{f_1(x_i)-f_0(x_i)\}_{i=1}^n$ are strictly positive, where the weights come from a family that may depend on the samples $x_1,\ldots,x_n$. The weights in the family should be chosen as to emphasize subsets of $x_1,\ldots,x_n$ according to a notion of locality (e.g., a measure of similarity between the samples provided by a domain expert), and obviates the need for explicit model assumptions on the densities $f_1$ and $f_0$.
The main advantage of our approach is that it allows for interpretable local null hypotheses that are concerned with the differences $\{f_1(x_i)-f_0(x_i)\}_{i=1}^n$, and can be rejected with finite-sample probabilistic guarantees with the only assumption that the labels conditional on the samples are independent. When rejected, these null hypotheses provide neighborhoods in which $f_1>f_0$.

As a useful special case of our framework, we leverage a weighted graph over the sample $\{x_1,\ldots,x_n\}$ to capture the geometry of the data. In this setting, the random variables associated with the graph nodes are the labels $z_1,\ldots,z_n$ conditioned on the samples $x_1,\ldots,x_n$, and the goal is to find neighborhoods in the graph where the local proportion of labels satisfying $z_i = 1$ is significantly larger than their global proportion $p$. We generalize the notion of a neighborhood of a data point to a probability distribution arising from a random walk started at this data point, and show that it provides a natural concept of location and scale (analogously to regions typically used in spatial scan statistics). 
Crucially, we show that this usage of a random walk allows for a computationally tractable testing procedure and favorable theoretical guarantees on the power and accuracy associated with detecting all the neighborhoods in which $f_1>f_0$. 
We also establish the optimality of the detection rate for our proposed testing procedure in a minimax sense, and demonstrate our approach on both simulated and experimental data.

The main contributions of this work are summarized as follows:
\begin{enumerate}
    \item \textbf{Framework.} We propose a general framework for local two-sample testing that avoids the curse of dimensionality and the need for explicit model assumptions on the densities $f_0$ and $f_1$. The hypotheses in our framework are adapted to the data by conditioning on the observations $x_1,\ldots,x_n$, and allow for finite-sample probabilistic guarantees in determining that $f_1>f_0$ in a certain local sense.
    \item \textbf{Testing procedure.} We derive a computationally-tractable testing procedure for determining whether $f_1>f_0$ in neighborhoods defined according to a random walk over the sample. We identify that a symmetric and positive semidefinite transition probability matrix is particularly well suited for this task, as it admits many useful properties and allows all random walk distributions to be well approximated by a set with controlled cardinality.
    \item \textbf{Analysis.} We analyze the performance of the testing procedure according to our hypothesis testing framework, and derive bounds on the power and the accuracy of the method to detect that $f_1>f_0$ in a given neighborhood. We also show that our method achieves the minimax detection rate in terms of a certain problem hardness parameter among all methods that are locally consistent as $n\rightarrow\infty$ (a concept defined in our framework via an appropriate risk function).
    \item \textbf{Examples.} We demonstrate the performance of our approach and the fit to our theory through numerical simulations, and show the practical usefulness of our method to extract novel scientific insights in real-world applications.
\end{enumerate}

\section{Problem formulation and main results}
Throughout this work we assume, unless stated otherwise, that all quantities are conditioned on $x_1,\ldots,x_n$ by default. Therefore, the only source of randomness is in the labels $z_1,\ldots,z_n$, and in particular, each label $z_i$ is independently sampled from $(Z \mid X=x_i)$, where $Z$ and $X$ are as described in Section~\ref{sec:inroduction}. Specifically, according to Bayes' law, we have
\begin{equation}
    (Z \mid X=x) = 
    \begin{dcases}
    1, & \text{with probability} \quad \frac{f_1(x) p}{f(x)},\\
    0, & \text{with probability} \quad \frac{f_0(x) (1-p)}{f(x)},
    \end{dcases} \label{eq:z_i distribution}
\end{equation}
where $f(x)$ is the marginal distribution of $X$, namely
\begin{equation}
    f(x)  = p f_1(x) + (1-p) f_0(x). \label{eq:f def}
\end{equation}
In other words, $x_1,\ldots,x_n$ are first sampled independently from their marginal distribution $f(x)$, after which they are fixed, and the labels $z_1,\ldots,z_n$ are sampled from independent but not identically distributed Bernoullis according to~\eqref{eq:z_i distribution}, which constitute the random variables in our setting. As motivated in the introduction, the conditioning on the samples $x_1,\ldots,x_n$ allows our hypothesis testing framework to be stated in terms of the samples, without having to consider the ambient space $\mathcal{X}$ directly.

\subsection{Problem formulation} \label{sec:approach and problem formulation}
\subsubsection{Local two-sample testing by distributions over the sample}
Given $x_1,\ldots,x_n$, we encapsulate the idea of ``regions of interest'' in a family of discrete probability distributions $\mathcal{F}$ over the sample $\{x_1,\ldots,x_n\}$. Specifically, $\mathcal{F}$ consists of a collection, possibly infinite, of nonnegative weight vectors of the form $\mathbf{w} = [w_1,\ldots,w_n]$, each satisfying $\sum_{i=1}^n w_i = 1$. In addition, we assume that
\begin{assump} \label{assump:F independent of labels}
$\mathcal{F}$ is deterministic given $x_1,\ldots,x_n$.
\end{assump}
That is, Assumption~\ref{assump:F independent of labels} requires that $\mathcal{F}$ depend only on the samples $x_1,\ldots,x_n$ and \textit{not} on the labels $z_1,\ldots,z_n$. This assumption is crucial for formulating our hypothesis testing framework and for providing finite-sample probabilistic guarantees on a testing procedure that uses the distributions in $\mathcal{F}$.
Loosely speaking, we consider each probability distribution $\mathbf{w}\in\mathcal{F}$ as describing a neighborhood of $\{x_1,\ldots,x_n\}$ according to the samples $x_i$ for which the corresponding value ${w}_i$ is sufficiently large. 
To further clarify this point, consider the following two examples.
\begin{example} \label{Example:forming F using clustering}
Let $\mathcal{C}_1,\ldots,\mathcal{C}_L$ be a clustering of $x_1,\ldots,x_n$ to $L$ disjoint clusters (by, e.g., $k$-means, spectral clustering, etc.), where $\mathcal{C}_\ell$ is the set of indices of the samples in the $\ell$'th cluster. Then, a corresponding choice of $\mathcal{F}$ is $\mathcal{F} = \{\mathbbm{1}_{\mathcal{C}_1}/ | \mathcal{C}_1 | , \ldots, \mathbbm{1}_{\mathcal{C}_L}/ | \mathcal{C}_L | \}$, where $\mathbbm{1}_{\mathcal{C}_\ell}$ is the indicator vector over $\mathcal{C}_\ell$, and $ | \mathcal{C}_\ell | $ is the size of the $\ell$'th cluster. 
\end{example}
The distributions in $\mathcal{F}$ do not have to be uniform nor defined over disjoint subsets of $\{x_1,\ldots,x_n\}$. As another example, $\mathcal{F}$ can be defined via a nonnegative kernel such as the Gaussian kernel, assuming that a metric $D_{\mathcal{X}}( \cdot, \cdot )$ over $\mathcal{X}$ is given, as follows. 
\begin{example} \label{Example:forming F using a kernel}
Suppose that $K\in\mathbb{R}^{n\times n}$ is given by $K_{i,j} = \operatorname{exp}\{-D^2_{\mathcal{X}}(x_i,x_j) /\sigma\}$, for some $\sigma>0$. Then, one can take $\mathcal{F} = \{K_1/\sum_{j} K_{1,j},\ldots,K_n/\sum_j K_{n,j}\}$, where $K_i$ is the $i$'th row of $K$. 
\end{example} 
We remark that $\mathcal{F}$ can be formed by combining the constructions from the above two examples using multiple kernels and clustering choices, thereby providing a rich multi-resolution description of the sample. For instance, one may use $\mathcal{F}$ from Example~\ref{Example:forming F using clustering} with different numbers of clusters $L$, together with $\mathcal{F}$ from Example~\ref{Example:forming F using a kernel} using different values of $\sigma$, and take the union of all resulting distributions. We propose a specific construction for $\mathcal{F}$ that similarly allows for a multi-resolution description of the sample in Section~\ref{sec: random-walk distributions}.

Next, to define a measure of discrepancy between $f_1$ and $f_0$, we employ the function
\begin{equation}
    s(x) := \mathbb{E} \left[Z \mid X=x\right] - p = \operatorname{Pr}\{ Z = 1 \mid X = x\} - p =  p(1-p)\frac{ f_1(x)-f_0(x) }{f(x)}, \label{eq:s def}
\end{equation}
where $f(x)$ is from~\eqref{eq:f def}, and $s(x)$ is defined only for $x\in\mathcal{X}$ for which $f(x)>0$.
Evidently, $s(x)$ is bounded, and in particular $-p \leq s(x) \leq 1-p$,
where $s(x)$ attains its maximum $1-p$ whenever $f_0(x)=0$, and attains its minimum $-p$ whenever $f_1(x)=0$. Therefore, the function $s(x)$ can be thought of as a normalized pointwise difference between $f_1$ and $f_0$. We note that $s(x)$ was also used as the target of the regression model in~\cite{kim2019global} for the purpose of local two-sample testing. 

Using $\mathcal{F}$ and the function $s(x)$, we consider the following problem.
\begin{prob} \label{prob:problem def}
Given $z_1,\ldots,z_n$ and $\mathcal{F}$, determine for which $\mathbf{w} \in \mathcal{F}$, if any, $\sum_{i=1}^n w_i s(x_i) > 0$.
\end{prob}
The motivation for using the quantity $\sum_{i=1}^n w_i s(x_i)$ in Problem~\ref{prob:problem def} is as follows. Suppose for simplicity that $\mathcal{X}$ is the Euclidean space $\mathbb{R}^D$, and let $\omega(x)$ be a nonnegative bounded function on $\mathbb{R}^D$. Since $x_1,\ldots,x_n$ were sampled independently from the marginal density $f(x)$, for sufficiently large $n$ we have
\begin{equation}
    \frac{1}{n} \sum_{i=1}^n \omega(x_i) s(x_i) \approx p(1-p)\int_{ \mathbb{R}^D} \omega(x) \left( f_1(x) - f_0(x) \right) dx. \label{eq: continuous-domain motivation}
\end{equation}
The right-hand side in~\eqref{eq: continuous-domain motivation}, up to a constant, is the weighted average of the difference between the densities $f_1(x)$ and $f_0(x)$ with respect to the ``weight'' function $\omega(x)$. Ideally, it is desirable to determine that the right-hand side in~\eqref{eq: continuous-domain motivation} is positive when $\omega(x)$ is either a localized function around some point $y\in\mathbb{R}^D$, e.g., the Gaussian kernel $\operatorname{exp}\{-\Vert y - x \Vert_2^2/\sigma \}$, or if $\omega(x)$ is the indicator function over some subset of $\mathbb{R}^D$. In the context of Problem~\ref{prob:problem def}, the quantity $\sum_{i=1}^n w_i s(x_i)$ can be thought of as a sample-based approximation to the right-hand side in~\eqref{eq: continuous-domain motivation} for a particular choice of $\omega(x)$ (up to the constant factor $n p (1-p)$). However, in this work we do not simply choose a predefined class of functions for $\omega(x)$ (and then take $\mathcal{F}$ by sampling these functions at $x_1,\ldots,x_n$), but allow $\mathcal{F}$ to adapt to the given samples (as in Example~\ref{Example:forming F using clustering}).

Since $w_i \geq 0$, and since $s(x_i) > 0$ if and only if $f_1(x_i) > f_0(x_i)$, by finding $\mathbf{w}$ that solves Problem~\ref{prob:problem def} we guarantee that $f_1(x_i) > f_0(x_i)$ for at least one index $i$ for which $w_i>0$. That is, if $\Omega(\mathbf{w}) = \{x_i: \; i\in \{1,\ldots,n\}, \; w_i>0\}$, then any $\mathbf{w}$ solving Problem~\ref{prob:problem def} implies that $f_1(x)>f_0(x)$ somewhere in $\Omega(\mathbf{w})$. Consequently, solving Problem~\ref{prob:problem def} is particularly advantageous when the distributions in $\mathcal{F}$ are sparse, as it allows us to effectively localize the phenomenon $f_1(x) > f_0(x)$ to a restricted subset of $\{x_1,\ldots,x_n\}$. More generally, finding $\mathbf{w}$ that solves Problem~\ref{prob:problem def} admits a useful interpretation even if $\mathbf{w}$ is not sparse, but sufficiently localized in a certain subset of $\{x_1,\ldots,x_n\}$. We further discuss the motivation behind Problem~\ref{prob:problem def} and compare it to alternative formulations in Remark~\ref{remark:problem formulation} below. In addition, we point out that we do not expect a method to detect \textit{all} $\mathbf{w}\in \mathcal{F}$ for which $\sum_{i=1}^n w_i s(x_i) > 0$. Specifically, we allow a method to detect a distribution $\hat{\mathbf{w}}$ that is very close (in an appropriate metric) to $\mathbf{w}$ that satisfies $\sum_{i=1}^n w_i s(x_i) > 0$, even if $\mathbf{w}$ itself is not detected. This concept is incorporated in our hypothesis testing framework, and in particular, in the local risk function that we define in the next section. 

\begin{remark} \label{remark:problem formulation}
Let $\mathcal{H}$ be a collection of subsets of $\{x_1,\ldots,x_n\}$. One task that can come to mind is to determine for which subset $\Omega \in \mathcal{H}$ we have $f_1(x)>f_0(x)$ for all $x\in\Omega$. However, such a task does not lend itself to non-parametric finite-sample probabilistic guarantees, since as explained in the introduction one cannot determine with high significance that $f_1(x_i)>f_0(x_i)$ for any single $x_i$, let alone for all $x\in\Omega$ (without explicit assumptions on the densities $f_1$ and $f_0$). Instead, one could consider the relaxed task of determining for which $\Omega \in \mathcal{H}$ we have $f_1(x)>f_0(x)$ for at least one $x\in\Omega$. However, this task is not sufficiently informative, since if one finds $\Omega$ that solves this problem, than any other set $\Omega^{'}$ that includes $\Omega$ is immediately also a solution, even if $f_1(x)>f_0(x)$ only in a tiny portion of $\Omega^{'}$. In contrast, Problem~\ref{prob:problem def} does not suffer from such redundancy, since even in the simple case that $\mathbf{w}$ and $\mathbf{w}^{'}$ are uniform distributions over $\Omega$ and $\Omega^{'}$, respectively, finding that $\mathbf{w}$ solves Problem~\ref{prob:problem def} does not imply that $\mathbf{w}^{'}$ also solves it. Whether or not $\mathbf{w}^{'}$ solves Problem~\ref{prob:problem def} depends on how the values of $s(x)$ behave on average for all $x\in \Omega^{'}$. Hence, determining that ${\mathbf{w}}^{'}$ solves Problem~\ref{prob:problem def} is a stronger finding than simply $f_1(x)>f_0(x)$ for some $x\in\Omega^{'}$.
\end{remark}

\subsubsection{Hypothesis testing framework} \label{sec:hypothesis testing framework}
We next describe a hypothesis testing framework that formalizes Problem~\ref{prob:problem def}, and provide appropriate risk functions to assess the performance of a corresponding testing procedure, both in the global sense of detecting that there exists some $\mathbf{w}\in\mathcal{F}$ that solves Problem~\ref{prob:problem def}, and also in an appropriate local sense of determining which one it is.

Let $\mathcal{G} \subseteq \mathcal{F}$, and denote $\mathbf{s} = [s(x_1),\ldots,s(x_n)]$. We define $H_0(\mathcal{G})$ as a parametrized null hypothesis according to
\begin{equation}
    H_0(\mathcal{G}): \;\; \langle \mathbf{w},\mathbf{s}\rangle \leq 0, \quad \forall \mathbf{w}\in\mathcal{G}, \label{eq:H_0 def}
\end{equation}
where $\langle \cdot, \cdot \rangle$ is the standard scalar product. We denote $H_0 = H_0(\mathcal{F})$, and refer to $H_0$ simply as the null hypothesis. Note that $H_0$ also includes the typical null hypothesis used in two-sample testing, where $s(x)=0$ for all $x\in\mathcal{X}$, i.e., $f_1$ and $f_0$ are identical.
Next, for $\mathbf{w}\in\mathcal{F}$ and $\gamma \geq 0$ we define $H_1(\mathbf{w},\gamma)$ as a specific alternative to $H_0(\mathbf{w})$ via
\begin{equation}
    H_1(\mathbf{w},\gamma): \;\; \langle \mathbf{w},\mathbf{s}\rangle > \gamma \Vert \mathbf{w} \Vert_2. \label{eq:H_1 specific def}
\end{equation}
Note that $H_1(\mathbf{w},0)$ implies that $\langle \mathbf{w},\mathbf{s}\rangle > 0$ and is therefore the alternative to $H_0(\mathbf{w})$.
The parameter $\gamma$ in~\eqref{eq:H_1 specific def} factors into account both the magnitude of the local deviation between $f_1$ and $f_0$ and the effective size of $\mathbf{w}$, as defined by $\Vert \mathbf{w}\Vert_2$ (which is smaller for more spread-out distributions); see the following remark.
\begin{remark} \label{Remark:example for gamma}
Suppose that $\mathbf{w}$ is the uniform distribution over a subset of nodes $\Omega\subset \{x_1,\ldots,x_n\}$ with $ \mid \Omega \mid  = m$, i.e.,
\begin{equation}
    w_i = 
    \begin{dcases}
        1/m, & x_i \in \Omega, \\
        0, & \text{otherwise}.
    \end{dcases}
\end{equation}
In this case we have $\Vert \mathbf{w}\Vert_2 = 1/\sqrt{m}$, and $H_1(\mathbf{w},\gamma)$ implies that $\sqrt{m} \left(\frac{1}{m}\sum_{x_i\in\Omega} s(x_i)\right)  > \gamma$. Therefore, under $H_1(\mathbf{w},\gamma)$, the quantity $\gamma^2$ can be interpreted as a lower bound on the number of samples in the region $\Omega$, multiplied by the squared average value of $s(x)$ in that region. Hence, $\gamma$ reflects both the size of the region corresponding to $\mathbf{w}$, and the discrepancy between $f_1$ and $f_0$ in that region.
\end{remark}
The reason that we include the quantity $\Vert \mathbf{w} \Vert_2$ in our alternative $H_1(\mathbf{w},\gamma)$ is that $\langle \mathbf{w},\mathbf{s}\rangle$ alone may be insufficient to describe the power of a test to detect that $\langle \mathbf{w},\mathbf{s}\rangle > 0$. For example, if we consider the setting in Remark~\ref{Remark:example for gamma}, knowing the quantity $\frac{1}{m}\sum_{x_i\in\Omega} s(x_i)$ is not enough to characterize the performance of a test (to detect that $\frac{1}{m}\sum_{x_i\in\Omega} s(x_i) > 0$ using the labels $z_1,\ldots,z_n$), as another crucial quantity is $\vert \Omega\vert = m$, which is the number of labels that are actually relevant for this task (acting as an effective sample size). Clearly, it may be impossible to detect that $\frac{1}{m}\sum_{x_i\in\Omega} s(x_i) > 0$ with high probability from the labels $\{z_i: \; x_i\in\Omega \}$ if $\vert \Omega\vert = m$ is very small (e.g., $m=1$), even if $\frac{1}{m}\sum_{x_i\in\Omega} s(x_i)$ is large.

Let $H_1$ be the alternative to $H_0$, i.e., $ H_1 = \cup_{\mathbf{w} \in \mathcal{F}} H_1(\mathbf{w},0)$, and consider a test $Q_{\mathbf{z}} \in \{0,1\}$ whose input is the vector of labels $\mathbf{z}=[z_1,\ldots,z_n]$, and output is $1$ for $H_1$ and $0$ for $H_0$. We define the \textit{global} risk of the test $Q_{\mathbf{z}}$, for a given $\gamma \geq 0$, as the sum of the worst-case type I and worst-case type II errors, that is
\begin{equation}
R^{(n)}_{\mathcal{F}}(Q_{\mathbf{z}},\gamma) = \sup_{f_1,f_0\in H_0}\operatorname{Pr}\{Q_{\mathbf{z}}=1  \mid  f_0, f_1 \} + \sup_{\mathbf{w}\in\mathcal{F}} \sup_{f_1,f_0\in H_1(\mathbf{w},\gamma)} \operatorname{Pr}\{Q_{\mathbf{z}}=0  \mid  f_0,f_1  \}. \label{eq:hypothesis testing global risk def}
\end{equation}
According to~\eqref{eq:hypothesis testing global risk def} and our definition of the null $H_0$, it is clear that the global risk $R^{(n)}_{\mathcal{F}}(Q_{\mathbf{z}},\gamma)$ penalizes any test that incorrectly determines the sign of $\langle \mathbf{w}, \mathbf{s}\rangle$, which is a crucial property if our goal is to decide whether $f_1(x) > f_0(x)$ locally. However, the global risk only quantifies our ability to determine whether some alternative is true, and not which one it is.

According to Problem~\ref{prob:problem def}, our testing methodology is not expected to return just a binary value (for $H_0$ or $H_1$), hence we consider its output to be a family of distributions $\hat{\mathcal{G}}_{\mathbf{z}} \subseteq \mathcal{F}$, where each $\mathbf{w} \in \hat{\mathcal{G}}_{\mathbf{z}}$ represents an accepted alternative $H_1(\mathbf{w},0)$. From this point onward, we will refer to $\hat{\mathcal{G}}_{\mathbf{z}}$ as a \textit{local} test.
The binary test $Q_{\mathbf{z}}$ can then be defined directly via the local test $\hat{\mathcal{G}}_{\mathbf{z}}$ by taking $Q_{\mathbf{z}}=0$ if $\hat{\mathcal{G}}_{\mathbf{z}}$ is an empty set, and $Q_{\mathbf{z}}=1$ otherwise. 
Given a local test $\hat{\mathcal{G}}_{\mathbf{z}}$ and a parameter $\gamma$, we introduce the \textit{local} risk $r^{(n)}_{\mathcal{F}}(\hat{\mathcal{G}}_{\mathbf{z}},\gamma)$ as
\begin{equation}
r^{(n)}_{\mathcal{F}}(\hat{\mathcal{G}}_{\mathbf{z}},\gamma) = \sup_{\mathcal{G} \subseteq \mathcal{F}}\sup_{f_1,f_0\in H_0(\mathcal{G})}\operatorname{Pr} \{ \hat{\mathcal{G}}_{\mathbf{z}} \cap \mathcal{G} \neq \emptyset   \mid  f_0, f_1 \} + \sup_{\mathbf{w}\in\mathcal{F}} \sup_{f_1,f_0\in H_1(\mathbf{w},\gamma)}  \mathbb{E}[ \inf_{\hat{\mathbf{w}} \in \hat{\mathcal{G}}_{\mathbf{z}}} \mathcal{E}_{\operatorname{TV}} (\hat{\mathbf{w}}, \mathbf{w})   \mid  f_0,f_1 ], \label{eq:hypothesis testing local risk def}
\end{equation}
where $\mathcal{E}_{\operatorname{TV}} (\hat{\mathbf{w}}, \mathbf{w}) = \frac{1}{2}\sum_{i=1}^n \vert \hat{w}_i - w_i \vert$ is the total variation distance between $\hat{\mathbf{w}}$ and $\mathbf{w}$, and we define 
$\inf_{\hat{\mathbf{w}} \in \hat{\mathcal{G}}_{\mathbf{z}}} \mathcal{E}_{\operatorname{TV}} (\hat{\mathbf{w}}, \mathbf{w}) = 1$ if $\hat{\mathcal{G}}_{\mathbf{z}}$ is an empty set. 
In plain words, the first summand in~\eqref{eq:hypothesis testing local risk def} is the worst-case probability to accept a false alternative, and the second summand in~\eqref{eq:hypothesis testing local risk def} is the worst-case expected error (in total variation distance) between $\mathbf{w}$ from a true alternative and its closest element in the output of the local test $\hat{\mathcal{G}}_{\mathbf{z}}$. 
The motivation behind~\eqref{eq:hypothesis testing local risk def} is as follows. By making the first term in~\eqref{eq:hypothesis testing local risk def} small, we guarantee that the local test is likely to output only distributions from true alternatives. By making the second term in~\eqref{eq:hypothesis testing local risk def} small, we guarantee that a distribution $\mathbf{w}$ from a true alternative $H_1(\mathbf{w},\gamma)$ is likely to be approximately detected (by finding a similar distribution $\hat{\mathbf{w}}\in \mathcal{F}$ in total variation distance). 

Observe that according to~\eqref{eq:H_1 specific def}, if $\gamma \leq \gamma^{'}$ then $\{f_0,f_1 \in H_1(\mathbf{w},\gamma^{'})\} \subseteq \{f_0,f_1 \in H_1(\mathbf{w},\gamma)\}$ for each $\mathbf{w}\in\mathcal{F}$. Consequently, the risks $R^{(n)}_{\mathcal{F}}(Q_{\mathbf{z}},\gamma)$ and $r^{(n)}_{\mathcal{F}}(\hat{\mathcal{G}}_{\mathbf{z}},\gamma)$ decrease as $\gamma$ increases. 
In the case that $\gamma$ is large enough so that the set $\{f_1,f_0\in \cup_{\mathbf{w}\in \mathcal{F} } H_1(\mathbf{w},\gamma)\}$ is empty, the second summands in both~\eqref{eq:hypothesis testing global risk def} and~\eqref{eq:hypothesis testing local risk def} are set to be zero, in which case taking $\hat{\mathcal{G}}_{\mathbf{z}} = \emptyset$ and $Q_\mathbf{z} = 0$ gives $R^{(n)}_{\mathcal{F}}(Q_{\mathbf{z}},\gamma) = r^{(n)}_{\mathcal{F}}(\hat{\mathcal{G}}_{\mathbf{z}},\gamma) = 0$.  
Therefore, the parameter $\gamma$ can be viewed as a problem-hardness parameter, where the larger $\gamma$ is, the easier the hypothesis testing problem is, both in terms of the global risk $R^{(n)}_{\mathcal{F}}(Q_{\mathbf{z}},\gamma)$ and of the local risk $r^{(n)}_{\mathcal{F}}(\hat{\mathcal{G}}_{\mathbf{z}},\gamma)$. 
In addition, we have the following relation between the global and local risks.
\begin{lem} \label{lem:relation between local and global risks}
If $Q_\mathbf{z}$ is the test associated with $\hat{\mathcal{G}}_{\mathbf{z}}$ (i.e., $Q=0$ if $\hat{\mathcal{G}}_{\mathbf{z}}$ is empty and $Q_\mathbf{z}=1$ otherwise), then for any family of distributions $\mathcal{F}$ and $\gamma \geq 0$
\begin{equation}
    R^{(n)}_{\mathcal{F}}(Q_{\mathbf{z}},\gamma) \leq r^{(n)}_{\mathcal{F}}(\hat{\mathcal{G}}_{\mathbf{z}},\gamma).
\end{equation}
\end{lem}
The proof can be found Appendix~\ref{appendix:relation between local and global risks}. 
Next, we define the notions of global and local consistencies.
\begin{defn} [Local and global consistency]
A local test $\hat{\mathcal{G}}_{\mathbf{z}}$ is said to be \textit{locally consistent} with respect to a sequence $\{\gamma_n\}$ if $\lim_{n\rightarrow \infty} r^{(n)}_{\mathcal{F}}(\hat{\mathcal{G}}_{\mathbf{z}},\gamma_n) = 0$, and is said to be \textit{globally consistent} if $\lim_{n\rightarrow \infty }R^{(n)}_{\mathcal{F}}(Q_{\mathbf{z}},\gamma_n) = 0$, where $Q_\mathbf{z}=0$ if $\hat{\mathcal{G}}_{\mathbf{z}}$ is empty and $Q_\mathbf{z}=1$ otherwise.
\end{defn}
Since the global and local risks are both nonnegative, and according to Lemma~\ref{lem:relation between local and global risks}, if a local test $\hat{\mathcal{G}}_{\mathbf{z}}$ is locally consistent, then it is also globally consistent. Hence, if $\hat{\mathcal{G}}_{\mathbf{z}}$ is not globally consistent, then it cannot possibly be locally consistent. Given a local test $\hat{\mathcal{G}}_\mathbf{z}$, it is of primary interest to determine for which sequences $\{\gamma_n\}_{n=1}^\infty$ it is globally and locally consistent. In particular, it is desirable to find a local test which is locally consistent for sequences $\{\gamma_n\}$ for which $\gamma_n$ is as small as possible as $n\rightarrow \infty$.

One of strong suits of our hypothesis testing framework is that it does not require any structural assumptions or parametric models on the densities $f_1$ and $f_0$, nor any assumptions on their smoothness in the ambient space $\mathcal{X}$. This is one of key advantages of conditioning on the observations $x_1,\ldots,x_n$, as the only assumption required for rejecting $\langle \mathbf{w}, \mathbf{s} \rangle \leq 0$ (with prescribed significance) is that the sample-label pairs $(x_1,z_1),\ldots,(x_n,z_n)$ are independent (which leads to the independence of the labels conditional on the samples). If $\langle \mathbf{w}, \mathbf{s} \rangle \leq 0$ is rejected for a specific $\mathbf{w}$, the large entries in $\mathbf{w}$ indicate the samples that are most important in the rejection. This observation provides a useful tool for subsequent analysis, particularly if $\mathbf{w}$ is sparse or approximately sparse, but also when $\mathbf{w}$ represents a cluster in the data, since such clusters typically have meaningful interpretations through expert knowledge.  

\subsubsection{Local two-sample testing by random-walk distributions} \label{sec: random-walk distributions}
The hypothesis testing framework detailed in the previous section  allows one to evaluate the performance of a testing procedure with respect to a given family $\mathcal{F}$, but is not concerned with the question of how to choose $\mathcal{F}$. In general, the choice of $\mathcal{F}$ should be guided by expert knowledge and by examining the standard tools and processing pipelines that are in current use for the type of data in question.
In this section and for the rest of this work, we focus on a specific choice of $\mathcal{F}$ that we believe to be particularly well suited for data residing on a low-dimensional manifold with interpretable latent parameters. Specifically, we propose and analyze a multi-resolution approach where $\mathcal{F}$ is a family of random-walk distributions over a given weighted graph.

Suppose that $G$ is a weighted graph whose vertices are $\{x_1,\ldots,x_n\}$, and whose edges (and weights) are given by a nonnegative adjacency matrix $W\in\mathbb{R}^{n\times n}$ satisfying the following assumption.
\begin{assump} \label{assump:W properties}
$W$ is symmetric, stochastic, positive semidefinite (PSD), and deterministic given $x_1,\ldots,x_n$.
\end{assump}
For instance, if $x_1,\ldots,x_n$ reside in Euclidean space, then a matrix $W$ satisfying Assumption~\ref{assump:W properties} can be obtained by diagonally scaling~\cite{sinkhorn1964relationship,knight2008sinkhorn} the Gaussian kernel $\operatorname{exp}\{\Vert x_i - x_j \Vert_2^2/\sigma\}$ to have row and column sums of $1$, i.e., $W_{i,j} = d_i \operatorname{exp}\{\Vert x_i - x_j \Vert_2^2/\sigma\} d_j$ for an appropriate choice of $d_1,\ldots,d_n$~\cite{landa2020doubly,marshall2019manifold}, where $\sigma$ is a tunable ``bandwidth'' parameter. It is worthwhile to point out that this choice of $W$ admits a useful interpretation in the context of manifold learning. Specifically, if $x_1,\ldots,x_n$ are sampled uniformly from a smooth low-dimensional Riemannian manifold embedded in Euclidean space, then the matrix $I-W$ provides a sample-based approximation to the Laplace-Beltrami operator on the manifold, and $W^t$ (i.e., the $t$'th matrix power of $W$) provides an approximation to the solution of the heat equation on the manifold at time $\sigma t$~\cite{marshall2019manifold,coifman2006diffusion,wormell2020spectral}. These solutions describe the propagation of heat on the manifold as time progresses, with the initial condition being a point source. Consequently, for a given point on the manifold, the propagation of heat from that point characterizes local regions on the manifold that grow as time progresses, where the notion of locality is adapted to the geometry of the manifold; see Section~\ref{sec:toy example circle} for a simple toy example that illustrates this concept.

Our approach in this work is not limited to a specific choice of $W$, and in Section~\ref{sec:W construction} we propose a procedure for constructing $W$ satisfying Assumption~\ref{assump:W properties} from an arbitrary nonnegative matrix $K$ provided by the user (and is deterministic given $x_1,\ldots,x_n$). Intuitively, $K$ should be an affinity matrix encoding the similarities between the points $x_1,\ldots,x_n$, and preferably should be exactly or approximately  sparse. If $K$ is not readily available, it can be formed by standard approaches such as nearest-neighbour graphs and nonnegative kernels via a metric over $\mathcal{X}$ (see for example~\cite{maier2009influence,maier2009optimal,coifman2006diffusion,berry2016consistent,landa2020doubly} and references therein), which is typically provided by a domain expert. 

According to Assumption~\ref{assump:W properties}, $W$ is stochastic, i.e., $W_{i,j}\geq 0$ for all $i,j$, and $\sum_{j=1}^n W_{i,j}= 1$ for all $i$, hence $W$ can serve as a transition probability matrix of a Markov chain over $\{x_1,\ldots,x_n\}$, where $W_{i,j}$ is the probability to transition from $x_i$ to $x_j$ at each step. Consequently, for a random walk that started at $x_i$, the quantity $W^t_{i,j}$, which is the $(i,j)$'th entry of the $t$'th matrix power of $W$, is the probability to be at $x_j$ after $t$ steps. 
We then take into $\mathcal{F}$ all distributions associated with the random walk, i.e., when starting at all possible vertices $x_1,\ldots,x_n$, and for all possible time steps $t=1,2,\ldots,\infty$.
Specifically,
\begin{equation}
    \mathcal{F} = \{ W_i^t: \; 1\leq i \leq n, \; t=1,2,\ldots, \infty \}, \label{eq:F def}
\end{equation}
where $W_i^t \in\mathbb{R}^n$ stands for the $i$'th row of $W^t$. Clearly, by Assumption~\ref{assump:W properties} it follows that $\mathcal{F}$ satisfies Assumption~\ref{assump:F independent of labels}.

Since $W_i^t$ is the distribution of the location of a random-walker that started at $x_i$ after $t$ steps, it describes neighborhoods around $x_i$ at multiple scales. 
More precisely, since $W$ is symmetric and stochastic, it is doubly stochastic, and it follows that the entropy of $W_i^t$ (given by $-\sum_{j=1}^n W_{i,j}^{t} \log W_{i,j}^{t}$) is monotonically increasing in $t$ (see e.g.,~\cite{marshall1979inequalities}), meaning that $W_i^t$ becomes more spread-out as $t$ increases. Moreover, since $W$ is PSD, the associated Markov chain is aperiodic, and $W_i^t$ converges to a stationary distribution as $t\rightarrow\infty$, which is the uniform distribution over the connected component of $G$ that contains $x_i$ (see Proposition~\ref{prop:W spectral properties} in Section~\ref{sec:preliminaries} for more details). Hence, the random walk distribution $W_i^t$ admits a natural notion of location and scale, where $i$ is the location parameter and $t$ is the scale. Figure~\ref{fig:random walk behavior swiss roll} illustrates a prototypical behavior of the random walk distributions $W_i^t$ on a dataset resembling the shape of a Swiss roll. Notably, the propagation of the random walk is restricted to the graph structure -- which captures the non-Euclidean geometry of the dataset.

\begin{figure} 
\begin{subfigure}{0.23\textwidth}
\includegraphics[width=\linewidth]{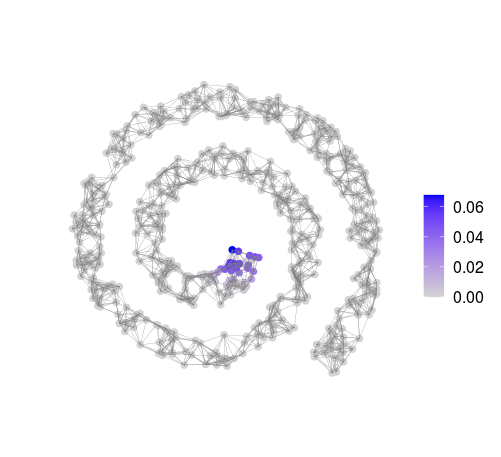}
\caption{$W_a^1$} \label{fig:a}
\end{subfigure}
\begin{subfigure}{0.23\textwidth}
\includegraphics[width=\linewidth]{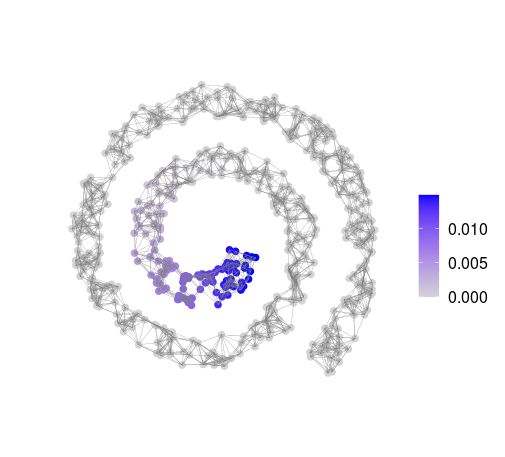}
\caption{$W_a^{30}$} \label{fig:a}
\end{subfigure}
\begin{subfigure}{0.23\textwidth}
\includegraphics[width=\linewidth]{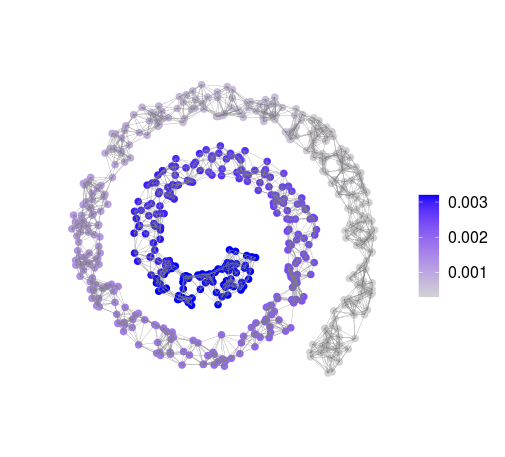}
\caption{$W_a^{1,000}$} \label{fig:b}
\end{subfigure}
\begin{subfigure}{0.23\textwidth}
\includegraphics[width=\linewidth]{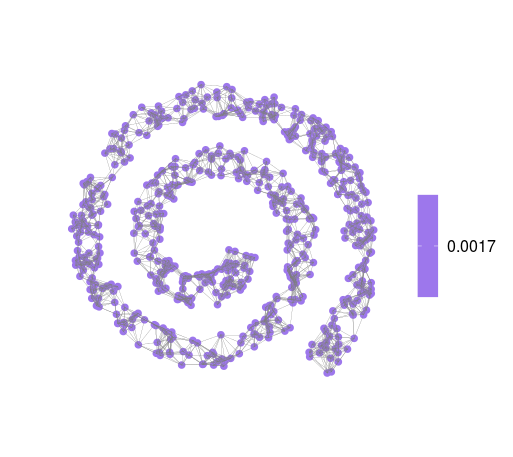}
\caption{$W_a^{100,000}$} \label{fig:b}
\end{subfigure}

\medskip
\begin{subfigure}{0.23\textwidth}
\includegraphics[width=\linewidth]{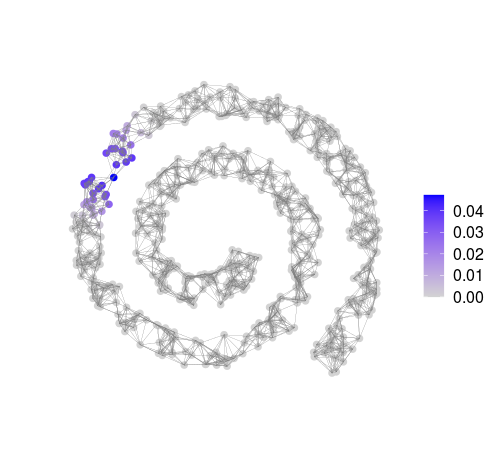}
\caption{$W_b^1$} \label{fig:a}
\end{subfigure}
\begin{subfigure}{0.23\textwidth}
\includegraphics[width=\linewidth]{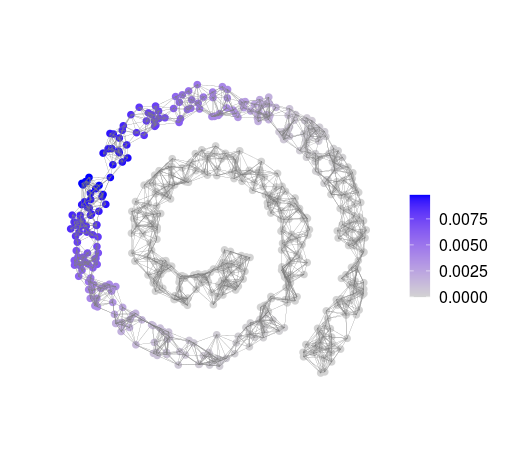}
\caption{$W_b^{30}$} \label{fig:a}
\end{subfigure}
\begin{subfigure}{0.23\textwidth}
\includegraphics[width=\linewidth]{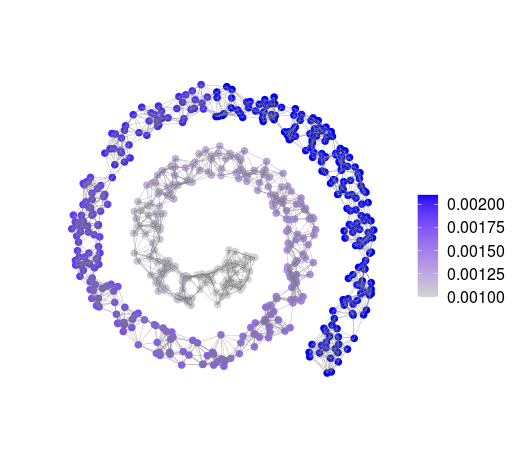}
\caption{$W_b^{1,000}$} \label{fig:b}
\end{subfigure}
\begin{subfigure}{0.23\textwidth}
\includegraphics[width=\linewidth]{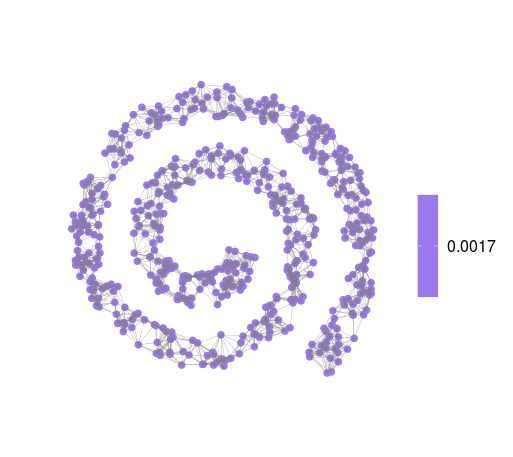}
\caption{$W_b^{100,000}$} \label{fig:b}
\end{subfigure}
\caption{Random-walk distributions $W_i^t$ with two distinct locations: $i = a$ (top row) and $i = b$ (bottom row), and scales $t=1,30,1000,100000$ (from left to right). The matrix $W$ was formed according to Section~\ref{sec:W construction} from a symmetric $10$-NN graph (where $x_i$ and $x_j$ are connected if at least one of them belongs to the $10$ nearest neighbours of the other). For $t=1$, $W_i^t$ encodes the immediate surroundings of $x_i$, and as $t$ grows, $W_i^t$ spreads further and further away from $x_i$ until it converges to the uniform distribution over the connected component that contains $x_i$.} \label{fig:random walk behavior swiss roll}
\end{figure}

The role of the scale parameter $t$ in our setting is analogous to that of the window length in one-dimensional spatial scan statistics (see~\cite{walther2020calibrating} and references therein), where a larger window length (i.e., larger scale) allows for better detection power given that everything else remains equal. Since one does not known apriori what is the smallest scale $t$ at which a difference between $f_1$ and $f_0$ might be detected (with prescribed significance), it is desirable to consider all possible scales, hence our proposed choice of $\mathcal{F}$. 
We mention that the requirements in Assumption~\ref{assump:W properties} that $W$ is symmetric and PSD ensure that the random-walk is sufficiently well-behaved for our purposes; see Proposition~\ref{prop:random walk distribution properties} in Section~\ref{sec:preliminaries}. All properties of $W$ appearing in Assumption~\ref{assump:W properties} are required in the derivation and analysis of our testing procedure. 

\subsection{Method and main results} \label{sec:main results}
\subsubsection{Testing procedure}
We now describe the main ideas and ingredients in our testing procedure, whose derivation can be found in Section~\ref{sec:testing methodology}. In Section~\ref{sec:random walk scan statistics} we show that for any given $\mathbf{w}\in\mathcal{F}$, one can reject $H_0(\mathbf{w})$ in favor of $H_1(\mathbf{w},0)$ at significance level $\alpha$ if the statistic
\begin{equation}
    \mathcal{S}(\mathbf{w}) = \frac{\langle \mathbf{w}, \mathbf{z} - p\rangle}{\Vert \mathbf{w} \Vert_2}, \label{eq:statistic def main results}
\end{equation}
exceeds the threshold $\sqrt{0.5 \log (1/\alpha)}$, where $\mathbf{z}=[z_1,\ldots,z_n]$. The numerator of~\eqref{eq:statistic def main results} is simply an unbiased estimator of $\langle \mathbf{w}, \mathbf{s}\rangle$, and the purpose of the denominator is to normalize the variance of the statistic according to the effective size of $\mathbf{w}$; see Section~\ref{sec:random walk scan statistics}.
If $\mathbf{w}$ is a uniform distribution over a subset of nodes, then $\mathcal{S}(\mathbf{w})$ is also known as the \textit{positive elevated-mean statistic}~\cite{qian2014connected,cadena2017near}, which often arises as a generalized likelihood ratio in certain simple parametric models.
The prior $p$ is assumed to be known throughout the derivation and analysis of our method in Section~\ref{sec:testing methodology}, and we describe how to adapt our approach to unknown $p$ in Section~\ref{sec:adapting to unknown p}.

Since the random-walk distribution $\mathbf{w}$ for which $\mathcal{S}(\mathbf{w})$ is most likely to reject the null is unknown in advance, and since computing $\mathcal{S}(\mathbf{w})$ for all $\mathbf{w}\in\mathcal{F}$ is infeasible, we first define a finite set of distributions $\widetilde{\mathcal{F}}\subset\mathcal{F}$ which represents $\mathcal{F}$ in a suitable way, and then compute $\mathcal{S}(\mathbf{w})$ only for $\mathbf{w}\in\widetilde{\mathcal{F}}$. A naive approach to choose $\widetilde{\mathcal{F}}$ is to exploit the fact that the random walk distributions $W_i^t$ converge to stationary distributions as $t\rightarrow \infty$. Namely, to scan over all integers $t$ up to a sufficiently large time step that is related to the \textit{mixing time}~\cite{levin2017markov} of the Markov chain. However, such an approach is unsatisfactory on its own, since the convergence to the stationary distributions can be arbitrarily slow, depending on $W$ (and specifically on the largest eigenvalue of $W$ that is smaller than $1$). Instead, in Section~\ref{sec:random walk scan statistics} we propose to form the set of distributions $\widetilde{\mathcal{F}}\subset\mathcal{F}$ via a sequence of judiciously-chosen time steps $1 = t_1^{(i)} < t_2^{(i)} < \ldots < t_{M_i}^{(i)}$ for each point $x_i$, and to take into $\widetilde{\mathcal{F}}$ only the random-walk distributions $\{W_i^{t_j^{(i)}}\}_{i,j}$. In particular, we find these time steps by ensuring that the statistics $\{\mathcal{S}(\mathbf{w})\}_{\mathbf{w}\in\widetilde{\mathcal{F}}}$ form an $\epsilon$-net over $\{\mathcal{S}(\mathbf{w})\}_{\mathbf{w}\in{\mathcal{F}}}$ (i.e., each element in the latter set is approximated by some element in the former set to a prescribed accuracy) for arbitrary labels $z_1,\ldots,z_n$, and for a prescribed accuracy parameter $\varepsilon$ (which we show how to tune automatically in Section~\ref{sec:analysis and consistency}). Furthermore, we show that $\widetilde{\mathcal{F}}$ forms an $\epsilon$-net over $\mathcal{F}$ in total variation distance with accuracy $\varepsilon/2$. Therefore, $\widetilde{\mathcal{F}}$ is a favorable surrogate for $\mathcal{F}$ if $\varepsilon$ is sufficiently small. Crucially, we show that this choice of $\widetilde{\mathcal{F}}$ leads to a computationally tractable testing procedure with favorable statistical guarantees, primarily due to the fact that the cardinality of  $\widetilde{\mathcal{F}}$ can be at most $\mathcal{O}(n^2 \log n)$ regardless of the matrix $W$; see the next section for more details.

After evaluating the set of distributions $\widetilde{\mathcal{F}} \subset \mathcal{F}$, our testing procedure is straightforward. We simply test $H_0(\mathbf{w})$ against $H_1(\mathbf{w},0)$ for each $w\in\widetilde{\mathcal{F}}$ using $\mathcal{S}(\mathbf{w})$, while correcting for multiple testing via the Bonferroni procedure. Specifically, for a prescribed significance level $\alpha \in (0,1)$, we reject all $H_0(\mathbf{w})$ with $\mathbf{w}\in\widetilde{\mathcal{F}}$, for which 
\begin{equation}
    \mathcal{S}(\mathbf{w}) > \sqrt{0.5 \log (\sum_{i=1}^n M_i/\alpha)}, \label{eq: local test criterion method summary}
\end{equation}
where $\sum_{i=1}^n M_i =  | \widetilde{\mathcal{F}} | $ is the total number of random-walk time steps chosen for the $\epsilon$-net. 
In the context of our hypothesis testing problem, our local test $\hat{\mathcal{G}}_\mathbf{z}$ includes all distributions $\mathbf{w} \in \widetilde{\mathcal{F}}$ that satisfy~\eqref{eq: local test criterion method summary}, and $Q_\mathbf{z}$ is the corresponding test that is $1$ if $\max_{\mathbf{w}\in\widetilde{\mathcal{F}}} \mathcal{S}(\mathbf{w})$ exceeds the threshold $\sqrt{0.5 \log(\sum_{i=1}^n M_i/\alpha)}$, and $0$ otherwise.
While our testing methodology here is conservative, it avoids the need for costly permutation tests or asymptotic approximations, and most importantly, we show that our particular construction of $\widetilde{\mathcal{F}}$ makes it optimal in a minimax sense in terms of the required sequences $\{\gamma_n\}_{n=1}^\infty$ to achieve local and global consistency (as defined in Section~\ref{sec:approach and problem formulation}).

If we inspect the expression in~\eqref{eq:statistic def main results}, we see that if the quantity $\langle \mathbf{w},\mathbf{z}\rangle$ is kept fixed, then the statistic $\mathcal{S}(\mathbf{w})$ attains larger values for smaller values of $\Vert \mathbf{w} \Vert_2$. By utilizing the properties of $W$ from Assumption~\ref{assump:W properties}, it is not difficult to show that $\Vert W_i^t \Vert_2$ is monotonically decreasing with $t$ (see Proposition~\ref{prop:random walk distribution properties} in Section~\ref{sec:preliminaries}). Therefore, it is generally easier to reject null hypotheses $H_0(W_i^t)$ that are associated with larger scales $t$. This fact emphasizes again the importance of scanning over multiple values of $t$. To further clarify this point, consider $t=1$ and suppose that the affinity matrix $W$ has only a few non-zeros in each row. Then, $\Vert W_i^1 \Vert_2$ may not be small enough for~\eqref{eq: local test criterion method summary} to hold, even if $\langle W_i^1,\mathbf{z}\rangle = 1$ (i.e., all labels in the vicinity of $x_i$ are $1$). On the other hand,~\eqref{eq: local test criterion method summary} can be easily satisfied for sufficiently large $t$ even if $\langle W_i^t,\mathbf{z}\rangle$ is only marginally larger than $p$, as long as the number of samples in the connected component that includes $x_i$ is not too small (since $W_i^t$ converges,  as $t\rightarrow \infty$, to the uniform distribution over the connected component that includes $x_i$, in which case $1/\Vert W_i^t \Vert_2$ converges to the square-root of the number of nodes in that connected component).

In Figure~\ref{fig:Swiss roll test example} we exemplify our testing procedure on the Swiss roll type data from Figure~\ref{fig:random walk behavior swiss roll}, where we used $n=2400$ samples with $p=0.5$; see Figure~\ref{fig:a Swiss roll test}. The densities $f_0$ and $f_1$ are piece-wise constant and chosen such that $f_1>f_0$, $f_1<f_0$, and $f_1=f_0$ on certain prescribed sections along the trajectory of the Swiss roll; see Figure~\ref{fig:b Swiss roll test}. To detect regions where $f_1>f_0$, we applied our testing procedure with significance $\alpha = 0.05$ and obtained the random-walk distribution $W_i^t$ that corresponds to the largest statistic (among $\mathcal{S}(W_i^t)$ for $t=t_1^{(i)},\ldots,t_{M_i}^{(i)}$ and $i=1,\ldots,n$) that passed the threshold~\eqref{eq: local test criterion method summary}; see Figure~\ref{fig:c Swiss roll test}. Similarly, we applied our testing procedure to the negated labels (i.e., when replacing ones with zeros and vice-versa) to detect where $f_1<f_0$, which is equivalent to taking the smallest statistic from the testing procedure when applied to the original labels (since $z_i$ is replaced with $1-z_i$ and $p$ is replaced with $1-p$); see Figure~\ref{fig:d Swiss roll test}. While it is difficult to identify the regions where $f_1>f_0$ and $f_1<f_0$ from Figure~\ref{fig:a Swiss roll test} visually, our method is able to capture these regions through the random walk distributions that passed the threshold~\eqref{eq:statistic def main results}. Evidently, the distributions depicted in Figures~\ref{fig:c Swiss roll test} and~\ref{fig:d Swiss roll test}, which correspond to time steps $t=111$ and $t=172$, respectively, are largely consistent with Figure~\ref{fig:b Swiss roll test}. It is important to mention that due to the curved nature of the data, it is inappropriate to define a testing procedure that ignores the samples $x_1,\ldots,x_n$, e.g., to scan over balls in the ambient space that vary in their center and radii. In particular, such a scan would not be able to exploit all samples in the contiguous regions where $f_1>f_0$ or $f_1<f_0$, and may often mix between the samples in the two regions.

\begin{figure} [h]
\hspace{0.1\textwidth}
\hspace{0.035\textwidth}
\begin{subfigure}{0.3\textwidth} 
\includegraphics[width=\linewidth]{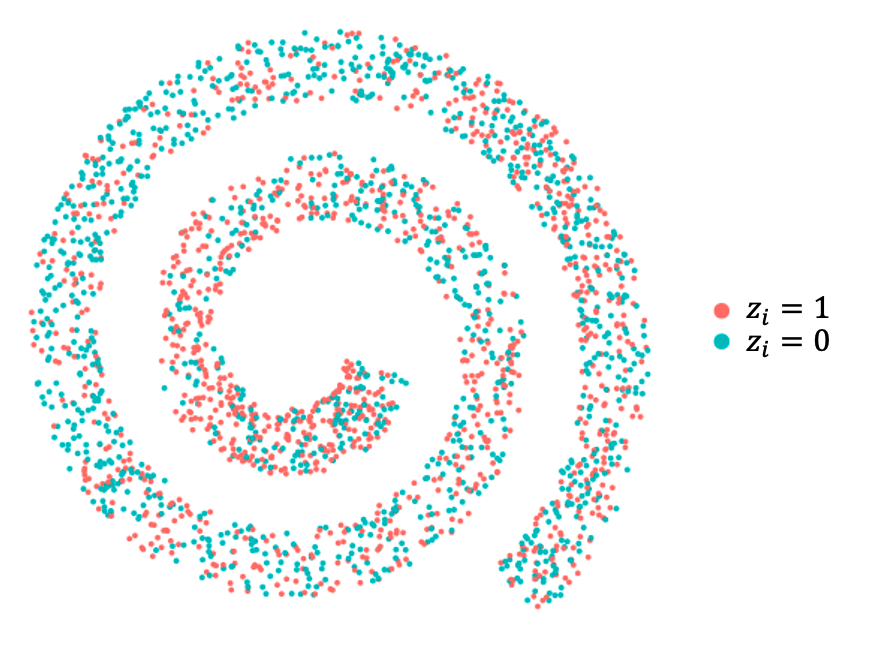}
\caption{$x_1,\ldots,x_n$ and $z_1,\ldots,z_n$} \label{fig:a Swiss roll test}
\end{subfigure}
\hspace{0.035\textwidth}
\begin{subfigure}{0.317\textwidth}
\includegraphics[width=\linewidth]{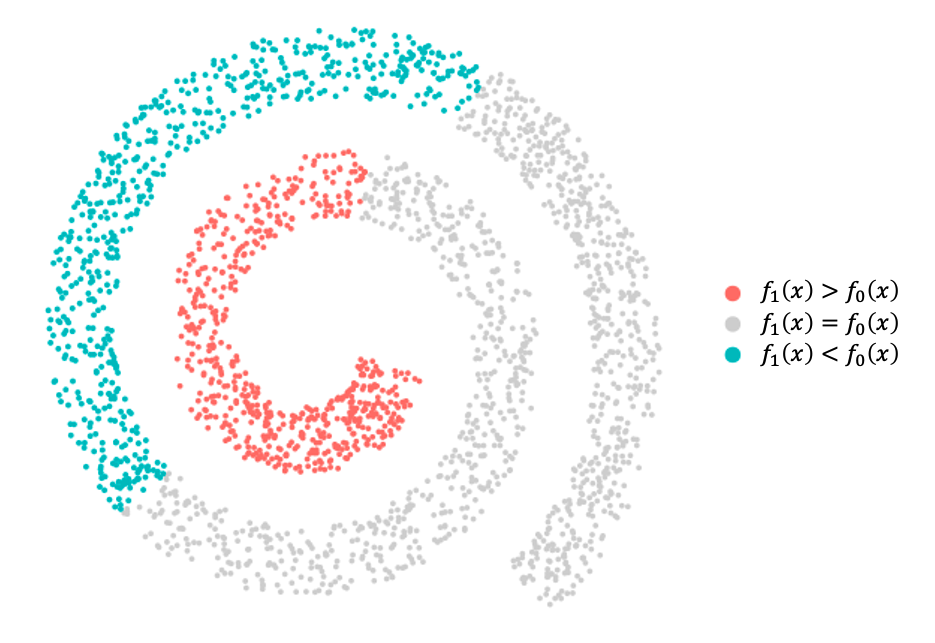}
\caption{Sections where $f_1(x) \lesseqgtr f_0(x)$} \label{fig:b Swiss roll test}
\end{subfigure}

\hspace{0.1\textwidth}
\hspace{0.03\textwidth}
\begin{subfigure}{0.315\textwidth}
\includegraphics[width=\linewidth]{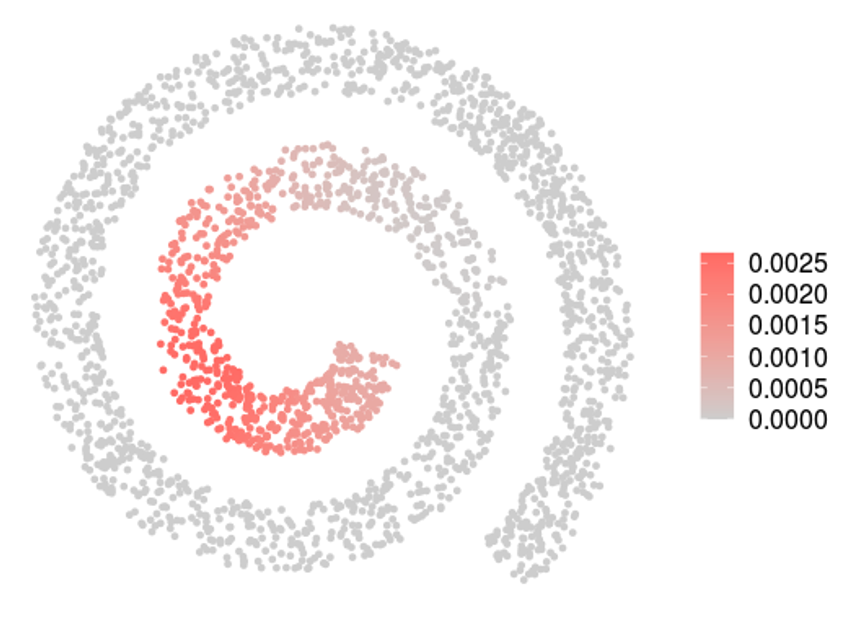}
\caption{$W_{i}^{t}$ with the largest statistic} \label{fig:c Swiss roll test}
\end{subfigure}
\hspace{0.025\textwidth}
\begin{subfigure}{0.305\textwidth}
\includegraphics[width=\linewidth]{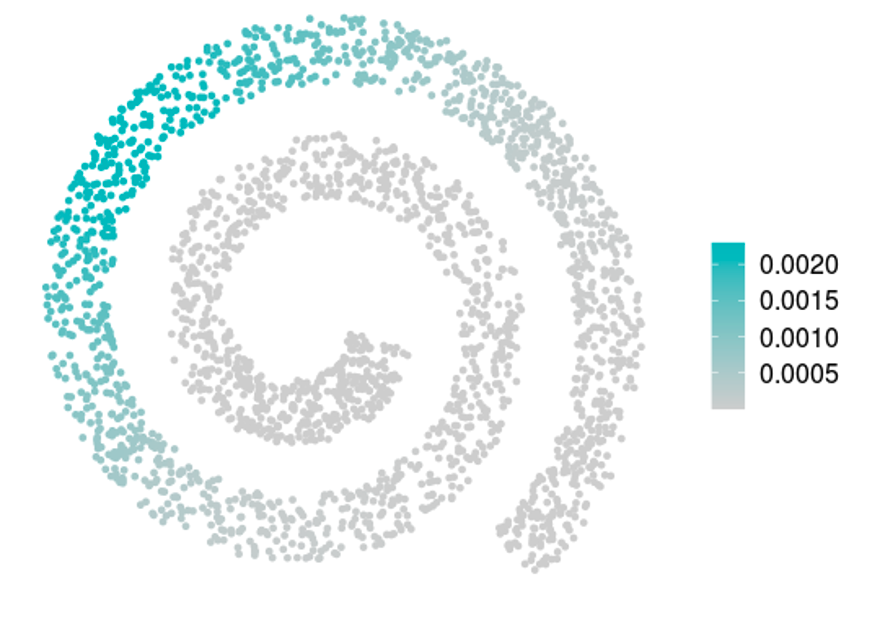}
\caption{$W_{i}^{t}$ with the smallest statistic} \label{fig:d Swiss roll test}
\end{subfigure}
\caption{(a) Simulated Swiss roll data with $p=0.5$ and $n=2400$ samples colored by their labels. (b) Samples $x_i$ for which $f_1(x_i) > f_0(x_i)$ (red), $f_1(x_i) < f_0(x_i)$ (blue), and $f_1(x_i) = f_0(x_i)$ (gray). (c) Random walk distribution $W_{i}^{t}$ corresponding to the largest statistic in the scan. (d) Random walk distribution $W_{i}^{t}$ corresponding to the smallest statistic in the scan.} \label{fig:Swiss roll test example}
\end{figure}

For practical purposes, when~\eqref{eq: local test criterion method summary} holds we not only accept $H_1(\mathbf{w},0)$, but also accept an alternative of the form $H_1(\mathbf{w},\hat{\gamma})$ by specifying $\hat{\gamma} > 0$; see Section~\ref{sec:random walk scan statistics}. Note that any accepted $H_1(\mathbf{w},\hat{\gamma})$ also implies $H_1(\mathbf{w},0)$, hence it is sufficient to provide only the former. The purpose of specifying $\hat{\gamma}$ in an alternative $H_1(\mathbf{w},\hat{\gamma})$ is that it provides a lower bound on $\langle \mathbf{w},\mathbf{s}\rangle$ (see~\eqref{eq:H_1 specific def}), which is a local measure of discrepancy between $f_1$ and $f_0$ describing effect size rather than significance.

Clearly, our approach requires the class prior probability $p$ and a matrix $W$ satisfying assumption~\ref{assump:W properties}. 
In Section~\ref{sec:adapting to unknown p} we use the full probabilistic model in Section~\ref{sec:inroduction} (i.e., without conditioning on $x_1,\ldots,x_n$) and describe how to modify our test to cope with unknown $p$ by estimating a confidence interval for binomial proportion (where the binomial variable is $\sum_{i=1}^n z_i$). 
In Section~\ref{sec:W construction} we describe how to construct $W$ satisfying Assumption~\ref{assump:W properties} from an arbitrary nonnegative matrix $K$ (which is deterministic given $x_1,\ldots,x_n$). Specifically, we use the geometric-mean for symmetrization and apply diagonal scaling for making the matrix doubly stochastic. We show that this procedure finds the nonnegative, symmetric, and stochastic matrix closest to $K$ in KL-divergence. After symmetrization and diagonal scaling, we square the resulting matrix to make it PSD, which is a non-restrictive step equivalent to taking only the even time steps in the previously-resulting random walk. 

For the reader's convenience, we summarize our testing procedure in Algorithm~\ref{alg:local two-sample testing by aRWSS}. We Analyze the computational complexity of Algorithm~\ref{alg:local two-sample testing by aRWSS} in Appendix~\ref{appendix:implementation details}.

\begin{algorithm}
\caption{Local two-sample testing by random-walk distributions}\label{alg:local two-sample testing by aRWSS}
\begin{algorithmic}[1]
\Statex{\textbf{Input:}} Binary labels $\mathbf{z} = [z_1,\ldots,z_n]$, nonnegative affinity matrix $K\in\mathbb{R}^{n \times n}$ (which is deterministic given $x_1,\ldots,x_n$), and significance level $\alpha \in (0,1)$. 
\State Construct $W$ using~\eqref{eq:doubly stochastic normalization} and~\eqref{eq:PSD normalization} in Section~\ref{sec:W construction}. \label{step:constructing W}
\State Find the connected components $\mathcal{C}_1,\ldots,\mathcal{C}_L \subset \{1,\ldots,n\}$ of the graph described by $W$. \label{step:connected components}
\item Evaluate the eigen-decomposition of each principal submatrix $W^{(\ell)} = [W_{i,j}]_{i\in \mathcal{C}_\ell,\; j\in \mathcal{C}_\ell}$, $\ell=1,\ldots,L$. \label{step:eigen-decomposition of W}
\State Choose $\varepsilon \in (0,1)$ by minimizing $h(\varepsilon)$ from~\eqref{eq:beta def}, or according to Table~\ref{table:values of vareps} in Appendix~\ref{appendix:tables}. \label{step:minimzing beta}
\State Compute the time steps $t_1^{(i)},\ldots,t_{M_i}^{(i)}$ for each $i=1,\ldots,n$ using Algorithm~\ref{alg:evaluating diffusion time steps}. \label{step:compute time steps}
\State If the class prior $p\in (0,1)$ is unknown, set it as the upper bound from the Clopper-Pearson method applied to $\sum_{i=1}^n z_i$ with coverage $1-\alpha/2$ (see Section~\ref{sec:adapting to unknown p}), and subsequently replace $\alpha$ with $\alpha/2$. \label{step:Clopper-Pearson}
\ForAll{$i\in\{1,\ldots,n\}$ and $j\in \{1,\ldots,M_i\}$} \label{step:compute alpha_hat}
\State Compute: $\hat{\gamma}_{i,j} = \mathcal{S}(W_{i}^{t_j^{(i)}}) - \sqrt{0.5 \log(\sum_{k=1}^n M_k/\alpha)}$.
\EndFor
\State For each pair $(i,j)$ with $\hat{\gamma}_{i,j} > 0$ reject $H_0(W_{i}^{t_j^{(i)}})$ and accept  $H_1(W_{i}^{t_j^{(i)}},\hat{\gamma}_{i,j})$.
\end{algorithmic}
\end{algorithm}

\subsubsection{Analysis and theoretical guarantees}
In Sections~\ref{sec:analysis and consistency} and~\ref{sec:minimax risk} we analyze our testing procedure assuming the prior $p$ is known. The main results are as follows, which for simplicity are presented here assuming the graph $G$ is connected (while the results in Sections~\ref{sec:analysis and consistency} consider an arbitrary number of connected components in $G$).
In Lemma~\ref{lem:t_1 and M_i bounds} in Section~\ref{sec:analysis and consistency} we show that
\begin{equation}
    M_i \leq \min \left\{ \left\lceil \frac{\log(n/\varepsilon)}{\log((\lambda_{< 1})^{-1})} \right\rceil, \left\lceil \frac{\log(n)}{\log(1+\varepsilon^2/n)} \right\rceil \right\}, \label{eq:bound on M_i main resuls}
\end{equation}
where $M_i$ is the number of chosen time steps $t_1^{(i)},\ldots,t_{M_i}^{(i)}$ (for the $\epsilon$-net) for each index $i$, and $\lambda_{< 1}$ is the largest eigenvalue of $W$ which is strictly smaller than $1$.
Of particular interest is the fact that the quantity $ {\log(n)}/{\log(1+\varepsilon^2/n)}$ in~\eqref{eq:bound on M_i main resuls} is independent of $W$ and its spectrum, meaning that the cardinality of our constructed $\epsilon$-net cannot be too large even when the convergence to a stationary distribution of the random walk is arbitrarily slow. As far as we know, this is a new result on the random walk distributions associated with a PSD transition probability matrix, and may be of independent interest. Furthermore,~\eqref{eq:bound on M_i main resuls} implies that $|  \widetilde{\mathcal{F}}|  = \mathcal{O} (n^2 \log n) $ (see the discussion following Lemma~\ref{lem:t_1 and M_i bounds} in Section~\ref{sec:analysis and consistency}), hence the scan over $\widetilde{\mathcal{F}}$ is always computationally feasible, even in the worst-case scenario where $\lambda_{< 1}$ is arbitrarily close to $1$. 

In Theorem~\ref{thm:test power} and equation~\eqref{eq:beta upper bound} in Section~\ref{sec:analysis and consistency}, we show that the power of $Q_\mathbf{z}$ (our binary test) to detect an alternative $H_1(\mathbf{w},\gamma)$ is at least 
\begin{equation}
    1 - \operatorname{exp}\left[-2(\gamma - \hat{h}_{n,\alpha,\lambda_{<1}}(\varepsilon))^2\right], \label{eq:power lower bound in main results}
\end{equation}
where $\varepsilon$ is the accuracy of the $\epsilon$-net, and $\hat{h}_{n,\alpha,\lambda_{<1}} (\varepsilon)$ is given in~\eqref{eq:beta upper bound} in Section~\ref{sec:analysis and consistency}.  In addition, we also show that the accuracy of our local test $\hat{\mathcal{G}}_\mathbf{z}$ can be quantified according to
\begin{equation}
    \sup_{f_1,f_0 \in H_1(\mathbf{w},\gamma)}\mathbb{E} [ \inf_{\hat{\mathbf{w}}\in \hat{\mathcal{G}}_\mathbf{z}} \mathcal{E}_{\operatorname{TV}} (\hat{\mathbf{w}},\mathbf{w})  \mid  f_1,f_0]  \leq \varepsilon \cdot \min \left\{\frac{1-p}{2\gamma}, \frac{1}{2}\right\} + \operatorname{exp}\left[-2(\gamma - \hat{h}_{n,\alpha,\lambda_{<1}}(\varepsilon))^2\right]. \label{eq:accuracy of local test in main results}
\end{equation}
Recall that the error in the left-hand side of~\eqref{eq:accuracy of local test in main results} appears in the second summand of the local risk~\eqref{eq:hypothesis testing local risk def}, and is the worst-case discrepancy (in expected total variation distance) between $\mathbf{w}$ from a true alternative $H_1(\mathbf{w},0)$ and its closest element in $\hat{\mathcal{G}}_\mathbf{z}$. 
Evidently, if the lower bound in~\eqref{eq:power lower bound in main results} on the power of the test $Q_\mathbf{z}$ is large, then the error in~\eqref{eq:accuracy of local test in main results} is small (provided that $\varepsilon$ is sufficiently small).
Therefore, it is desirable to minimize $\hat{h}_{n,\alpha,\lambda_{<1}} (\varepsilon)$ over $\varepsilon$. We refer the reader to tables~\ref{table:values of beta_hat} and~\ref{table:values of vareps} in Appendix~\ref{appendix:tables} for the minimized values of $\hat{h}_{n,\alpha,\lambda_{<1}} (\varepsilon)$ and the corresponding best $\varepsilon$, respectively, for a wide range of the parameters $n$, $\lambda_{<1}$, and $\alpha$ (covering most practical situations). It is worthwhile to point out that the minimized values of $\hat{h}_{n,\alpha,\lambda_{<1}} (\varepsilon)$ are confined to the interval $(2.6,4.7)$ for the considered parameters, which is useful for extracting interpretable non-asymptotic guarantees on the power of the test $Q_\mathbf{z}$ and the accuracy of the local test $\hat{\mathcal{G}}_\mathbf{z}$.

By characterizing $\hat{h}_{n,\alpha,\lambda_{<1}} (\varepsilon)$ asymptotically, we provide in Theorem~\ref{thm:consistency rate} in Section~\ref{sec:analysis and consistency} sufficient conditions for the local consistency of our test. In particular, we show that our test is locally consistent (taking the significance $\alpha$ as $1/\log n$) if
\begin{equation}
    \liminf_{n\rightarrow \infty}\frac{\gamma_n}{\sqrt{\log n}} > C,
\end{equation}
where $C$ can always be taken as $1$, or it can be taken between $\sqrt{0.5}$ and $1$ depending on the behavior of $\lambda_{< 1}$ with $n$. In particular, we can take $C=\sqrt{0.5}$ if $\lambda_{<1}$ is bounded away from $1$ for all $n$, in which case our test is locally consistent if $\gamma_n$ is asymptotically larger than $\sqrt{0.5 \log n}$.
Otherwise, our test is always locally consistent if $\gamma_n$ is asymptotically larger than $\sqrt{\log n}$, regardless of $W$. Recall that if a local test is locally consistent, then it is also globally consistent.
In Theorem~\ref{thm:minimax risk} in Section~\ref{sec:minimax risk} we complement our consistency guarantees by showing that if
\begin{equation}
    \limsup_{n \rightarrow \infty} \frac{\gamma_n}{\sqrt{\log n}} < \frac{1-p}{\sqrt{\log (p^{-1})}},
\end{equation}
then there exists no single local test that can be globally consistent (and hence locally consistent) universally for all matrices $W$ satisfying Assumption~\ref{assump:W properties}. This result, together with our local consistency guarantees, implies that our local test achieves the minimax detection rate in terms of $\gamma_n$ (of a single test for all matrices $W$), which is $\sqrt{\log n}$, for attaining either global or local consistency. See Figure~\ref{fig:summary of cinsistency guarantees} for a summary our consistency guarantees.

\begin{figure}
\begin{centering}
\includegraphics[width=0.8\linewidth]{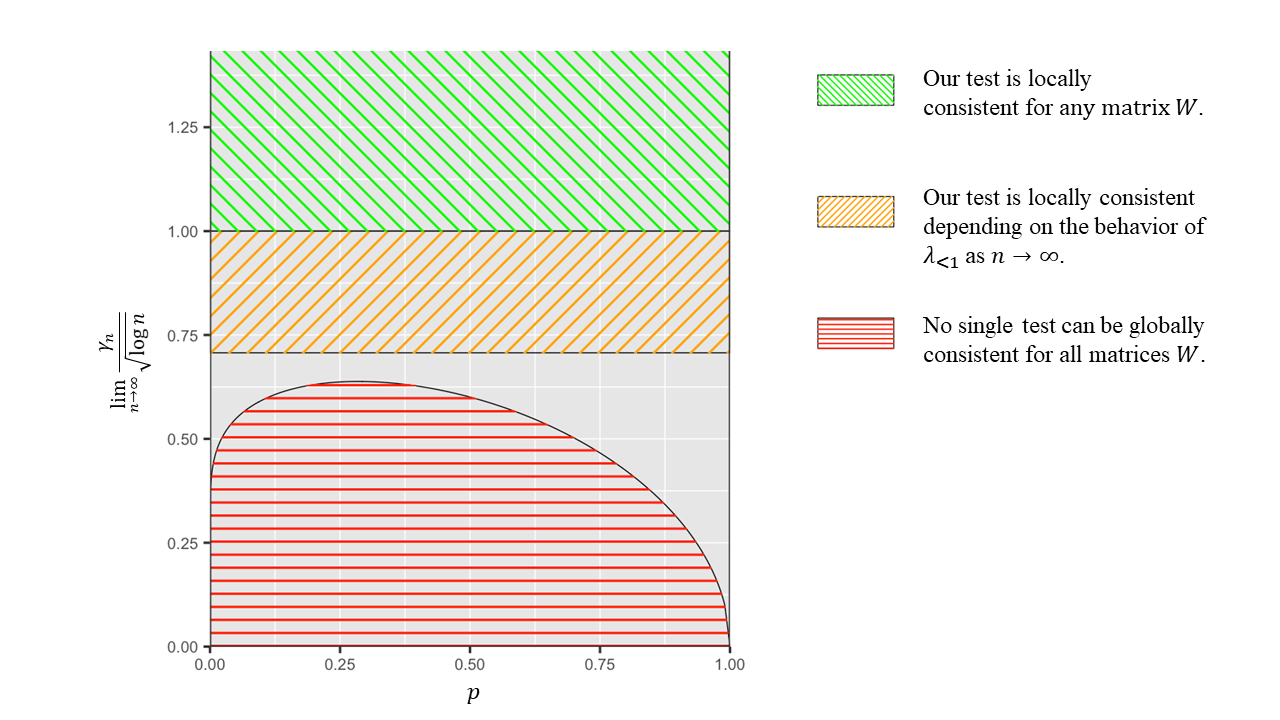}
   \caption{Summary of our consistency guarantees. If $\lim_{n\rightarrow \infty}\gamma_n / \sqrt{\log n} < (1-p)/\sqrt{\log (p^{-1})}$, then no local test can be globally (and locally) consistent for all matrices $W$ satisfying Assumption~\ref{assump:W properties}; see Theorem~\ref{thm:minimax risk} in Section~\ref{sec:minimax risk}. If $\lim_{n\rightarrow \infty}\gamma_n / \sqrt{\log n} \in (\sqrt{0.5}, 1)$, then the local (and global) consistency of our test depends on the convergence rate of $\lambda_{<1}$ to $1$ as $n\rightarrow \infty$; see Theorem~\ref{thm:consistency rate} in Section~\ref{sec:analysis and consistency}. If $\lim_{n\rightarrow \infty}\gamma_n / \sqrt{\log n} > 1$, then our test is locally (and globally) consistent for all matrices $W$ satisfying Assumption~\ref{assump:W properties}}. \label{fig:summary of cinsistency guarantees}
   \end{centering}
\end{figure} 

\subsubsection{Examples}
In Section~\ref{sec:toy example circle} we demonstrate our approach on a synthetic example where $p=0.5$, $f_0$ and $f_1$ are restricted to a closed curve in the Euclidean space, and the Gaussian kernel is used to form $K$. We then showcase our ability to accurately detect the region of the curve where $f_1>f_0$, and corroborate our theoretical results concerning the power of our test. In Sections~\ref{sec:arsenic example} and~\ref{sec:scrna-seq example} we apply our method to two real-world datasets, the first concerns Arsenic well contamination across the United States, and the second is single-cell RNA sequencing data from melanoma patients before and after immunotherapy treatment. We analyze each of these datasets using Algorithm~\ref{alg:local two-sample testing by aRWSS}, and demonstrate how to extract useful insights from the alternatives that are accepted by our method.

\section{Method derivation and analysis} \label{sec:testing methodology}
\subsection{Preliminaries} \label{sec:preliminaries}
Let $G$ be a weighted graph over $\{x_1,\ldots,x_n\}$ with adjacency matrix $W$ satisfying Assumption~\ref{assump:W properties}, where $W_{i,j} \neq 0$ if and only if $x_i$ and $x_j$ are connected. Suppose that $G$ has exactly $L$ connected components, where the vertex indices of the $\ell$'th connected component are given by $\mathcal{C}_\ell \subset \{1,\ldots,n\}$, with size $n_\ell :=  | \mathcal{C}_\ell  | $. We define $W^{(\ell)}$ to be the principal submatrix of $W$ that corresponds to $\mathcal{C}_\ell$, namely $W^{(\ell)}:= [W]_{i\in\mathcal{C}_\ell,\; j\in\mathcal{C}_\ell}$, where $[W]_\Omega$ is the restriction of $W$ to its entries in the subset $\Omega$.

The next proposition provides several trivial spectral properties of $W$.
 \begin{prop} [Spectral properties of $W$] \label{prop:W spectral properties}
 We have the following:
 \begin{enumerate}
         \item $W$ admits an eigen-decomposition with real-valued and nonnegative eigenvalues $\{\lambda_1^{(\ell)},\ldots,\lambda^{(\ell)}_{n_\ell}\}_{\ell=1}^L$ and corresponding orthonormal eigenvectors $\{\psi_1^{(\ell)},\ldots,\psi_{n_\ell}^{(\ell)}\}_{\ell=1}^L$. \label{prop:spctral properties 1}
         \item For each $\ell=1,\ldots,L$, the values $\lambda_1^{(\ell)},\ldots,\lambda^{(\ell)}_{n_\ell}$ and the vectors $[\psi_1^{(\ell)}[j]]_{j\in\mathcal{C}_\ell},\ldots,[\psi_{n_\ell}^{(\ell)}[j]]_{j\in\mathcal{C}_\ell} \in \mathbb{R}^{n_\ell}$ are the eigenvalues and eigenvectors, respectively, of the principal submatrix $W^{(\ell)}$. Additionally, for each $\ell=1,\ldots,L$ we have $\psi_k^{(\ell)}[j] = 0$ for all $j \notin \mathcal{C}_\ell$ and all $k$. \label{prop:spctral properties 2}
         \item For each $\ell = 1,\ldots,L$, we have $ 1 = \lambda_1^{(\ell)} > \lambda_2^{(\ell)} \geq \lambda_3^{(\ell)} \geq \ldots \geq \lambda_{n_\ell}^{(\ell)} \geq 0$,  and $\psi_1^{(\ell)}[j] = 1/\sqrt{n_\ell}$ for all $j\in\mathcal{C}_\ell$. \label{prop:spctral properties 3}
     \end{enumerate}
 \end{prop}
 \begin{proof}
 Property~\ref{prop:spctral properties 1} follows from the fact that $W$ is symmetric and PSD, while property~\ref{prop:spctral properties 2} follows from the fact that the rows and columns of $W$ can be permuted (symmetrically) into a block-diagonal form with $\{W^{(\ell)}\}_{\ell=1}^L$ on its main diagonal. Last, property~\ref{prop:spctral properties 3} follows from the fact that for each $\ell=1,\ldots,L$, $W^{(\ell)}$ is nonnegative, irreducible, and doubly stochastic (see chapter 8 in~\cite{horn2012matrix}).
 \end{proof}
  Using the spectral properties of $W$ described in Proposition~\ref{prop:W spectral properties}, the following proposition establishes key properties of the random-walk distributions $\{W_i^t\}_{i,t}$, fundamental to the results and their proofs presented throughout Section~\ref{sec:testing methodology} and the appendices. 
 \begin{prop} [Properties of the random-walk distributions $W_i^t$] \label{prop:random walk distribution properties}
 The following holds for all $i=1,\ldots,n$:
     \begin{enumerate}
         \item $\Vert W_i^t \Vert_2$ is monotonically decreasing in $t$, i.e., $\Vert W_i^\tau \Vert_2 \leq \Vert W_i^t \Vert_2$ for all positive integers $\tau \geq t$. \label{prop:random-walk distribution properties 2}
         \item For all positive integers $t^{'} \leq t \leq \tau \leq \tau^{'}$, we have $\Vert W_i^t - W_i^\tau \Vert^2_2
         \leq \Vert W_i^{t^{'}} \Vert_2^2 - \Vert W_i^{\tau^{'}} \Vert^2_2$. \label{prop:random-walk distribution properties 3}
         \item If $i\in\mathcal{C}_\ell$, then $W_{i,j}^t = 0$ for all $j\notin \mathcal{C}_\ell$ and integers $t>0$, and ${1}/{n_\ell} \leq  \Vert W_i^t \Vert_2^2 \leq {1}/{n_\ell} + (\lambda_2^{(\ell)})^{2t}$. \label{prop:random-walk distribution properties 1}
     \end{enumerate}
       \label{property:diffusion time effects}
 \end{prop}
 The proof appears in Appendix~\ref{appendix:proof of proposition on random-walk distribution properties}, and is based on the eigen-decomposition of $W^t$ and Proposition~\ref{prop:W spectral properties}.

\subsection{Procedure for local two-sample testing over the family $\mathcal{F}$} \label{sec:random walk scan statistics}
As mentioned in the beginning of Section~\ref{sec:approach and problem formulation}, we assume throughout this section that $x_1,\ldots,x_n$ are given, and all quantities are conditioned on $x_1,\ldots,x_n$ by default. Consequently, $\mathcal{F}$ is deterministic according to Assumption~\ref{assump:F independent of labels}, and each label $z_i$ is independently sampled from $(Z \mid X=x_i)$, see~\eqref{eq:z_i distribution}. 
Throughout Section~\ref{sec:testing methodology} we also assume that the class prior $p$ is known, and we refer the reader to Section~\ref{sec:adapting to unknown p} for an adaptation of our approach to the setting where $p$ is unknown.

Let us define
\begin{equation}
    Y = Z - p, \qquad\qquad y_i = z_i - {p}, \quad i=1,\ldots,n, \label{eq:y def}
\end{equation}
and denote $\mathbf{y} = [y_1,\ldots,y_n]$. Using~\eqref{eq:s def}, we have that
\begin{equation}
    \mathbb{E} \left[Y \mid X=x\right] = s(x). \label{eq:expected value of y_i is s(x_i)}
\end{equation}
Notice that $y_1,\ldots,y_n$ from~\eqref{eq:y def} are independent (since $z_1,\ldots,z_n$ are independent), and bounded. Hence, when conditioned on $x_1,\ldots,x_n$, the variables $\{w_i y_i\}_{i=1}^n$ are also independent and bounded (since $\mathcal{F}$ is deterministic according to Assumption~\ref{assump:F independent of labels}). Consequently, we have the following lemma, which is based on Hoeffding's inequality~\cite{hoeffding1994probability} for sums of independent and bounded random variables.
\begin{lem} \label{lem:specific w bound}
For any fixed $\alpha\in (0,1)$ and $\mathbf{w}\in\mathcal{F}$, we have that
\begin{align}
     &\operatorname{Pr} \left\{\frac{\langle \mathbf{w}, \mathbf{y}\rangle}{\Vert \mathbf{w} \Vert_2} >  \frac{\langle \mathbf{w}, \mathbf{s} \rangle}{\Vert \mathbf{w} \Vert_2} + \sqrt{0.5 \log(1/\alpha)} \right\} \leq \alpha, \label{eq:specific w upper bound} \\
     &\operatorname{Pr} \left\{ \frac{\langle \mathbf{w}, \mathbf{y}\rangle}{\Vert \mathbf{w} \Vert_2} < \frac{\langle \mathbf{w}, \mathbf{s} \rangle}{\Vert \mathbf{w} \Vert_2} - \sqrt{0.5 \log(1/\alpha)} \right\} \leq \alpha. \label{eq:specific w lower bound}
\end{align} 
\end{lem} 
The proof can be found in Appendix~\ref{appendix:proof of specific w bound}.
Recall that under the specific null $H_0(\mathbf{w})$ we have that $\langle \mathbf{w}, \mathbf{s} \rangle \leq 0$. Therefore, given $\alpha\in (0,1)$ and a specific $\mathbf{w}\in\mathcal{F}$, if we wish to test $H_0(\mathbf{w})$ against $H_1(\mathbf{w},0)$ we can reject $H_0(\mathbf{w})$ at significance level $\alpha$ if
\begin{equation}
   \mathcal{S}(\mathbf{w}) := \frac{\langle \mathbf{w}, \mathbf{y}\rangle}{\Vert \mathbf{w} \Vert_2} > \sqrt{0.5 \log(1/\alpha)}. \label{eq:test for specific w}
\end{equation}
It is important to note that the above test is invariant to the choice of $\mathcal{F}$, and holds for any $\mathbf{w}$ which is a deterministic probability distribution over $\{x_1,\ldots,x_n\}$. 

Evidently, the test~\eqref{eq:test for specific w} is useful only for testing $H_0(\mathbf{w})$ against $H_1(\mathbf{w},0)$ for a predefined $\mathbf{w} \in \mathcal{F}$. To test $H_0$ against $H_1$, it is natural to consider the scan statistic $\sup_{\mathbf{w}\in\mathcal{F}} \mathcal{S}(\mathbf{w}) $
which extends $\mathcal{S}(\mathbf{w})$ by searching over all $\mathbf{w}\in\mathcal{F}$.
Notably, $\sup_{\mathbf{w}\in\mathcal{F}} \mathcal{S}(\mathbf{w}) $ scans over an infinite number of time steps $t$ for each $W_i^t$, hence it is desirable to approximate $\sup_{\mathbf{w}\in\mathcal{F}} \mathcal{S}(\mathbf{w}) $ by computing $\mathcal{S}(\mathbf{w})$ only for $\mathbf{w}$ in a finite subset of $\mathcal{F}$. 
Specifically, for each index $i$, we approximate $\mathcal{S}(W_i^t)$ by replacing $W_i^t$ with one of $M_i$ possible distributions $W_i^{t_1},\ldots,W_i^{t_{M_i}}$ corresponding to carefully-chosen time steps $1 = t_{1}^{(i)} < t_{2}^{(i)} < \ldots < t_{M_i}^{(i)}$, whose evaluation we describe next. 

Given a prescribed accuracy parameter $\varepsilon\in (0,1)$, we define
\begin{equation}
T_\ell = \left\lceil \frac{\log(n_\ell/\varepsilon)}{\log((\lambda_2^{(\ell)})^{-1})} \right\rceil, \qquad \ell = 1,\ldots,L, \label{eq:T def}
\end{equation}
recalling that $\lambda_{2}^{(\ell)}$ is the second-largest eigenvalue of $W^{(\ell)} = [W_{i,j}]_{i\in\mathcal{C}_\ell, \; j\in\mathcal{C}_\ell}$ (see Proposition~\ref{prop:W spectral properties}), and $n_\ell$ is the size of the $\ell$'th connected component of $G$ (i.e., $n_\ell=  | \mathcal{C}_\ell | $). If $\lambda_{2}^{(\ell)} = 0$, we define $T_\ell=1$ (which is the limit of~\eqref{eq:T def} as $\lambda_{2}^{(\ell)}\rightarrow 0$).
Note that $\lambda_{2}^{(\ell)}<1$ according to Proposition~\ref{prop:W spectral properties}, hence $\log((\lambda_2^{(\ell)})^{-1}) > 0$. 
Now, for each $i=1,\ldots,n$, we start by taking $t_{M_i}^{(i)} = T_\ell$, and proceed by finding the time steps $t_{M_i-1}^{(i)} > t_{M_i-2}^{(i)} > \ldots > t_{1}^{(i)} = 1$ recursively, going backwards towards smaller values of $t$, by taking $t_{j-1}^{(i)}$ as the smallest integer $t \in [1,t_j^{(i)}-1]$ that satisfies
\begin{equation}
    \Vert W_i^{t+1}\Vert_2^2 \leq \Vert W_{i}^{t_j^{(i)}} \Vert_2^2 (1+{\varepsilon^2}/{n_\ell}). \label{eq:choosing t_j iterative inequality}
\end{equation}
This procedure is summarized in Algorithm~\ref{alg:evaluating diffusion time steps}. Note that the time steps $t_1^{(i)},\ldots,t_{M_i}^{(i)}$ (and their number -- $M_i$) may vary for each index $i$.
In Section~\ref{sec:analysis and consistency} we show that choosing the points $\{t_{1}^{(i)}, \ldots, t_{M_i}^{(i)}\}_{i=1}^n$ according to Algorithm~\ref{alg:evaluating diffusion time steps} guarantees that the set $\{ \mathcal{S} (\mathbf{w}): \; \mathbf{w}\in\widetilde{\mathcal{F}}\}$ forms an $\epsilon$-net over $\{ \mathcal{S} (\mathbf{w}): \;\mathbf{w}\in{\mathcal{F}} \}$ with controlled accuracy $\varepsilon$ (for arbitrary labels $z_1,\ldots,z_n$).
\begin{algorithm}
\caption{Evaluating $\{t_1^{(i)},\ldots,t_{M_i}^{(i)} \}_{i=1}^n$}\label{alg:evaluating diffusion time steps}
\begin{algorithmic}[1]
\ForAll{$\ell\in \{1,\ldots,L\}$ and $i \in \mathcal{C}_\ell$}
    \State Initialize: $\tau_1^{(i)}=T_\ell$ from~\eqref{eq:T def}, $k=1$. \label{step:find t_1}
    \While{$\tau_k^{(i)} \neq 1$}:
        \State Take $\tau_{k+1}^{(i)}$ as the smallest integer $t \in [1,\tau_k^{(i)}-1]$ such that $\Vert W_i^{t+1}\Vert_2^2 \leq \Vert W_{i}^{\tau_k^{(i)}} \Vert_2^2 (1+{\varepsilon^2}/{n_\ell})$. \label{step:find t_j_i}
        \State Update: $k\leftarrow k+1$.
    \EndWhile
    \State Set $M_i = k$.
\EndFor
\State Return $t_j^{(i)} = \tau_{M_i-j+1}^{(i)}$ for all $j=1,\ldots,M_i$ and $i=1,\ldots,n$.
\end{algorithmic}
\end{algorithm}

Given $\{t_1^{(i)},\ldots,t_{M_i}^{(i)} \}_{i=1}^n$, we define $\widetilde{\mathcal{F}}$ as the set
\begin{equation}
    \widetilde{\mathcal{F}} = \{W_i^t: \; 1\leq i \leq n, \; t = t_1^{(i)},\ldots,t_{M_i}^{(i)} \}.
\end{equation}
Since the cardinality of $\widetilde{\mathcal{F}}$ is $\sum_{i=1}^{n} M_i$, using Lemma~\ref{lem:specific w bound} and applying the union bound over the set $\widetilde{\mathcal{F}}$ (while replacing $\alpha$ with $\alpha/\sum_{i=1}^n M_i$), we get that
\begin{equation}
    \operatorname{Pr} \left\{ \bigcup_{\mathbf{w} \in \widetilde{\mathcal{F}}} \left\{ \frac{\langle \mathbf{w}, \mathbf{y}\rangle}{\Vert \mathbf{w} \Vert_2} >  \sqrt{0.5 \log(\sum_{i=1}^n M_i/\alpha)} + \frac{\langle \mathbf{w}, \mathbf{s} \rangle}{\Vert \mathbf{w} \Vert_2} \right\} \right\} \leq \alpha. \label{eq:random-walk scan statistic bound for w in F_hat}
\end{equation}
Therefore, we define our local test $\hat{\mathcal{G}}_\mathbf{z}$ as
\begin{equation}
    \hat{\mathcal{G}}_\mathbf{z} = \left\{ \mathbf{w}\in\widetilde{\mathcal{F}}: \; \mathcal{S}(\mathbf{w}) > \sqrt{0.5 \log(\sum_{i=1}^n M_i/\alpha)} \right\}, \label{eq:G_hat def}
\end{equation}
which according to~\eqref{eq:random-walk scan statistic bound for w in F_hat} (see also~\eqref{eq:type I error inequality between global and local} in Appendix~\ref{appendix:relation between local and global risks}) guarantees that 
\begin{equation}
    \sup_{f_0,f_1 \in H_0} \operatorname{Pr} \{ Q_\mathbf{z} = 1   \mid  f_0, f_1 \} 
    \leq \sup_{\mathcal{G} \subseteq \mathcal{F}}\sup_{f_0,f_1\in H_0(\mathcal{G})}\operatorname{Pr} \{ \hat{\mathcal{G}}_{\mathbf{z}} \cap \mathcal{G} \neq \emptyset   \mid  f_0, f_1 \} 
    \leq \alpha, \label{eq:type I error bound in local test}
\end{equation}
where $Q_\mathbf{z}=0$ if $\hat{\mathcal{G}}_\mathbf{z}$ is empty and $Q_\mathbf{z}=1$ otherwise, recalling that $H_0 = H_0(\mathcal{F})$. Note that in this case $Q_\mathbf{z}$ is the test that rejects $H_0$ and accepts $H_1$ if the scan statistic $\max_{\mathbf{w}\in\widetilde{\mathcal{F}}} \mathcal{S}(\mathbf{w})$ exceeds the threshold $\sqrt{0.5 \log(\sum_{i=1}^n M_i/\alpha)}$. Since each $\mathbf{w}\in \hat{\mathcal{G}}_\mathbf{z}$ describes a rejected null hypothesis $H_0(\mathbf{w})$, our approach here is equivalent to applying the test in~\eqref{eq:test for specific w} to all $\mathbf{w}\in \widetilde{\mathcal{F}}$ while correcting for multiple testing via the Bonferroni procedure (to control the family-wise error rate in the strong sense). While our overall approach here is conservative, we show in Sections~\ref{sec:analysis and consistency} and~\ref{sec:minimax risk} that for our particular choice of $\widetilde{\mathcal{F}}$ this approach is in fact nearly optimal in a minimax sense in terms of the sequences $\{\gamma_n\}$ that guarantee global and local consistency. We mention that one may use a procedure other than Bonferroni's for controlling the family-wise error rate in the strong sense, such as Holm's~\cite{holm1979simple} or Hochberg's~\cite{hochberg1988sharper}. The advantage of using Bonferroni, aside from its simplicity, is that the null hypotheses that are rejected by our method are in line with a certain quantity that we provide for describing effect size, as discussed next.

For practical purposes, aside from providing distributions $\mathbf{w}$ from rejected $H_0(\mathbf{w})$, when analyzing a two-sample dataset it is useful to have a quantitative measure of discrepancy between $f_1$ and $f_0$ for each accepted alternative $H_1(\mathbf{w},0)$, which can speak of effect size rather than significance.
To this end, observe that according to~\eqref{eq:random-walk scan statistic bound for w in F_hat}, if we reject $H_0(\mathbf{w})$ at significance $\alpha$, we also reject the hypothesis $\langle \mathbf{w}, \mathbf{s}\rangle < \hat{\gamma} \Vert \mathbf{w}\Vert_2$ at significance $\alpha$, where
\begin{equation}
   \hat{\gamma} =  \mathcal{S}(\mathbf{w}) - \sqrt{0.5 \log(\sum_{i=1}^n M_i/\alpha)}, \label{eq:approx random-walk scan statistic estimate alpha}
\end{equation}
which is equivalent to accepting the alternative $H_1(\mathbf{w},\hat{\gamma})$.
Importantly, each accepted alternative $H_1(\mathbf{w},\hat{\gamma})$ provides a lower bound on $\langle \mathbf{w}, \mathbf{s} \rangle$ through~\eqref{eq:H_1 specific def}, which acts as a local measure of discrepancy between $f_1$ and $f_0$. 

\subsection{Analysis and consistency guarantees} \label{sec:analysis and consistency}
As mentioned in Section~\ref{sec:random walk scan statistics}, the purpose of Algorithm~\ref{alg:evaluating diffusion time steps} is to provide an $\epsilon$-net over $\{\mathcal{S}(\mathbf{w}):\; \mathbf{w}\in\mathcal{F}\}$, in the sense that each statistic $\mathcal{S}(W_i^t)$ can be approximated by $\mathcal{S}(W_i^{t_j^{(i)}})$, for some $j$, to a prescribed accuracy $\varepsilon$.
Specifically, we define an approximation scheme as follows. Given the points $\{t_1^{(i)},\ldots,t_{M_i}^{(i)} \}_{i=1}^n$ from Algorithm~\ref{alg:evaluating diffusion time steps}, we approximate $\mathcal{S}(W_i^t)$ by $\mathcal{S}(W_i^{\pi(t)})$, where ${\pi}$ is the map
\begin{equation}
    {\pi}(t) := 
    \begin{dcases}
    1, & t=1, \\
    t_{j+1}^{(i)}, & t_{j}^{(i)} < t \leq t_{j+1}^{(i)}, \quad 1 \leq j \leq M_i-1, \\
    t_{M_i}^{(i)}, & t > t_{M_i}^{(i)}. 
    \end{dcases} \label{eq:pi map def}
\end{equation}
We then have the following result.
\begin{lem} [$\epsilon$-net properties] \label{lem:epsilon net}
Fix $0<\varepsilon<1$, and let $\{ t_1^{(i)},\ldots,t_{M_i}^{(i)} \}_{i=1}^n$ be the output of Algorithm~\ref{alg:evaluating diffusion time steps}. Then,
\begin{equation}
\left\vert \mathcal{S}(W_i^t) - \mathcal{S}( W_{i}^{\pi(t)}) \right\vert \leq \varepsilon,  \label{eq:eps-net thm bound}
\end{equation}
and
\begin{equation}
\mathcal{E}_{\operatorname{TV}}(W_i^t,W_i^{\pi(t)}) 
\leq \frac{\varepsilon \Vert W_i^t \Vert_2}{2}, \label{eq:total variation dist bound eps-net}
\end{equation}
for all ${i=1,\ldots,n}$ and ${t=1,2,\ldots,\infty}$.
\end{lem}
The proof can be found in Appendix~\ref{appendix:proof of epsilon net lemma}.
Lemma~\ref{lem:epsilon net} establishes that $\{\mathcal{S}(\mathbf{w})\}_{\mathbf{w} \in \widetilde{\mathcal{F}}}$ is an $\epsilon$-net over $\{\mathcal{S}(\mathbf{w})\}_{\mathbf{w} \in {\mathcal{F}}}$ with controlled accuracy $\varepsilon$. Hence, if we take $\varepsilon$ small enough, we lose very little information by not computing all possible statistics $\{\mathcal{S}(\mathbf{W})\}_{\mathbf \in \mathcal{F}}$. Additionally, since $\Vert W_i^t \Vert_2 \leq 1$ for all $i$ and $t$, Lemma~\ref{lem:epsilon net} also establishes that $\widetilde{\mathcal{F}}$ is an $\epsilon$-net over $\mathcal{F}$ in total variation distance with accuracy $\varepsilon/2$. Therefore, for sufficiently small $\varepsilon$, not only that $\mathcal{S}(W_i^{\pi(t)})$ is large if $\mathcal{S}(W_i^{t})$ large, but also $W_i^{\pi(t)}$ is close to $W_i^{t}$ in total variation distance.
For these arguments to be useful, we need to show that $\mathcal{S}(W_i^{\pi(t)})$ is large under the alternative $H_1(W_i^t,\gamma)$ (if $\gamma$ is large enough). This is the subject of the following corollary of Lemmas~\ref{lem:specific w bound} and~\ref{lem:epsilon net}.
\begin{cor} \label{cor:H_1(alpha) bound}
Fix $0<\varepsilon<1$ and $0<\widetilde{\alpha}<1$, and let $\{t_1^{(i)},\ldots,t_{M_i}^{(i)}\}_{i=1}^n$ be from Algorithm~\ref{alg:evaluating diffusion time steps}. 
Then, under $H_1(W_i^t,\gamma)$, with probability at least $1-\widetilde{\alpha}$ we have
\begin{equation}
    \mathcal{S}(W_i^{\pi(t)}) > \gamma - \varepsilon - \sqrt{0.5\log ({1}/{\widetilde{\alpha}} )}, \label{eq:test 1 bound H_1}
\end{equation}
for all $i=1,\ldots,n$ and $t=1,2,\ldots,\infty$.
\end{cor} 
\begin{proof}
Under $H_1(W_i^t,\gamma)$, we know that ${\langle W_{i}^{t} ,\mathbf{s}\rangle}/{\Vert W_{i}^{t} \Vert_2 } > \gamma$.
Therefore, employing Lemma~\ref{lem:epsilon net} and the probabilistic lower bound~\eqref{eq:specific w lower bound}, we can write 
\begin{align}
    \mathcal{S}(W_i^{\pi(t)}) \geq \mathcal{S}(W_{i}^{t}) - \varepsilon 
    \geq \frac{\langle W_{i}^{t} ,\mathbf{s}\rangle}{\Vert W_{i}^{t} \Vert_2 } - \varepsilon - \sqrt{0.5 \log(1/\widetilde{\alpha})} >  \gamma - \varepsilon - \sqrt{0.5 \log(1/\widetilde{\alpha})},
\end{align}
with probability at least $1-\widetilde{\alpha}$.
\end{proof}
Recall that according to our definition of $\hat{\mathcal{G}}_\mathbf{z}$ in~\eqref{eq:G_hat def}, we accept all alternatives $H_1(\mathbf{w},0)$ for which $\mathcal{S}(\mathbf{w}) > \sqrt{0.5 \log(\sum_{i=1}^n M_i/\alpha)}$, where $\mathbf{w} \in \widetilde{\mathcal{F}}$ and $M_i$ is number of chosen time steps in the $\epsilon$-net for each index $i$. Consequently, Corollary~\ref{cor:H_1(alpha) bound} implies that under $H_1(W_i^t,\gamma)$, we accept the alternative $H_1(W_i^{\pi(t)},0)$ with probability at least $1-\widetilde{\alpha}$ if $\gamma - \varepsilon - \sqrt{0.5 \log(1/\widetilde{\alpha})} > \sqrt{0.5 \log(\sum_{i=1}^n M_i/\alpha)}$. This enables us to obtain the power of the test $Q_\mathbf{z}$ against any alternative $H_1(W_i^t,\gamma)$ in terms of the quantity $\sum_{i=1}^n M_i =  | \widetilde{\mathcal{F}} | $. 
To make this quantity more meaningful, we have the following result concerning $M_i$.
\begin{lem} [Number of $\epsilon$-net nodes] \label{lem:t_1 and M_i bounds}
Fix $0<\varepsilon<1$, and let $\{ t_1^{(i)},\ldots,t_{M_i}^{(i)} \}_{i=1}^n$ be the output of Algorithm~\ref{alg:evaluating diffusion time steps}. Then,
\begin{equation}
    M_i \leq \min \left\{ T_\ell,  \left\lceil \frac{\log(n_\ell)}{\log(1+\varepsilon^2/n_\ell)} \right\rceil \right\},  \qquad i\in \mathcal{C}_\ell, \quad \ell=1,\ldots,L, \label{eq:M_i bound}
\end{equation}
where $T_\ell$ is from~\eqref{eq:T def}.
\end{lem}
The proof can be found in Appendix~\ref{appendix:proof of lemma t_1 and M_i bounds}, and is based on the recurrence relation (in $j$) $\Vert W_i^{t_{j-1}^{(i)}}\Vert^2_2 > \Vert W_i^{t_{j}^{(i)}} \Vert^2_2 (1+\frac{\varepsilon^2}{n_\ell})$, which follows immediately from step~\ref{step:find t_j_i} in Algorithm~\ref{alg:evaluating diffusion time steps}. It is noteworthy that $T_\ell$ (defined in~\eqref{eq:T def}) can be arbitrarily large if $\lambda_{2}^{(\ell)}$ approaches $1$. Nevertheless, and perhaps somewhat surprisingly, Lemma~\ref{lem:t_1 and M_i bounds} asserts that $M_i$ admits a universal bound independent of $W^{(\ell)}$ and its spectrum.
Fixing $\varepsilon$, it is of interest to briefly discuss the asymptotic behavior of $M_i$ and the size of $\widetilde{\mathcal{F}}$ as $n\rightarrow \infty$.
According to the definition of $T_\ell$, if $\lambda_{2}^{(\ell)}$ is bounded away from $1$ as $n\rightarrow \infty$, then Lemma~\ref{lem:t_1 and M_i bounds} asserts that $M_i  = \mathcal{O} (\log n_\ell) = \mathcal{O} (\log n)$, in which case $ | \widetilde{\mathcal{F}} |  = \sum_{i=1}^n M_i = \mathcal{O}(n \log n)$. 
On the other hand, even if $\lambda_{2}^{(\ell)}$ approaches $1$ arbitrarily fast as $n \rightarrow \infty$, we can use
\begin{equation}
    \left\lceil \frac{\log(n_\ell)}{\log(1+\varepsilon^2/n_\ell)} \right\rceil \leq \left\lceil \frac{\log(n)}{\log(1+\varepsilon^2/n)} \right\rceil \underset{n\rightarrow \infty}{\sim}{\frac{n \log n}{\varepsilon^2}}. \label{eq:log(1+n)/log(1+n/varEps^2) asymptotics}
\end{equation}
Therefore, for a fixed $\varepsilon$ we always have that $M_i  = \mathcal{O} (n \log n)$, and consequently $ | \widetilde{\mathcal{F}} |  = \mathcal{O}(n^2 \log n)$, regardless of $W$ and its spectrum.

Employing Corollary~\ref{cor:H_1(alpha) bound} and Lemma~\ref{lem:t_1 and M_i bounds}, we can now provide a lower bound on the power of the test $Q_\mathbf{z}$ against any alternative $H_1(\mathbf{w},\gamma)$, and also an upper bound on the error $\sup_{f_0,f_1 \in H_1(\mathbf{w},\gamma)} \mathbb{E} [ \inf_{\hat{\mathbf{w}}\in \hat{\mathcal{G}}_\mathbf{z}} \mathcal{E}_{\operatorname{TV}} (\hat{\mathbf{w}},\mathbf{w})  \mid  f_0,f_1]$, in terms of the quantities appearing in~\eqref{eq:M_i bound}. This is the subject of the next theorem.
\begin{thm} [Power and accuracy] \label{thm:test power}
Fix $0<\varepsilon<1$, $0<\alpha<1$, and let $\{ t_1^{(i)},\ldots,t_{M_i}^{(i)} \}_{i=1}^n$ be the output of Algorithm~\ref{alg:evaluating diffusion time steps}. Let $\hat{\mathcal{G}}_\mathbf{z}$ be as in~\eqref{eq:G_hat def} and suppose that $Q_{\mathbf{z}}$ is the test that outputs $1$ if $\hat{\mathcal{G}}_\mathbf{z}$ is empty, and $0$ otherwise.
Then, the power of the test $Q_{\mathbf{z}}$ over any alternative $H_1(\mathbf{w},\gamma)$, with $\mathbf{w}\in \mathcal{F}$, is at least 
\begin{equation}
    1 - \operatorname{exp}\left[-2(\gamma - h(\varepsilon))^2\right], \label{eq: power lower bound}
\end{equation}
for all $\gamma > h(\varepsilon)$, where
\begin{equation}
    h(\varepsilon) = \varepsilon + \sqrt{0.5\log\left(\sum_{\ell=1}^L \frac{n_\ell}{\alpha} \min \left\{ \left\lceil \frac{\log(n_\ell/\varepsilon)}{\log([\lambda_{2}^{(\ell)}]^{-1})} \right\rceil,  \left\lceil \frac{\log(n_\ell)}{\log(1+\varepsilon^2/n_\ell)} \right\rceil \right\} \right)}. \label{eq:beta def}
\end{equation}
Furthermore, for all $\mathbf{w}\in \mathcal{F}$ and $\gamma > h(\varepsilon)$ we have
\begin{equation}
    \sup_{f_1,f_0 \in H_1(\mathbf{w},\gamma)}\mathbb{E} [ \inf_{\hat{\mathbf{w}}\in \hat{\mathcal{G}}_\mathbf{z}} \mathcal{E}_{\operatorname{TV}} (\hat{\mathbf{w}},\mathbf{w})  \mid  f_1,f_0]  \leq \varepsilon \cdot \min \left\{\frac{1-p}{2\gamma}, \frac{1}{2}\right\} + \operatorname{exp}\left[-2(\gamma - h(\varepsilon))^2\right]. \label{eq: test accuracy upper bound}
\end{equation}
\end{thm}
The proof can be found in Appendix~\ref{appendix:proof of testing power and accuracy}.
Naturally, to maximize the power of the test against any alternative $H_1(\mathbf{w},\gamma)$ (for a fixed significance $\alpha$) it is desirable to make $h(\varepsilon)$ as small as possible. Therefore, the parameter $\varepsilon\in (0,1)$ should be chosen by minimizing the right-hand side of~\eqref{eq:beta def}, which can be accomplished numerically given $\alpha$, $\{n_\ell\}_{\ell=1}^L$, and $\{\lambda_{2}^{(\ell)}\}_{\ell=1}^L$. To somewhat simplify this minimization and related subsequent analysis, notice that
\begin{equation}
    h(\varepsilon) \leq \hat{h}_{n,\alpha,\lambda_{<1}}(\varepsilon) := \varepsilon + \sqrt{0.5\log\left( \frac{n}{\alpha} \min \left\{ \left\lceil \frac{\log(n/\varepsilon)}{\log(\lambda_{<1}^{-1})} \right\rceil,  \left\lceil \frac{\log(n)}{\log(1+\varepsilon^2/n)} \right\rceil \right\} \right)}, \label{eq:beta upper bound}
\end{equation}
which only depends only on $n$, $\varepsilon$, and $\lambda_{<1} := \max_{\ell} \lambda_{2}^{(\ell)}$, which is the largest eigenvalue of $W$ which is strictly smaller than $1$ (or equivalently, the $(L+1)$'th largest eigenvalue of $W$). Clearly, the results in Theorem~\ref{thm:test power} also hold if we replace $h(\varepsilon)$ with its upper bound $\hat{h}_{n,\alpha,\lambda_{<1}}(\varepsilon)$.
We refer the reader to Tables~\ref{table:values of beta_hat} and~\ref{table:values of vareps} in Appendix~\ref{appendix:tables}, where we list the values of $\varepsilon$ and the corresponding values of $\hat{h}_{n,\alpha,\lambda_{<1}}(\varepsilon)$ that minimize $\hat{h}_{n,\alpha,\lambda_{<1}}(\varepsilon)$ (via a grid search) for the array of parameters $n = 10^3,10^4,10^5,10^6$, $\alpha= 10^{-1},10^{-2},10^{-3}$, and $\lambda_{<1}= 0.9,0.99,0.999,1-10^{-4},1-10^{-5},1-10^{-6},1-10^{-9},1-10^{-12},1-10^{-16}$. Notably, for all of the above-mentioned values of $n$, $\alpha$, and $\lambda_{<1}$, the minimized values of $\hat{h}_{n,\alpha,\lambda_{<1}}(\varepsilon)$ are confined to the interval $(2.6, 4.7)$. Furthermore, when $\lambda_{<1} < 1 - 10^{-5}$ the bound in~\eqref{eq:beta upper bound} is dominated by $\log(n/\varepsilon)/\log(\lambda_{<1}^{-1})$ and the corresponding best values of $\varepsilon$ are around $0.005$. On the other hand, when $\lambda_{<1} > 1 - 10^{-5}$ the bound in~\eqref{eq:beta upper bound} becomes dominated by $\log(n)/\log(1+\varepsilon^2/n)$ and the corresponding best values of $\varepsilon$ are around $0.1$.

In essence, Theorem~\ref{thm:test power} provides a guarantee on the power of the test $Q_\mathbf{z}$ against any alternative $H_1(\mathbf{w},\gamma)$ for $\gamma>h(\varepsilon)$. Since the test $Q_\mathbf{z}$ controls the type I error at level $\alpha$ (see~\eqref{eq:type I error bound in local test}), Theorem~\ref{thm:test power} immediately provides an upper bound on the global Risk $R_\mathcal{F}^{(n)} (Q_\mathbf{z},\gamma)$ from~\eqref{eq:hypothesis testing global risk def}. Similarly, equation~\eqref{eq: test accuracy upper bound} in Theorem~\ref{thm:test power} provides an upper bound on the second summand of the local risk $r_\mathcal{F}^{(n)} (\hat{\mathcal{G}}_\mathbf{z},\gamma)$ from~\eqref{eq:hypothesis testing local risk def}, whereas the first summand in~\eqref{eq:hypothesis testing local risk def} is upper bounded by $\alpha$ (see~\eqref{eq:type I error bound in local test}). By analyzing the resulting upper bound on $r_\mathcal{F}^{(n)} (\hat{\mathcal{G}}_\mathbf{z},\gamma)$ asymptotically (as $n\rightarrow \infty$), we get the following theorem characterizing the local consistency of $\hat{\mathcal{G}}_\mathbf{z}$ from~\eqref{eq:G_hat def} (recalling the definitions of global and local consistencies in Section~\ref{sec:hypothesis testing framework}).
\begin{thm} [Local consistency guarantees] \label{thm:consistency rate}
Fix $\varepsilon \in (0,1)$, take $\alpha = 1/\log n$, and let $\hat{\mathcal{G}}_\mathbf{z}$ be the local test described in~\eqref{eq:G_hat def}. Additionally, let $\{\gamma_n\}$ be a sequence, and define
\begin{equation}
    c = \liminf_{n\rightarrow \infty}\frac{\gamma_n}{\sqrt{\log n}}.
\end{equation}
Then, $\hat{\mathcal{G}}_\mathbf{z}$ is locally consistent w.r.t. $\{\gamma_n\}$ if one of the following holds:
\begin{enumerate}
    \item $c>1$. \label{cor_part:cor consistency part 1}
    \item $c > \sqrt{0.5}$ and $\lambda_{<1}$ is bounded away from $1$ for all $n$. \label{cor_part:cor consistency part 2}
    \item $c > \sqrt{0.5 + \delta}$ for some $\delta \in (0,0.5)$, and $\lim_{n\rightarrow \infty} (1-\lambda_{<1}) n^\delta > 0$. \label{cor_part:cor consistency part 3}
\end{enumerate}
\end{thm}
The proof can be found in Appendix~\ref{appendix:proof of onsistency}. Fundamentally, part~\ref{cor_part:cor consistency part 1} of Theorem~\ref{thm:consistency rate} states that $\hat{\mathcal{G}}_\mathbf{z}$ is locally consistent, i.e., $r_{\mathcal{F}}^{(n)}(\hat{\mathcal{G}}_\mathbf{z},\gamma_n) \rightarrow 0$ as $n\rightarrow\infty$, as long as $\gamma_n$ grows asymptotically faster than $\sqrt{\log n}$ (even if by a factor slightly larger than $1$) for any matrix $W$ satisfying Assumption~\ref{assump:W properties}. Parts~\ref{cor_part:cor consistency part 2} and~\ref{cor_part:cor consistency part 3} improve upon the required growth of $\gamma_n$ from part~\ref{cor_part:cor consistency part 1} (by a constant factor) if $\lambda_{<1}$ is either bounded away from $1$, or it converges to $1$ no faster than $1/{{n}^\delta}$ for some $0 < \delta < 0.5$. As discussed in Section~\ref{sec:approach and problem formulation}, local consistency implies global consistency, hence $\hat{\mathcal{G}}_\mathbf{z}$ is globally consistent under the same conditions as in Theorem~\ref{thm:consistency rate}.

\subsection{Problem impossibility and minimax detection rate} \label{sec:minimax risk}
Let $\mathcal{W}$ be the space of all matrices satisfying assumption~\ref{assump:W properties}. In order to complement the consistency guarantees provided in Theorem~\ref{thm:consistency rate}, we consider the minimax global risk
\begin{equation}
    \widetilde{R}^{(n)}(\gamma) = \min_{Q_\mathbf{z}} \sup_{W \in \mathcal{W}} R_{\mathcal{F}}^{(n)}(Q_{\mathbf{z}},\gamma),  \label{eq:R_tilde minimax risk def}
\end{equation}
where the minimization in~\eqref{eq:R_tilde minimax risk def} is over all deterministic tests $Q_{\mathbf{z}}:\{0,1\}^n \rightarrow \{0,1\}$. In particular, our aim here is to provide necessary conditions on $\{\gamma_n\}$ so that $\lim_{n\rightarrow \infty}\widetilde{R}^{(n)}(\gamma_n) = 0$ can hold. Such conditions on $\{\gamma_n\}$ are necessary for any single local test to be globally consistent universally for all matrices $W\in \mathcal{W}$. In turn, such conditions are also necessary for any local test to be locally consistent universally for all matrices $W\in\mathcal{W}$ (by the virtue of Lemma~\ref{lem:relation between local and global risks}). Towards that end, we have the following theorem.

\begin{thm} \label{thm:minimax risk}
If $\{\gamma_n\}$ is a sequence that satisfies
\begin{equation}
    \limsup_{n \rightarrow \infty} \frac{\gamma_n}{\sqrt{\log n}} < \frac{1-p}{\sqrt{\log (p^{-1})}}, \label{eq:alpha_n lower rate bound}
\end{equation}
then $\lim_{n\rightarrow \infty }\widetilde{R}^{(n)}(\gamma_n) \geq 1$.
\end{thm}

The proof of Theorem~\ref{thm:minimax risk} can be found in Appendix~\ref{appendix:Proof of minimax risk}, and is based on analyzing the setting where the graph $G$ has approximately $ n/\log n$ connected components of size proportional to $ \log n$. Essentially, Theorem~\ref{thm:minimax risk} states that no local test can be globally consistent universally for all matrices $W\in\mathcal{W}$ (satisfying Assumption~\ref{assump:W properties}) if $\gamma_n$ is asymptotically smaller than $\sqrt{\left( (1-p)^2/\log (p^{-1}) \right) \log n}$. Therefore, the same conclusion holds for local consistency. On the other hand, from Theorem~\ref{thm:consistency rate} we know that $\hat{\mathcal{G}}_\mathbf{z}$ from~\eqref{eq:G_hat def} is locally consistent (and hence globally consistent) universally for all $W\in\mathcal{W}$ if $\gamma_n$ is asymptotically larger than $\sqrt{\log n}$. Combining Theorems~\ref{thm:consistency rate} and~\ref{thm:minimax risk} implies that $\gamma_n$ has to grow with rate at least $\sqrt{\log n}$ (disregarding constants) for any single $\hat{\mathcal{G}}_\mathbf{z}$ to be locally or globally consistent for all $W\in\mathcal{W}$. We therefore conclude that our local test $\hat{\mathcal{G}}_\mathbf{z}$ achieves the minimax detection rate of $\gamma_n$ both globally and locally. 

\section{Adapting to unknown prior $p$} \label{sec:adapting to unknown p}
Next, we treat the case where the prior $p$ is unknown, and must be inferred from the labels $z_1,\ldots,z_n$. To this end, we return to the full probabilistic model described in the introduction, where both $z_1,\ldots,z_n$ and $x_1,\ldots,x_n$ are random, sampled independently from the joint distribution of $Z$ and $X$. 
In other words, we do not condition on $x_1,\ldots,x_n$, which was required in Sections~\ref{sec:approach and problem formulation} and~\ref{sec:testing methodology} for the hypothesis testing framework to be well defined (by making $W$ and $\mathcal{F}$ non random). 

In the full probabilistic model discussed here, the labels $z_1,\ldots,z_n$ are sampled independently from Bernoulli($p$),  and therefore $\sum_{i=1}^n z_i \sim \text{Binomial}(n,p)$.
Estimating confidence intervals for a binomial proportion $p$ is a problem with a long history and extensive literature (see for example~\cite{vollset1993confidence,brown2001interval} and the references therein). One popular approach is the Clopper–Pearson method~\cite{clopper1934use}, which is based on inverting a Binomial test. The Clopper–Pearson method is an \textit{exact} method, meaning that for a prescribed $\alpha\in (0,1)$, the one-sided Clopper–Pearson method outputs an upper bound $\hat{p}_+$ such that $p \leq \hat{p}_+$ with probability at least $1-\alpha$, a probability referred to as the \textit{coverage}. 

In order to modify our method from Section~\ref{sec:random walk scan statistics} to account for unknown $p$, we require an upper bound for $p$. We have the following proposition, which is an analogue of~\eqref{eq:random-walk scan statistic bound for w in F_hat} when using an estimated upper bound for $p$ instead of $p$ directly.
We emphasize that the claim ``with probability'' in the next proposition is interpreted in the sense of sampling $z_1,\ldots,z_n$ and $x_1,\ldots,x_n$ from the joint distribution of $Z$ and $X$.
\begin{prop} \label{prop:statistic bound bound for all w in F and unknown p}
Suppose that $p\leq \hat{p}_{+}$ with probability at least $1-\alpha/2$. Then, with probability at least $1-\alpha$
\begin{equation}
    \hat{\mathcal{S}}(\mathbf{w}):=\frac{\langle \mathbf{w},\mathbf{z} - \hat{p}_{+}\rangle}{\Vert \mathbf{w} \Vert_2} \leq  \sqrt{0.5 \log(2\sum_{i=1}^n M_i/\alpha)} + \frac{\langle \mathbf{w}, \mathbf{s} \rangle}{\Vert \mathbf{w} \Vert_2}, \qquad \forall \mathbf{w}\in\widetilde{\mathcal{F}}. \label{eq:random-walk scan statistic bound for w in F_hat with unknown p}
\end{equation}
\end{prop}
\begin{proof}
Observe that Lemma~\ref{lem:specific w bound} and subsequently the probabilistic bound in~\eqref{eq:random-walk scan statistic bound for w in F_hat} (obtained from Lemma~\ref{lem:specific w bound} using the union bound) hold conditionally on any $x_1,\ldots,x_n$. Therefore, even though $W$, $\mathcal{F}$, and $\{t_j^{(i)}\}_{i,j}$ are random variables in the setting of this section, Lemma~\ref{lem:specific w bound} and~\eqref{eq:random-walk scan statistic bound for w in F_hat} also hold unconditionally of $x_1,\ldots,x_n$. 
Hence, using the probabilistic bound~\eqref{eq:random-walk scan statistic bound for w in F_hat} with $\alpha/2$ instead of $\alpha$, and together with the union bound, we have with probability at least $1-{\alpha}$
\begin{equation}
    \frac{\langle \mathbf{w}, \mathbf{z} - \hat{p}_{+}\rangle}{\Vert \mathbf{w} \Vert_2} = \frac{\langle \mathbf{w}, {\mathbf{y}} + p-\hat{p}_{+}\rangle}{\Vert \mathbf{w} \Vert_2} \leq \frac{\langle \mathbf{w}, {\mathbf{y}}\rangle}{\Vert \mathbf{w} \Vert_2} \leq \sqrt{0.5 \log(2 \sum_{i=1}^n M_i/\alpha)} + \frac{\langle \mathbf{w}, \mathbf{s} \rangle}{\Vert \mathbf{w} \Vert_2},
\end{equation}
where we used the fact that $p-\hat{p}_{+}\leq 0$ with probability at least $1-{\alpha}/2$.
\end{proof}
Evidently, $\hat{p}_+$ from the (one-sided) Clopper–Pearson method can be used in conjunction with Proposition~\ref{prop:statistic bound bound for all w in F and unknown p}.
Employing Proposition~\ref{prop:statistic bound bound for all w in F and unknown p}, we can use the method described in Section~\ref{sec:random walk scan statistics} if we replace $\mathcal{S}(\mathbf{w})$ with $\hat{\mathcal{S}}(\mathbf{w})$ from~\eqref{eq:random-walk scan statistic bound for w in F_hat with unknown p} and replace $\alpha$ with $\alpha/2$, respectively. This modification can be found in Step~\ref{step:Clopper-Pearson} of Algorithm~\ref{alg:local two-sample testing by aRWSS} in Section~\ref{sec:main results}. We remark that when using $\hat{p}_{+}$ (which is an estimated upper bound for $p$) rather than knowing $p$ (as in Section~\ref{sec:random walk scan statistics}) the significance level $\alpha$ should be interpreted in the sense of the full probabilistic model described in Section~\ref{sec:inroduction}.

\section{Construction of $W$} \label{sec:W construction}
Our statistical framework is based on a matrix $W$ that satisfies Assumption~\ref{assump:W properties}.
To construct such a matrix, we require a graph (either directed or undirected) over $\{x_1,\ldots,x_n\}$ that encodes the similarities between $x_1,\ldots,x_n$. We assume that this graph is represented by a nonnegative affinity matrix $K\in\mathbb{R}^{n\times n}$, which is a deterministic function of $x_1,\ldots,x_n$. If an unweighted graph over $\{x_1,\ldots,x_n\}$ is provided, we simply take $K_{i,j}$ to be $1$ if $x_i$ is connected to $x_j$ and $0$ otherwise. 

Given a nonnegative matrix $K$, we define the matrices $\widetilde{W}, \widetilde{K} \in \mathbb{R}^{n\times n}$ via
     \begin{align}
         \widetilde{W}_{i,j} = d_i \widetilde{K}_{i,j} d_j, \qquad \qquad \widetilde{K}_{i,j} = \sqrt{K_{i,j} K_{j,i}} , \label{eq:doubly stochastic normalization} 
     \end{align}
where $d_1,\ldots,d_n > 0$ are the diagonal scaling factors of a doubly stochastic normalization of $\widetilde{K}$, i.e., $d_1,\ldots,d_n$ are such that the sums of all rows and all columns of $\widetilde{W}$ are $1$~\cite{sinkhorn1964relationship,idel2016review}. Such scaling factors exist if the matrix $\widetilde{K}$ has full-support~\cite{csima1972dad}, a property related to the zero-pattern of $\widetilde{K}$. Importantly, $d_1,\ldots,d_n$ always exist if $\widetilde{K}$ is strictly positive, or if it is zero only on its main diagonal (see~\cite{landa2020doubly}). Otherwise, the typical situation where $d_1,\ldots,d_n$ would not exist is if some rows/columns of $\widetilde{K}$ are too sparse, a problem that can be circumvented by discarding these rows and columns.
When they exist, the scaling factors $d_1,\ldots,d_n$ can be computed by the classical Sinkhorn-Knopp iterations~\cite{sinkhorn1967concerning}, or by more recent algorithms employing convex optimization~\cite{allen2017much}. 

Clearly, $\widetilde{W}$ from~\eqref{eq:doubly stochastic normalization} is nonnegative, symmetric, and stochastic. While there are countless ways to construct a matrix with these properties from $K$, $\widetilde{W}$ from~\eqref{eq:doubly stochastic normalization} admits the favorable property that it is the closest symmetric and stochastic matrix to $K$ in KL-divergence. Specifically, we have the following proposition.
\begin{prop} \label{prop:doubly stochastic KL-divergence interpretation}
Suppose that there exist $d_1,\ldots,d_n>0$ such that $\widetilde{W}$ from~\eqref{eq:doubly stochastic normalization} is stochastic. Then, $\widetilde{W}$ is also the solution to
\begin{align}
    &\underset{H\in \mathbb{R}^{n\times n}_+}{\text{Minimize}} \quad \sum_{i=1}^n D_{\text{KL}} (H_i  \mid  \mid  K_i) \quad \text{Subject to} \quad H \mathbf{1}_n = \mathbf{1}_n, \quad H = H^T, \label{eq:KL-divergence optim}
\end{align}
where $H_i$ and $K_i$ are the $i$'th rows of $H$ and $K$ respectively, $\mathbf{1}_n$ is a column vector of $n$ ones, and $D_{\text{KL}} (H_i  \mid  \mid  K_i) = \sum_{j=1}^n H_{i,j} \log ( {H_{i,j}}/{K_{i,j}} )$ is the Kullback–Leibler divergence from $K_i$ to $H_i$.
\end{prop}
We note that the result described in Proposition~\ref{prop:doubly stochastic KL-divergence interpretation} is already known for the special that $K$ is symmetric (and hence $\widetilde{K} = K$), see Proposition 2 in~\cite{zass2007doubly}. Therefore, the contribution of Proposition~\ref{prop:doubly stochastic KL-divergence interpretation} is to describe the appropriate form of symmetrization (i.e., the formula for $\widetilde{K}$) in the context of finding the closest symmetric and doubly stochastic matrix to an arbitrary nonnegative matrix $K$ (under KL-divergence loss). The proof of Proposition~\ref{prop:doubly stochastic KL-divergence interpretation} can be found in Appendix~\ref{appendix:doubly stochastic KL-divergence interpretation proof}, and follows from the Lagrangian of~\eqref{eq:KL-divergence optim}. Note that the KL-divergence is typically used for measuring discrepancies between probability distributions, whereas $\{K_i\}$ are not proper probability distributions. Nevertheless, it is easy to verify that replacing $K_i$ in~\eqref{eq:KL-divergence optim} with its normalized variant $K_i/\sum_{j=1}^n K_{i,j}$ leads to an equivalent optimization problem (using the fact that $\sum_{j=1}^n W_{i,j} = 1$ for all $i=1,\ldots,n$).

Finally, in order to obtain $W$ satisfying Assumption~\ref{assump:W properties} from $\widetilde{W}$ (in case $K$ is not PSD), we take
     \begin{equation} 
         W = \widetilde{W}^2, \label{eq:PSD normalization}
     \end{equation}
which further ensures that $W$ is PSD while retaining the properties held by $\widetilde{W}$ of non-negativity, symmetry, and stochasticity. Note that the random-walk arising from $W$ is equivalent to the one arising from $\widetilde{W}$ when restricting the latter to even time steps. Therefore,~\eqref{eq:PSD normalization} can also be interpreted as bypassing unstable periodic behavior of the random walk associated with $\widetilde{W}$ if it describes a bipartite graph.

\section{Examples}
\subsection{Simulation: data sampled from a closed curve} \label{sec:toy example circle}
In our first example, we simulated data points sampled uniformly from a smooth closed curve, over which $f_1(x)>f_0(x)$ in a certain localized region. We then analyzed the performance of the test described in Section~\ref{sec:random walk scan statistics}. In particular, we investigated the power of the test with respect to the effective size of the deviating region (where $f_1(x)>f_0(x)$) and the magnitude of the difference between $f_1$ and $f_0$. 
 
\subsubsection{The setup}
In this example, the points $x_1,\ldots,x_n \in \mathbb{R}^2$ were sampled from the unit circle $\mathbb{S}_1$, i.e.,
\begin{equation}
\begin{aligned}
    x_i[1] = \cos (\theta_i), \qquad 
    x_i[2] = \sin (\theta_i),
    \end{aligned} \label{eq:unit circle def}
\end{equation}
for $i=1,\ldots,n$. The angles $\theta_1,\ldots,\theta_n \in [-\pi, \pi)$ are i.i.d, each sampled from $f_1(\theta)$ or $f_0(\theta)$ with probability $p=0.5$, where
\begin{equation}
    \begin{aligned}
    f_1(\theta) = 
    \begin{dcases}
    \frac{1+b \cos(\omega \theta)}{2\pi}, &  \mid \theta \mid  < \frac{3\pi}{2\omega},\\
    \frac{1}{2\pi}, & \text{Otherwise}
    \end{dcases}, \qquad 
    f_0(\theta) = 
    \begin{dcases}
    \frac{1-b \cos(\omega \theta)}{2\pi}, &  \mid \theta \mid  < \frac{3\pi}{2\omega},\\
    \frac{1}{2\pi}, & \text{Otherwise}.
    \end{dcases}
    \end{aligned} \label{eq:unit circle example f_1 and f_0 def}
\end{equation}
We have that
\begin{equation}
    \begin{aligned}
    &f(\theta) = p f_1(\theta) + (1-p)f_0(\theta) = \frac{1}{2\pi}, \\
    &s(\theta) = p(1-p)\frac{f_1(\theta)-f_0(\theta)}{f(\theta)} =
    \begin{dcases}
    2p(1-p)b\cos(\omega \theta), &  \mid \theta \mid  < \frac{3\pi}{2\omega}, \\
    0, & \text{Otherwise}.
    \end{dcases}
    \end{aligned} \label{eq:unit circle example f and s def}
\end{equation}
It is evident from~\eqref{eq:unit circle example f and s def} that the points $x_1\ldots,x_n$ are sampled uniformly from the unit circle. 
Observe that $f_1(\theta) > f_0(\theta)$ only in a single contiguous region around $\theta=0$, whose size depends on the parameter $\omega$ (smaller values of $\omega$ correspond to larger regions where $f_1(\theta)>f_0(\theta)$, and vice-versa). Additionally, the parameter $b\in [0,1]$ controls the magnitude of the difference between $f_1$ and $f_0$, where $b=0$ results in the null hypothesis $H_0$, since $f_1(\theta) = f_0(\theta)$.
Figure~\ref{fig:Ring example f1,f2,f} illustrates the distributions $f_1(\theta)$, $f_0(\theta)$, and $f(\theta)$, for two scenarios where $(b,\omega) = (0.5,1)$ (figure (a)) and $(b,\omega) = (0.2,4)$ (figure (b)). The former corresponds to an ``easy'' scenario where the magnitude of $f_1-f_0$ is relatively large, and $f_1(\theta)>f_0(\theta)$ on a large portion of $[-\pi,\pi)$, whereas the latter corresponds to a ``hard'' scenario where the magnitude of $f_1-f_0$ is much smaller, and $f_1(\theta)>f_0(\theta)$ only in a restricted part of $[-\pi,\pi)$.

\begin{figure}
\begin{center}
    
\begin{subfigure}{0.35\textwidth}
\includegraphics[width=\linewidth]{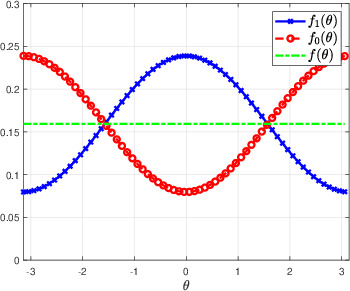}
\caption{``Easy'': $b=0.5$, $\omega=1$} \label{fig:a}
  \end{subfigure}
\hspace{0.1\textwidth}
\begin{subfigure}{0.35\textwidth}
\includegraphics[width=\linewidth]{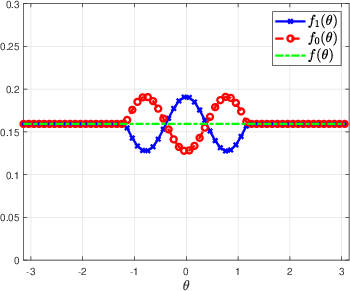}
\caption{``Hard'': $b=0.2$, $\omega=4$} \label{fig:b}
\end{subfigure}
\end{center}
   \caption{The distributions $f_1$, $f_0$, and $f$ corresponding to the unit circle example described in~\eqref{eq:unit circle def}--\eqref{eq:unit circle example f and s def}. In this example $f(\theta)$ is the uniform distribution, and $f_1(\theta)>f_0(\theta)$ around $\theta=0$ in a region determined by $\omega$, while the magnitude of $f_1(\theta)-f_0(\theta)$ is determined by $b$.} \label{fig:Ring example f1,f2,f}
\end{figure} 

After generating the points $x_1,\ldots,x_n$ together with the labels $z_1,\ldots,z_n$, we formed an affinity matrix $K$ using the Gaussian kernel
\begin{equation}
    K_{i,j} = \operatorname{exp}(-\Vert x_i - x_j \Vert_2^2/\sigma),  \label{eq:unit circle example Gaussian kernel}
\end{equation}
and followed with the construction of $W$ according to~\eqref{eq:doubly stochastic normalization} and~\eqref{eq:PSD normalization} in Section~\ref{sec:W construction}, using $\sigma = 0.01$.
Since $x_1,\ldots,x_n$ are sampled uniformly from a smooth Riemannian manifold without boundary, as $n\rightarrow \infty$ and $\sigma \rightarrow 0$ the matrix $W^t$ is expected to converge pointwise to the heat kernel on the manifold~\cite{marshall2019manifold}. Since our manifold is a smooth closed curve, the heat kernel is approximately the Gaussian kernel with respect to the geodesic distance. Therefore, for a suitable range of parameters $\sigma$, $n$, and $t$, we use the approximation
\begin{equation}
    W_{i,j}^t \approx C_i G_{2\sigma t}(\theta_i,\theta_j), \qquad \qquad G_\tau(\theta, \varphi) := \operatorname{exp}\{ -(\operatorname{mod}\{\theta - \varphi,2\pi\})^2/\tau\}. \label{eq:W_i^t analytical}
\end{equation}
where $ C_i$ is a normalization constant (accommodating for the fact that $\sum_{j=1}^n W_{i,j}^t = 1$).
Figure~\ref{fig:Ring example W_i_t} compares between $W_1^t$ and the right-hand side of~\eqref{eq:W_i^t analytical} for several values of $t$, where we used $\theta_1 = 0$ and sampled $\theta_2,\ldots,\theta_n$, with $n=5000$, independently and uniformly from $[0,2\pi)$. Indeed, Figure~\ref{fig:Ring example W_i_t} suggests that the approximation~\eqref{eq:W_i^t analytical} is highly accurate.

\begin{figure} 
\begin{subfigure}{0.28\textwidth}
\includegraphics[width=\linewidth]{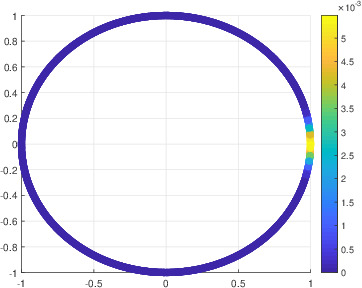}
\caption{$W_1^1$} \label{fig:a}
\end{subfigure}
\hspace{0.05\textwidth}
\begin{subfigure}{0.28\textwidth}
\includegraphics[width=\linewidth]{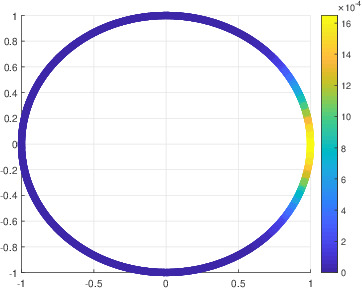}
\caption{$W_1^{10}$} \label{fig:a}
\end{subfigure}
\hspace{0.05\textwidth}
\begin{subfigure}{0.28\textwidth}
\includegraphics[width=\linewidth]{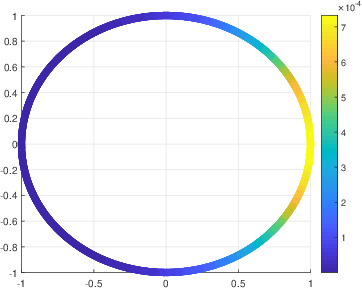}
\caption{$W_1^{50}$} \label{fig:b}
\end{subfigure}

\medskip
\begin{subfigure}{0.27\textwidth}
\includegraphics[width=\linewidth]{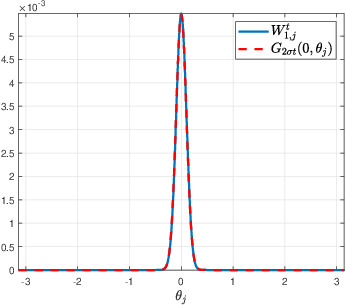}
\caption{$t=1$} \label{fig:e}
\end{subfigure}
\hspace{0.063\textwidth}
\begin{subfigure}{0.27\textwidth}
\includegraphics[width=\linewidth]{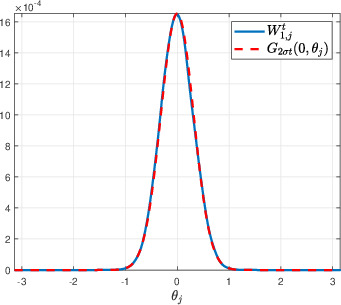}
\caption{$t=10$} \label{fig:e}
\end{subfigure}
\hspace{0.063\textwidth}
\begin{subfigure}{0.27\textwidth}
\includegraphics[width=\linewidth]{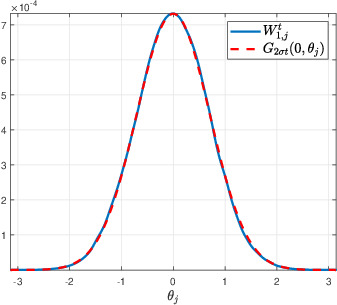}
\caption{$t=50$} \label{fig:f}
\end{subfigure}
\caption{The random-walk distributions $W_1^t$ (with $\theta_1 = 0$), $t=\{1,10,50\}$, versus their approximations in~\eqref{eq:W_i^t analytical}, for $n=5000$ and $\sigma=0.01$}. \label{fig:Ring example W_i_t}
\end{figure}

\subsubsection{Results}
In Figure~\ref{fig:ring example labels} we illustrate a typical array of points $x_1,\ldots,x_n$ and their labels $z_1,\ldots,z_n$, sampled according to~\eqref{eq:unit circle def}--\eqref{eq:unit circle example f and s def}, for $n=5000$, $b=0.3$, and $\omega = 2$. We added a small amount of noise to the coordinates of $x_1,\ldots,x_n$ for improved visibility of the labels. Although somewhat difficult to determine by the naked eye, there is a slightly larger number of blue points than red in the vicinity of $\theta = 0$, matching the fact that $f_1(\theta)>f_0(\theta)$ in that region. We next applied Algorithm~\ref{alg:local two-sample testing by aRWSS} to the labels $z_1,\ldots,z_n$ and the matrix $K$, with $\alpha=0.05$ and treating $p=0.5$ as known. We then found the indices $(i^\star,j^\star)$ that correspond to the largest statistic among $\{\mathcal{S}(W_i^{t_j^{(i)}})\}_{i,j}$, and denoted $t^\star = t_{j^{\star}}^{(i^\star)}$. Figure~\eqref{fig:ring example best W_i_t} colors the values of $W_{i^\star}^{t^\star}$ over the points $x_1,\ldots,x_n$, and Figure~\eqref{fig:ring example best W_i_t vs f_1 and f_0} compares between $f_1$, $f_0$ and $W_{i^\star}^{t^\star}$ (which was normalized appropriately for better visualization). Indeed, $W_{i^\star}^{t^\star}$ captures the region where $f_1(x)>f_0(x)$ over the points $x_1,\ldots,x_n$, as $W_{i^\star}^{t^\star}$ is localized around $\theta = 0$, and is extremely small for $ \mid \theta \mid >\pi/4$ (the region where $f_1(\theta)\leq f_0(\theta)$).

\begin{figure} 
\begin{subfigure}{0.33\textwidth}
\includegraphics[width=\linewidth]{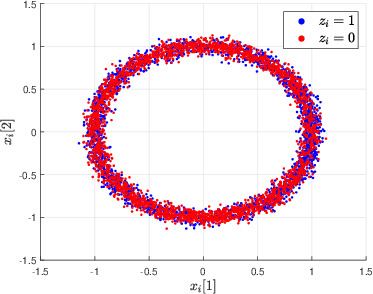}
\caption{$x_1,\ldots,x_n$ and $z_1,\ldots,z_n$} \label{fig:ring example labels}
\end{subfigure}
\begin{subfigure}{0.33\textwidth}
\includegraphics[width=\linewidth]{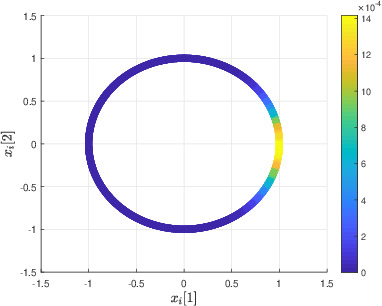}
\caption{$W_{i\star}^{t^\star}$} \label{fig:ring example best W_i_t}
\end{subfigure}
\begin{subfigure}{0.3\textwidth}
\includegraphics[width=\linewidth]{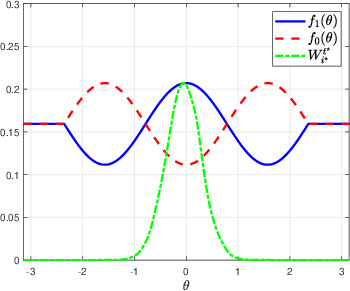} 
\caption{$f_1$ and $f_0$ versus $W_{i\star}^{t^\star}$} \label{fig:ring example best W_i_t vs f_1 and f_0}
\end{subfigure}
\caption{Figure (a) illustrates a typical array of points and labels sampled according to~\eqref{eq:unit circle def}--\eqref{eq:unit circle example f and s def} (with noise added for visualization purposes), for $n=5000$, $b=0.3$, and $\omega = 2$. Figures (b) and (c) show the distribution $W_{i^\star}^{t^\star}$ corresponding to the largest statistic $\mathcal{S}(W_i^{t_j^{(i)}})$ obtained from our method, using $K$ from~\eqref{eq:unit circle example Gaussian kernel} with $\sigma = 0.01$, $\alpha=0.05$, and $p=0.5$.}.
\end{figure}

According to Theorem~\ref{thm:test power}, our test has positive power to detect any alternative $H_1(\mathbf{w},\gamma)$ for which $\gamma > h(\varepsilon)$. Consequently, to analyze the performance of our test in this setting and compare with numerical findings, we need to characterize $\gamma$ from~\eqref{eq:H_1 specific def}. Note that we can take
\begin{equation}
    \gamma = \max_{1\leq i \leq n} \sup_{t=1,\ldots,\infty} \frac{\langle W_i^t, \mathbf{s}\rangle}{\Vert W_i^t \Vert_2}. \label{eq:alpha def explicit}
\end{equation}
In Appendix~\ref{appendix:analysis for unit circle example}, we analyze~\eqref{eq:alpha def explicit} using~\eqref{eq:W_i^t analytical} in the regime of large $n$ and small $\sigma$, and show that
\begin{equation}
    \gamma \approx  ({8}/{\pi e})^{1/4} b p (1-p) \sqrt{{n}/{\omega}}. \label{eq:Ring example alpha approx}
\end{equation}
Therefore, employing Theorem~\ref{thm:test power}, we expect our test $Q_\mathbf{z}$ to reject the null with probability at least $\rho$ if
\begin{equation}
  \gamma \approx  ({8}/{\pi e})^{1/4} b p (1-p) \sqrt{{n}/{\omega}}  > h(\varepsilon) + \sqrt{0.5 \log\left(1/(1-\rho)\right)}.
\end{equation}
The above condition can be written equivalently as a condition on $b$ or on $\omega$, according to
\begin{equation}
    b \gtrapprox \frac{h(\varepsilon) + \sqrt{0.5 \log\left(1/(1-\rho)\right)}}{({8}/{\pi e})^{1/4} p(1-p) \sqrt{{n}/{\omega}}}, \qquad \text{or} \qquad \omega \lessapprox \frac{n b^2 p^2 (1-p)^2 \sqrt{{8}/{\pi e}}  }{(h(\varepsilon) + \sqrt{0.5 \log\left(1/(1-\rho)\right)})^2}. \label{eq:numerical example 1 b or w cond}
\end{equation}

In Figure~\ref{fig:Ring toy example b vs n} we show the probability of rejecting $H_0$, as estimated from $20$ randomized trials over a grid of values of ${n}$ and $b$, using $\alpha=0.05$, $\varepsilon = 0.005$, and $\omega=1$. In Figure~\ref{fig:Ring toy example b vs n} we also plot the curve corresponding to the condition on $b$ in~\eqref{eq:numerical example 1 b or w cond} for power at least $\rho = 0.9$.
Analogously to Figure~\ref{fig:Ring toy example b vs n}, in Figure~\ref{fig:Ring toy example omega vs n} we show the probability of rejecting $H_0$ over a grid of values of ${n}$ and $\omega$, using the same $\alpha$, $\varepsilon$, and number of trials as for Figure~\ref{fig:Ring toy example b vs n}, while fixing $b=0.5$. In Figure~\ref{fig:Ring toy example omega vs n} we also plot the curve corresponding to the condition on $\omega$ in~\eqref{eq:numerical example 1 b or w cond} for power at least $\rho = 0.9$. As expected, from Figures~\ref{fig:Ring toy example b vs n} and~\ref{fig:Ring toy example omega vs n} we see that detecting $f_1(x)>f_0(x)$ locally becomes easier as $b$ increases or $\omega$ decreases. Even for small values of $b$ or large values of $\omega$, detection of $f_1(x)>f_0(x)$ is eventually possible for sufficiently large $n$, since $\gamma \propto \sqrt{n} > \sqrt{\log n}$. It can be observed from both Figures~\ref{fig:Ring toy example b vs n} and~\ref{fig:Ring toy example omega vs n}, that the theoretical curves corresponding to the conditions on $b$ or $\omega$ in~\eqref{eq:numerical example 1 b or w cond} agree very well with the simulation results.

\begin{figure}
\begin{subfigure}{0.45\textwidth}
\includegraphics[width=\linewidth]{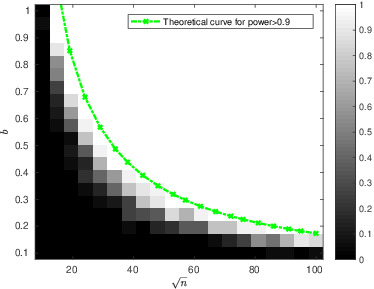}
\caption{$b$ versus $\sqrt{n}$, $\omega=1$} \label{fig:Ring toy example b vs n}
  \end{subfigure}\hspace*{\fill}
\begin{subfigure}{0.44\textwidth}
\includegraphics[width=\linewidth]{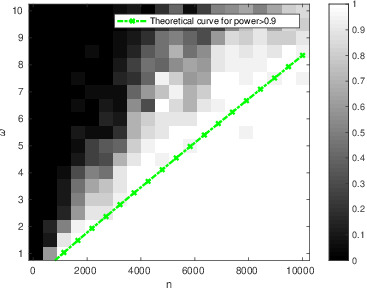}
\caption{$\omega$ versus $n$, $b=0.5$} \label{fig:Ring toy example omega vs n}
\end{subfigure}
   \caption{Empirical probability to accept $H_1$ from $20$ trials, as a function of $b$, $\omega$, and $n$, using $\alpha=0.05$, $\varepsilon=0.005$, and $\sigma=0.01$. The green curve appearing in figure (a) corresponds to the condition on $b$ from~\eqref{eq:numerical example 1 b or w cond}, while the green curve appearing in figure (b) corresponds to the condition on $\omega$ also from~\eqref{eq:numerical example 1 b or w cond}, where $\rho = 0.9$ (which is a lower bound on the power).}
\end{figure} 

\subsection{Detection and localization of arsenic well contamination} \label{sec:arsenic example}
To exemplify our method on a relatively simple low-dimensional real-world application, we analyzed arsenic concentration levels in $20,043$ domestic wells across the conterminous United States, where the goal is to detect regions of significant arsenic contamination. 
Even though the underlying geographic data here is low-dimensional, the purpose of this example is to illustrate that our random walk based approach is able to adapt to the particular spatial arrangement of the wells, which is highly nonuniform, without the need to predefine shapes for scanning the whole geographic region, as in classical spatial scan statistics.
We used data collected between 1973 and 2001, retrieved from the USGS National Water Information System~\cite{focazio2000retrospective}. We assigned a label of $1$ to all wells with arsenic concentration exceeding the U.S. Environmental Protection Agency's maximum allowed contamination level of $10 \mu g/L$, which makes about 10\% of all measured wells. Throughout this example, we set the significance level at $\alpha=0.01$, and formed $K$ from a symmetric $5$-nearest-neighbour graph between the wells (using geographic location). That is, $K_{i,j} = K_{j,i} = 1$ if well $x_i$ is one of the $5$ nearest wells to $x_j$ (excluding itself) or vice versa, and $K_{i,j} = K_{j,i} = 0$ otherwise. This graph construction was chosen primarily due to its simplicity, and the number of nearest neighbours was chosen small so to make $K$ sparse (and to consider only the immediate surroundings of each well). We remark that one may apply our testing methodology to several different graphs, and combine the results by correcting for multiple testing. 
We then formed $W$ as described in Section~\ref{sec:W construction}, where we repeatedly removed the sparsest rows/columns of $\widetilde{W}$ until it could be diagonally-scaled (to be doubly stochastic). We ended up with $19,569$ wells in $229$ different connected components in the graph $G$ (described by $W$). Most of the connected components were small ($123$ components with at most $20$ wells), and $25$ of them had at least $100$ wells, where the largest connected component contained $3,730$ wells.
An upper bound on the prior $p$ was found to be around $0.12$ (using the Clopper-Pearson method, see step~\ref{step:Clopper-Pearson} in Algorithm~\ref{alg:local two-sample testing by aRWSS}). 

Overall, Algorithm~\ref{alg:local two-sample testing by aRWSS} yielded roughly $ 10^7$ random-walk distributions that rejected the null, i.e., distributions $\mathbf{w}\in\widetilde{\mathcal{F}}$ for which $\langle \mathbf{w},\mathbf{s}\rangle>0$ (with significance $0.01$). These distributions correspond to random-walks that started at $2,262$ different wells in $7$ connected components (out of $229$). 
In Figure~\ref{fig:arsenic example W with smallest p value} we display the values of the random-walk distribution $W_i^{t_{j}^{(i)}}$ with the pair $(i,j)$ that corresponds to the largest statistic $\mathcal{S}(W_i^{t_j^{(i)}})$ from the scan. It is notable that this distribution is quite spread-out, highlighting mainly two communities of wells in Nevada, but also other regions in Oregon, Idaho, and Washington. The distribution depicted in Figure~\ref{fig:arsenic example W with smallest p value} is associated with a walk time $t_j^{(i)} = 256$, and provides the lower bound $\langle W_i^{t_{j}^{(i)}}, \mathbf{s}\rangle > 0.26$ (computed from $\hat{\gamma}_{i,j} \Vert W_i^{t_j^{(i)}}\Vert_2$, see step~\ref{step:compute alpha_hat} in Algorithm~\ref{alg:local two-sample testing by aRWSS} and equation~\eqref{eq:H_1 specific def}). In Figure~\ref{fig:arsenic example W with largest lower bound} we depict the distribution $W_i^{t_{j}^{(i)}}$ with the pair $(i,j)$ that provides the largest lower bound on $\langle \mathbf{w},\mathbf{s}\rangle$ among $\mathbf{w}\in\widetilde{\mathcal{F}}$, which is $\langle W_i^{t_{j}^{(i)}}, \mathbf{s}\rangle > 0.4$ for $t_j^{(i)}=24$. Recall that $\langle \mathbf{w}, \mathbf{s}\rangle \leq 1-p \approx 0.9$ for any distribution $\mathbf{w}$, and hence a lower bound of $0.4$ speaks of a substantial local difference between $f_1$ and $f_0$. The distribution in Figure~\ref{fig:arsenic example W with largest lower bound} is clearly much more localized than the distribution in Figure~\ref{fig:arsenic example W with smallest p value}. Interestingly, this distribution takes its largest value in the city of Fallon, Nevada, which is known for its high arsenic concentration in ground water, and has been the subject of several related studies~\cite{dauphine2013case,steinmaus2004probability}.
Our method can therefore be used to detect communities with significantly high arsenic contamination, allowing to further investigate the origins of the arsenic contamination or its effect on population health.

In Figure~\ref{fig:arsenic example Ws lower bound 0} we highlight all wells $x_i$ for which $\langle W_i^t, \mathbf{s}\rangle >0 $ for at least one value of walk time $t\in \{t_j^{(i)}\}_{j=1}^{M_i}$. Notably, almost all such wells are located in the west, specifically in Arizona, Nevada, California, Oregon, Idaho, Washington, and Montana. This is largely consistent with specialized literature on the subject (see for example~\cite{desimone2014water,ayotte2017estimating}), which observed that most severe arsenic contaminations are found in the west of the USA. It is important to mention that the random-walk distributions corresponding to the locations depicted in Figure~\ref{fig:arsenic example Ws lower bound 0} can potentially be very spread-out (depending on the walk time that rejected the null), and may describe a region of the size of a state or even a few states. Furthermore, even if a distribution $W_i^t$ rejected the null with high significance, the quantity $\langle W_i^t, \mathbf{s}\rangle$ could be very small, making the result less substantial. To complement the picture, in Figure~\ref{fig:arsenic example Ws lower bound 0.05} we highlight all wells $x_i$ for which the scan found $\langle W_i^t, \mathbf{s}\rangle >0.05$ for at least one value of walk time $t\in \{t_j^{(i)}\}_{j=1}^{M_i}$. Figure~\ref{fig:arsenic example Ws lower bound 0.05} highlights substantially less wells compared to Figure~\ref{fig:arsenic example Ws lower bound 0}, and perhaps paints a more meaningful picture that is based on effect size rather than significance. To better understand the regions in which $f_1>f_0$, it is important to explore the actual distributions $W_i^t$ that rejected the null for each index $i$, possibly focusing on the ones that rejected the null with the smallest walk time, or the ones that provide the largest lower bounds on $\langle W_i^t, \mathbf{s}\rangle$.

\begin{figure}
\begin{centering}
\includegraphics[width=0.7\linewidth]{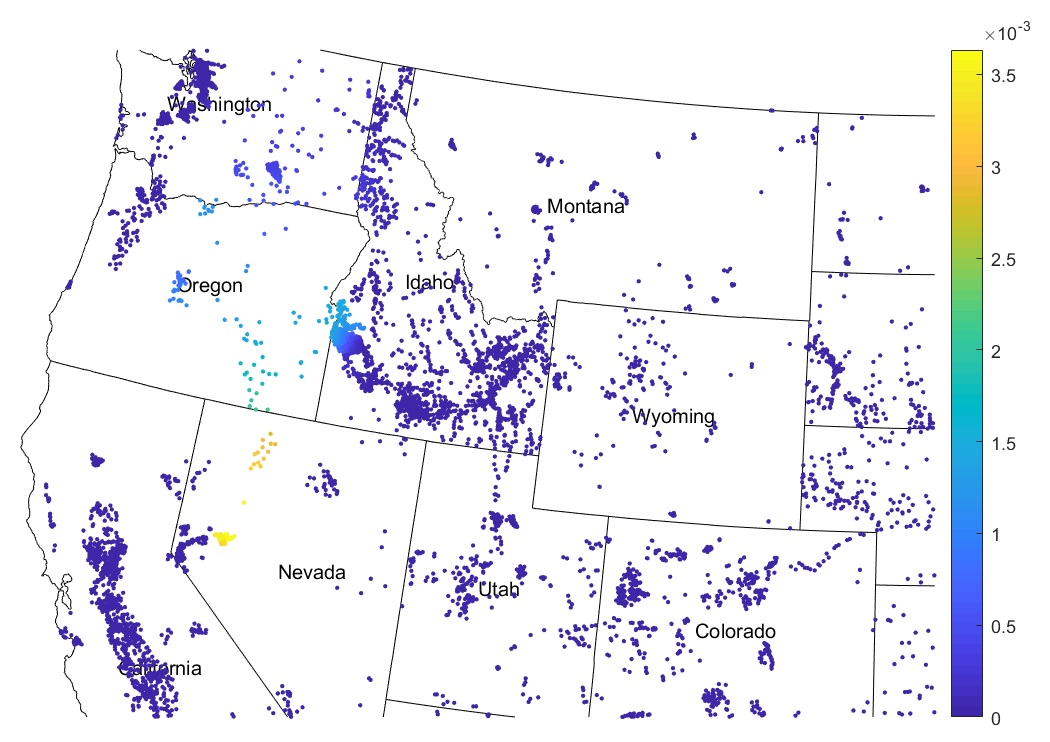}
   \caption{Values of the distribution $W_i^{t_j^{(i)}}$ for the pair $(i,j)$ that corresponds to the largest statistic $\mathcal{S}(W_i^{t_j^{(i)}})$. This distribution is associated with the lower bound $\langle W_i^{t_j^{(i)}}, \mathbf{s} \rangle>\hat{\gamma}_{i,j} \Vert W_i^{t_j^{(i)}} \Vert_2 = 0.26$.} \label{fig:arsenic example W with smallest p value}
   \end{centering}
\end{figure} 

\begin{figure}
\begin{center}
    \includegraphics[width=0.7\linewidth]{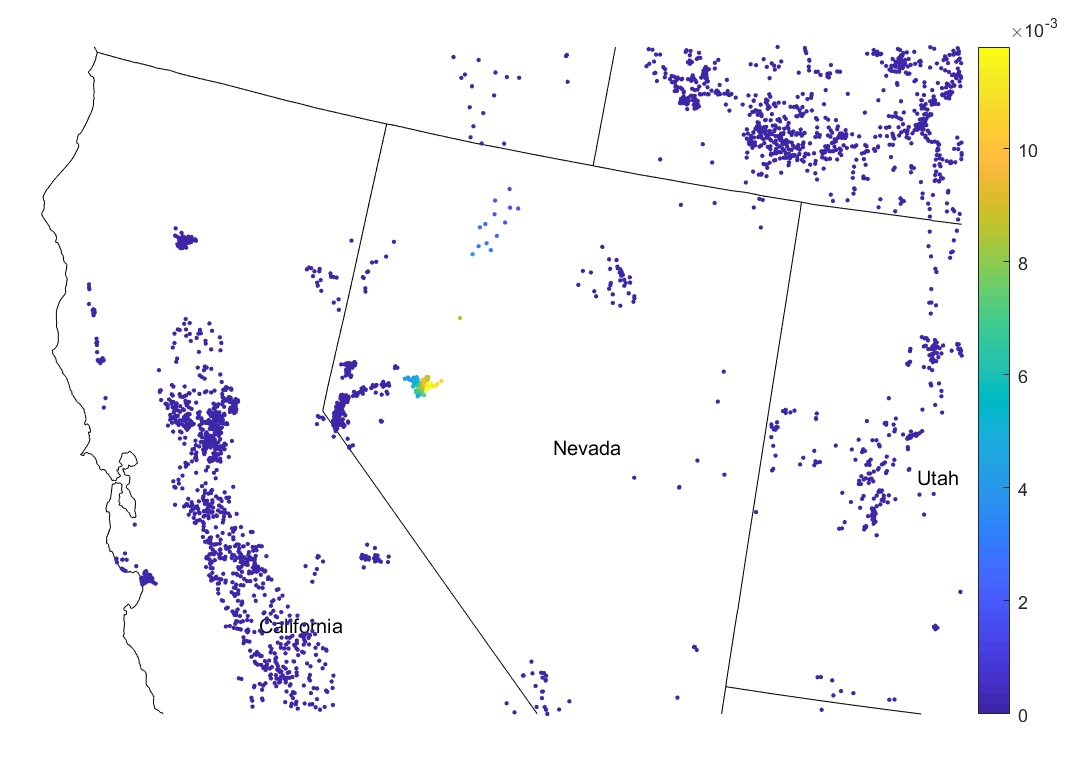}
   \caption{Values of the distribution $W_i^{t_j^{(i)}}$ for the pair $(i,j)$ that corresponds to the largest lower bound $\langle W_i^{t_j^{(i)}}, \mathbf{s} \rangle>\hat{\gamma}_{i,j} \Vert W_i^{t_j^{(i)}} \Vert_2 = 0.4$. The largest values of $W_i^{t_j^{(i)}}$ are highly concentrated in the city of Fallon, Nevada.} \label{fig:arsenic example W with largest lower bound}
\end{center}
\end{figure} 

\begin{figure}
\begin{center}
    \includegraphics[width=1\linewidth]{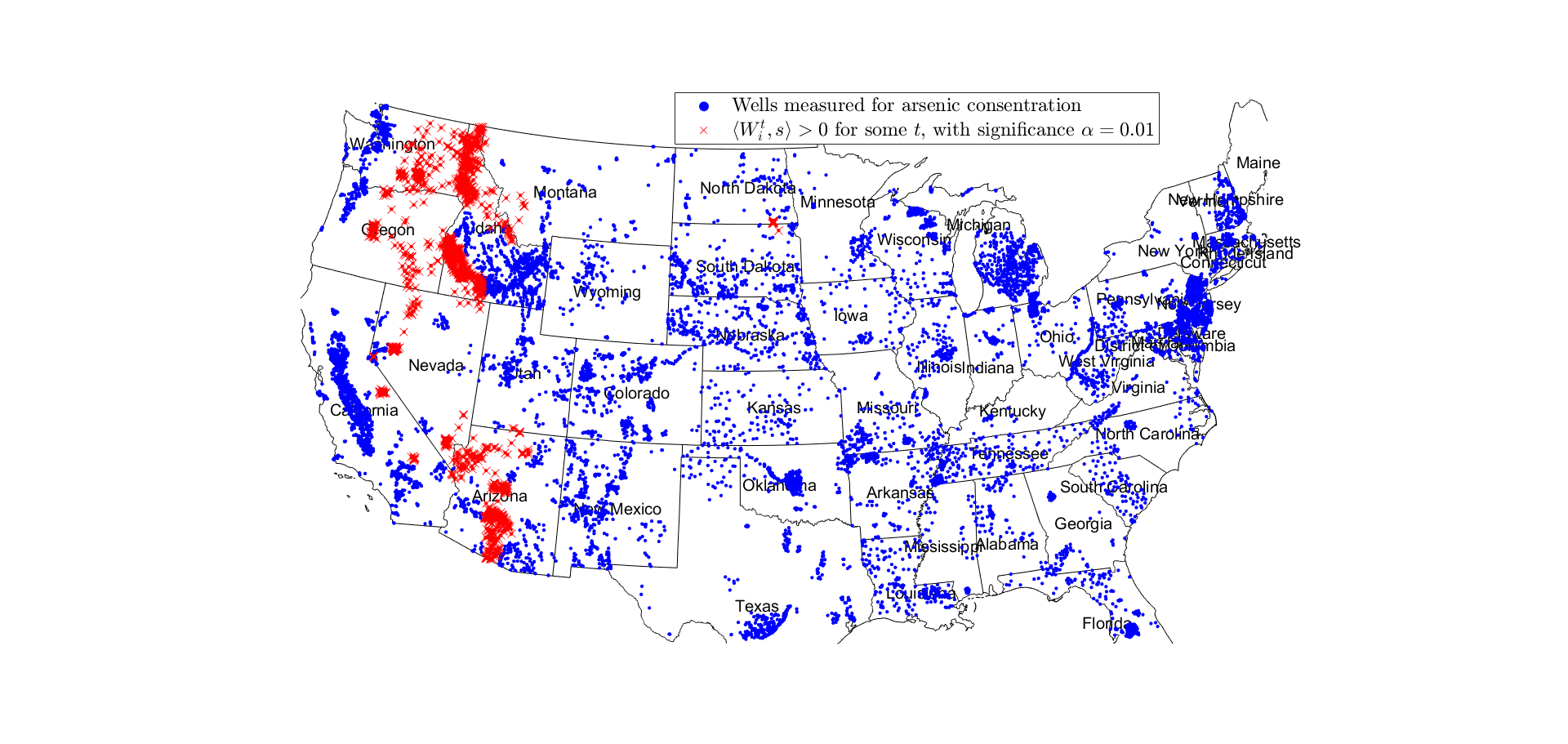}
   \caption{Wells $x_i$ for which it was detected that $\langle W_{i}^t,\mathbf{s}\rangle >0$ with significance $\alpha=0.01$ for some walk time $t\in \{t_j^{(i)}\}_{j=1}^{M_i}$, versus all wells in which arsenic level was measured.} \label{fig:arsenic example Ws lower bound 0}
\end{center}
\end{figure} 

\begin{figure}
\begin{center}
    \includegraphics[width=1\linewidth]{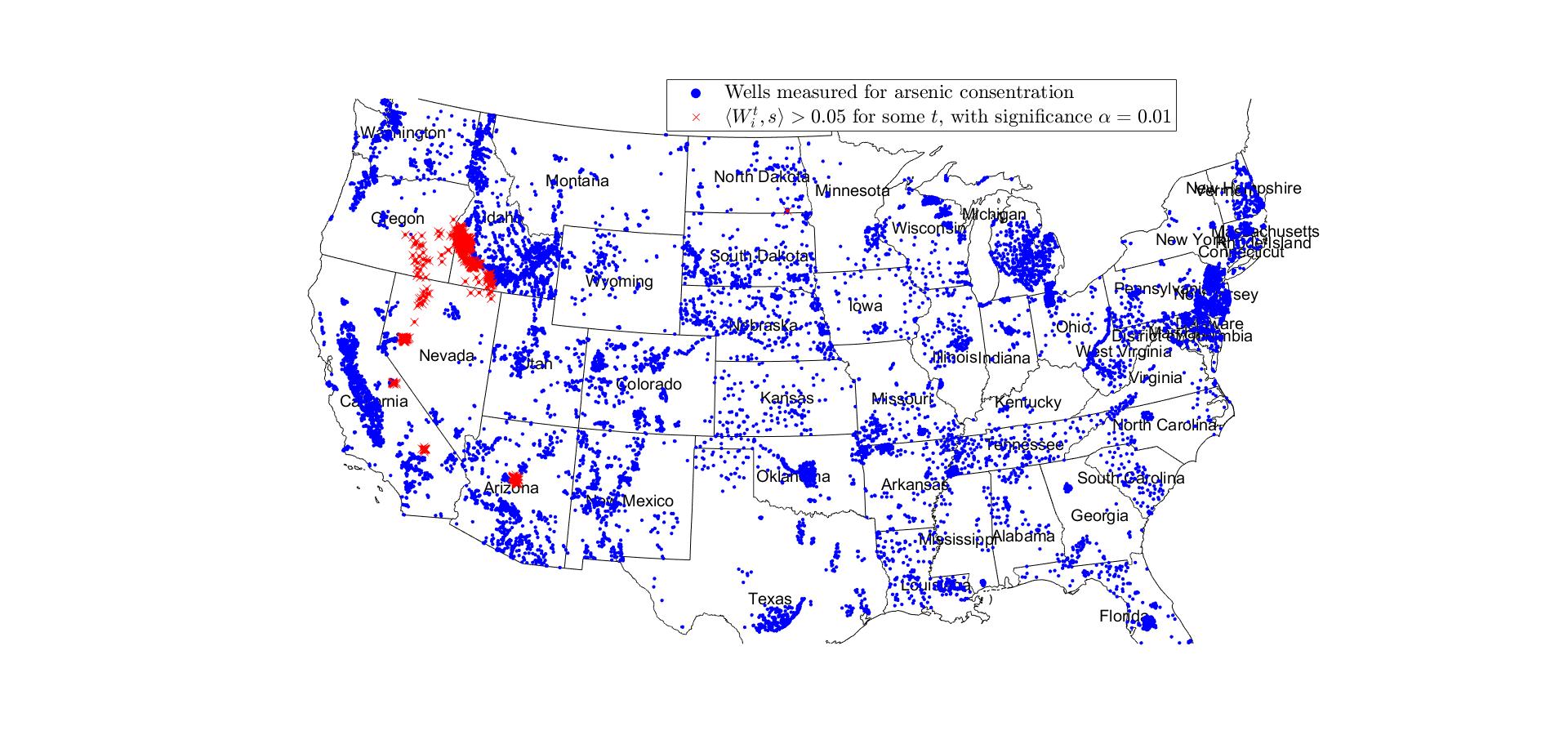}
   \caption{Wells $x_i$ for which it was detected that $\langle W_{i}^t,\mathbf{s}\rangle >0.05$ with significance $\alpha=0.01$ for some walk time $t\in \{t_j^{(i)}\}_{j=1}^{M_i}$, versus all wells in which arsenic level was measured.} \label{fig:arsenic example Ws lower bound 0.05}
\end{center}
\end{figure} 

\subsection{Scientific discovery in single cell RNA sequencing data} \label{sec:scrna-seq example}
In our third example, we applied our method to a published single cell RNA sequencing dataset of immune cells from melanoma patients~\cite{sade2018defining}. 
Single-cell RNA sequencing (scRNA-seq)~\cite{tang2009mrna,macosko2015highly}, is an experimental procedure where large and heterogeneous samples of cells are characterized by their gene signatures. The gene signature of each cell, referred to as a gene expression profile, is a high-dimensional vector in the gene feature-space. 
In many scRNA-seq experiments, cell populations from two distinct conditions/states are compared. For example, such cell populations can be sampled from a patient before or after treatment. A fundamental question then is whether there is a difference between the densities of the cell populations (in the gene signature space) before and after the treatment, and if so, in which cell sub-populations in particular~\cite{sade2018defining}. 
Such cell sub-populations can be further investigated for unique functionality or characteristics, thereby providing novel scientific insights. 

In the study carried out by Sade-Feldman et al.~\cite{sade2018defining}, they collected $16,291$ cells from multiple melanoma patients that were treated with immunotherapy, where the expression levels of $55,737$ genes were quantified for each cell. Therefore, the data is represented as matrix of size $55,737\times 16,291$, whose $(i,j)$'th entry represents the expression level of gene $i$ in cell $j$. The authors of~\cite{sade2018defining} identified $11$ distinct cell-types in this data, which was based on whether specific genes of interest were highly expressed. In addition, the patients that showed positive response to the therapy in the post-treatment assessment were labelled as \textit{responders}, while the others were labelled as \textit{non-responders}. 
Thus, the cells from all patients were pooled into two sets - a set of cells from all responders ($5,564$ cells in total) with each one labelled as “R”, and a set of cells from all non-responders ($10,727$ cells in total), with each one labelled as “NR”. 

We downloaded the preprocessed data matrix from the NCBI Gene Expression Omnibus~\cite{barrett2012ncbi} and used the R package Seurat~\cite{stuart2019comprehensive} to process the data. Following the preprocessing pipeline described in~\cite{sade2018defining}, we generated a low-rank representation of the data (a matrix of size $30 \times 16,291$). 
In Figure~\ref{fig:scRNA umap cell} we visualize the resulting representation of the cells using UMAP~\cite{mcinnes2018umap}, which is a popular dimensionality reduction technique for scRNA-seq data~\cite{becht2019dimensionality,chua2020covid}, and color the cells according to their labels (Figure~\ref{fig:scRNA umap cell sources}) and types (Figure~\ref{fig:scRNA umap cell types}).
For simplicity, we retained the same notation for the cell types as in~\cite{sade2018defining} by indexing $11$ cell types from G1 to G11. Sade-Feldman et al.~\cite{sade2018defining} examined for which cell type there is a significant discrepancy between the frequencies of “R” cells and “NR” cells. They found two cell types (G1, G10) in which the frequency of “R" cells is larger than of “NR" cells. Additionally, they found four cell types (G3, G4, G6, G11) in which the frequency of “NR" cells is larger than of ``R'' cells. 
Evidently, in the approach by~\cite{sade2018defining}, the comparison between the densities of ``R'' and ``NR'' cells is conducted only at the cell type level.
Distinctly, our goal here is to identify specific regions within the cell types where the density of the ``R'' cells is significantly larger than of the ``NR'' cells, or vice versa.

 Based on the low-rank representation of the data as described above, we first calculated pairwise Euclidean distance among all cells and then formed the affinity matrix $K$ based on the Gaussian kernel (see equation~\eqref{eq:unit circle example Gaussian kernel}), selecting the bandwidth $\sigma$ to be equal to the $0.01\%$ quantile of the distribution of all pairwise distances. The main diagonal of $K$ was zeroed out (as suggested in~\cite{landa2020doubly} for improved robustness to heteroskedastic noise) and $W$ was constructed as described in Section~\ref{sec:W construction}. 
 
 We initially set out to find neighborhoods of the sample where the density of ``NR'' cells is larger than that of ``R'' cells. Towards that end, we labelled the “NR" cells by $1$, i.e., $z_i=1$ for the “NR" cells and $z_i=0$ for the ``R'' cells. Then we applied Algorithm~\ref{alg:local two-sample testing by aRWSS} to the matrix $W$ and the labels $z_1,\ldots,z_n$, setting the significance level at $\alpha=0.01$. An upper bound on the prior $p$ was estimated to be around $0.67$ using the Clopper-Pearson method (see step~\ref{step:Clopper-Pearson} in Algorithm~\ref{alg:local two-sample testing by aRWSS}). Finally, we found $328,260$ distributions that rejected the null hypothesis, which were generated by random-walks starting from $6,261$ cells. We labelled the $4,553$ cells that satisfy $\langle W_i^t, \mathbf{s}\rangle >0.05$ for at least one value of walk time $t\in \{t_j^{(i)}\}_{j=1}^{M_i}$ as “NR-enriched”. 
 
 Next, we turned to find neighborhoods of the sample where the density of ``R'' cells is larger than that of the ``NR''. Towards that end, we repeated the same procedure as before but assigned labels of $1$ to the “R" cells, i.e., $z_i=1$ for the “R" cells and $z_i=0$ for the ``NR'' cells.
 This resulted in $85,240$ distributions corresponding to $5,986$ cells passing the test (with an upper bound on the prior $p$ around $0.35$). 
 Similarly, we labelled the $3,390$ cells with $\langle W_i^t, \mathbf{s}\rangle >0.05$ for at least one value of walk time $t\in \{t_j^{(i)}\}_{j=1}^{M_i}$ as “R-enriched”. The two groups of cells “R-enriched" and “NR-enriched" are highlighted in Figure~\ref{fig:scRNA umap enriched labels}. 
 
 Interestingly, within each of the four cell types G3, G4, G6, G11 that were found in~\cite{sade2018defining} to have larger frequency of “NR" cells, our method also detected a corresponding subset of “NR-enriched” cells. Furthermore, the detected ``NR-enriched'' cells form contiguous structures within the original cell types, perhaps indicative of special cell sub-types. 
 In addition, we also identified roughly $24.42\%$ of the G9 cells (Exhausted/HS CD8+ T-cells) as “NR-enriched”, whereas G9 was not reported to have a significant difference in frequencies of “R" and “NR" cells in~\cite{sade2018defining}. This new finding suggests that a subset of the G9 cells is more likely to be found in non-responders (i.e., patients who failed to respond to immunotherapy), and reveals the advantage of our approach over those that are based on predefined cell types.
 
\begin{figure}
\begin{subfigure}{0.45\textwidth}
\includegraphics[width=\linewidth]{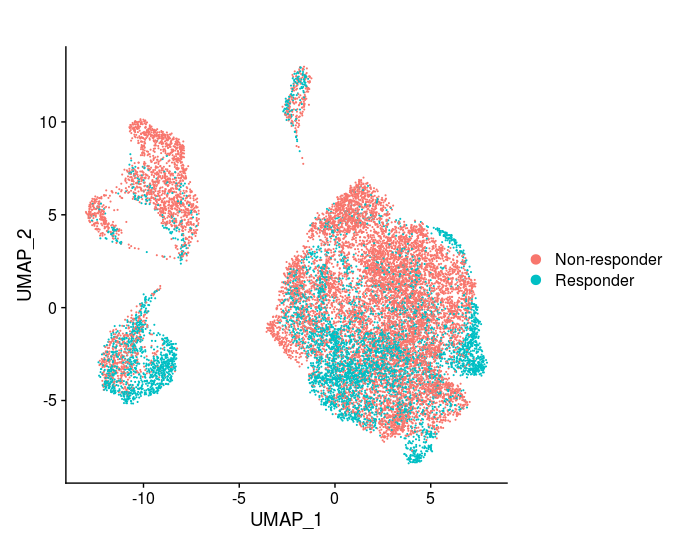}
\caption{UMAP colored by labels}
 \label{fig:scRNA umap cell sources}
  \end{subfigure}\hspace*{\fill}
\begin{subfigure}{0.56\textwidth}
\includegraphics[width=\linewidth]{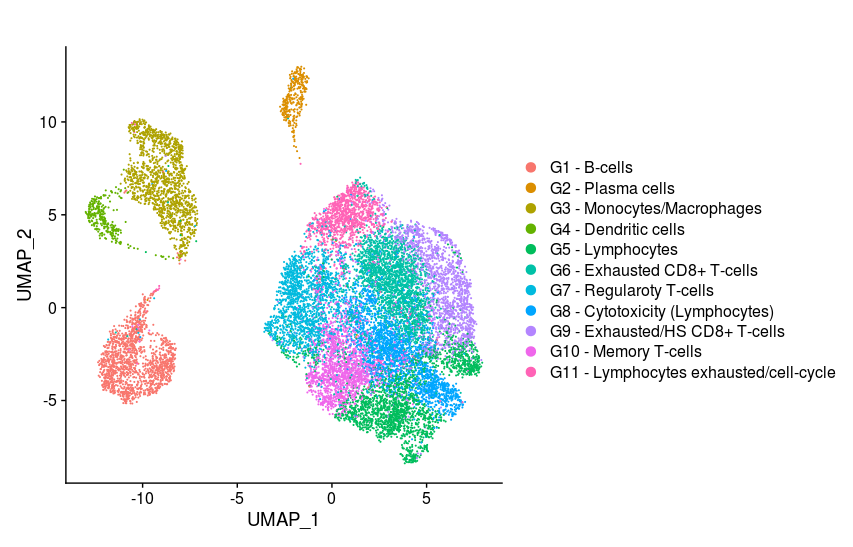}
\caption{UMAP colored by cell types}
 \label{fig:scRNA umap cell types}
\end{subfigure}
   \caption{UMAP embedding of cells. Cells are colored according to their labels of “R" (responders to immunotherapy) or “NR" (non-responders to immunotherapy), and by the pre-annotated cell types from~\cite{sade2018defining}}
    \label{fig:scRNA umap cell}
\end{figure}

\begin{figure}
\begin{center}
    \includegraphics[width=0.8\linewidth]{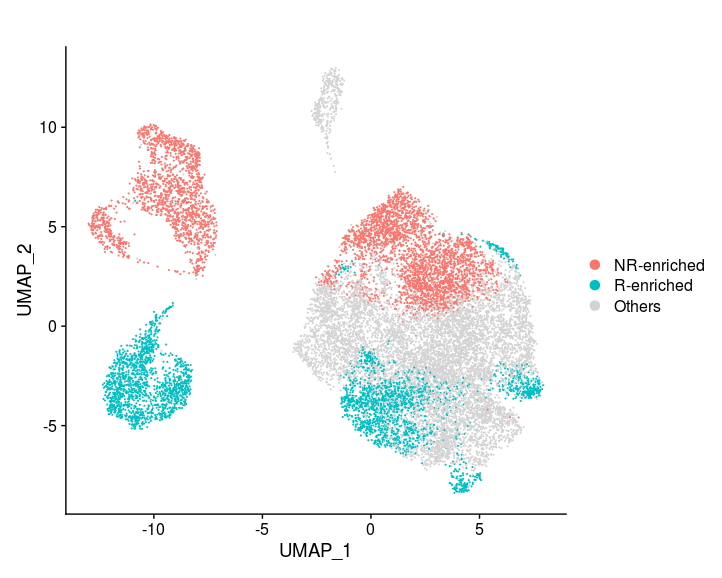}
   \caption{Cells $x_i$ for which it was detected that $\langle W_i^t, \mathbf{s}\rangle >0.05$ for at least one value of walk time $t\in \{t_j^{(i)}\}_{j=1}^{M_i}$ using significance level $\alpha = 0.01$. Red points are labelled as “NR-enriched", and correspond to the setting where $z_i = 1$ for “NR" cells, and $z_i = 0$ for “R" cells. On the other hand, blue points are labeled as “R-enriched", and correspond to the setting where $z_i = 1$ for “R" cells, and $z_i = 0$ for “NR" cells.}
   \label{fig:scRNA umap enriched labels}
\end{center}
\end{figure}

\section{Acknowledgements}
We would like to thank Stefan Steinerberger, Ronald Coifman, Boaz Nadler, Vladimir Rokhlin, Sahand Negahban, and Dan Kluger for useful discussions and suggestions. This work was supported by the National Institutes of Health [R01GM131642, UM1DA051410, P50CA121974 and R61DA047037].

\begin{appendices}

\section{Implementation details and computational complexity} \label{appendix:implementation details}
We now discuss the implementation of the different steps of Algorithm~\ref{alg:local two-sample testing by aRWSS}, and analyze their computational complexities.
We begin with step~\ref{step:constructing W}. Evaluating $\widetilde{K}$ directly according to~\eqref{eq:doubly stochastic normalization} requires $\mathcal{O}(n^2)$ operations. Computing $\widetilde{W}$ using the Sinkhorn-Knopp algorithm~\cite{sinkhorn1967concerning} costs $\mathcal{O}(n^2)$ operations for each iteration, and requires $\mathcal{O}(1/\log(\lambda_{<1}^{-1}))$ iterations if $\widetilde{K}$ is fully indecomposable~\cite{knight2008sinkhorn}. Hence, the overall computational complexity of evaluating $\widetilde{W}$ is $\mathcal{O}(n^2/\log(\lambda_{<1}^{-1})) = \mathcal{O}(n^2/(1-\lambda_{<1}))$. Then, computing $W$ according to~\eqref{eq:PSD normalization} by direct squaring requires $\mathcal{O}(n^3)$ operations. Next, in step~\ref{step:connected components}, the connected components of $W$ can be obtained using standard Breadth-first or Depth-first searches (BFS or DFS), with computational complexity of $\mathcal{O}(n^2)$. The computational complexity of computing the eigen-decompositions of $\{W^{(\ell)}\}_{\ell=1}^L$ in step~\ref{step:eigen-decomposition of W} is $\sum_{\ell=1}^L \mathcal{O}(n_\ell^3) = \mathcal{O}(n^3)$.
Continuing, the minimization of~\eqref{eq:beta upper bound} in step~\ref{step:minimzing beta} can be approximated by a grid search over $\varepsilon\in (0,1)$, with a resulting computational complexity of $\mathcal{O}(1)$ (since this minimization is independent of $n$). 

We proceed by analyzing the computational complexity of Algorithm~\ref{alg:evaluating diffusion time steps} (appearing in step~\ref{step:compute time steps} of Algorithm~\ref{alg:local two-sample testing by aRWSS}). Using the orthogonality of the eigenvectors of $W$, we can write
\begin{equation}
    \Vert W_i^t \Vert_2^2 = \sum_{k = 1}^{n_\ell} (\lambda_{k}^{(\ell)})^{2t} (\psi_k^{(\ell)}[i])^2, \label{eq:W_i^t norm computational complexity}
\end{equation}
which costs $\mathcal{O}(n_\ell)$ operations to compute for each $i\in\mathcal{C}_\ell$ and $t$ (given the eigen-decomposition of $W^{(\ell)}$). Since $\Vert W_i^t \Vert_2^2$ is monotonically decreasing in $t$ (see Proposition~\ref{prop:random walk distribution properties}), step~\ref{step:find t_j_i} in Algorithm~\ref{alg:evaluating diffusion time steps} can be implemented using the bisection method. Therefore, step~\ref{step:find t_j_i} in Algorithm~\ref{alg:evaluating diffusion time steps} require $ \mathcal{O}(n_\ell \log (T_\ell))$ operations (using the fact that $t_j^{(i)} \leq T_\ell$). Since the inner while loop in Algorithm~\ref{alg:evaluating diffusion time steps} runs for $M_i-1$ iterations, the computational complexity of Algorithm~\ref{alg:evaluating diffusion time steps} is
\begin{equation}
    \mathcal{O}\left(\sum_{\ell=1}^L \sum_{i\in\mathcal{C}_\ell} M_i n_\ell \log (T_\ell)\right) = \mathcal{O}\left(n^2 \cdot \max_{i=1\ldots,n} M_i \cdot \max_{\ell=1,\ldots,L} \log (T_\ell)\right), \label{eq:algorithm1 computational complexity}
\end{equation}
where we used the fact that $\sum_{\ell=1}^L {n_\ell} = n$. Therefore, if $\lambda_{<1}$ is bounded away from $1$ as $n\rightarrow\infty$, then the computational complexity of Algorithm~\ref{alg:evaluating diffusion time steps} is essentially $\mathcal{O}(n^2 \cdot \log n \cdot \log \log n)$, which is smaller than $\mathcal{O}(n^3)$. Otherwise, the computational complexity of Algorithm~\ref{alg:evaluating diffusion time steps} depends on the convergence rate of $\lambda_{<1}$ to $1$, and can be determined on a case-by-case basis using Lemma~\ref{lem:t_1 and M_i bounds}. If we assume that $1-\lambda_{<1} \underset{n\rightarrow\infty}{\sim} n^{-r}$ for some $r>0$, then the computational complexity of Algorithm~\ref{alg:evaluating diffusion time steps} is $\mathcal{O}(n^3 \cdot (\log n)^2)$.
Next, step~\ref{step:Clopper-Pearson} in Algorithm~\ref{alg:local two-sample testing by aRWSS} requires $\mathcal{O}(n)$ operations to compute $\sum_{i=1}^n z_i$, and $\mathcal{O}(1)$ operations for the Clopper-Pearson method. 

Last, we put our focus on analyzing the computational complexity of step~\ref{step:compute alpha_hat} of Algorithm~\ref{alg:local two-sample testing by aRWSS}. The computationally-dominant part of this step is evaluating
\begin{equation}
    \frac{\langle W_{i}^{t_j^{(i)}}, \mathbf{z} - p \rangle}{\Vert  W_{i}^{t_j^{(i)}} \Vert_2}, \label{eq:statistics computational complexity}
\end{equation}
for all $i=1,\ldots,n$ and $j=1,\ldots,M_i$. Note that Algorithm~\ref{alg:evaluating diffusion time steps} already evaluates the quantities $\Vert  W_{i}^{t_j^{(i)}} \Vert_2$ for all $i$ and $j$ (using~\eqref{eq:W_i^t norm computational complexity}), hence we only need to compute the numerator of~\eqref{eq:statistics computational complexity}. Similarly to the computation of $\Vert  W_{i}^{t_j^{(i)}} \Vert_2$ in~\eqref{eq:W_i^t norm computational complexity}, this can be accomplished efficiently using the eigen-decomposition of $W$, as we can write
\begin{equation}
    \langle W_{i}^{t_j^{(i)}}, \mathbf{z} - p \rangle = \sum_{k=1}^{n_\ell} \psi_k^{(\ell)}[i] (\lambda_k^{(\ell)})^{t_j^{(i)}} \langle \psi_k^{(\ell)}, \mathbf{z}-p\rangle, \label{eq:statistic numerator expansion}
\end{equation}
for any $i\in\mathcal{C}_{\ell}$ and $j\in \{1,\ldots,M_i\}$. Consequently, we first compute $ \langle \psi_k^{(\ell)}, \mathbf{z}-p\rangle$ for all $k=1,\ldots,n_\ell$ and $\ell=1,\ldots,L$, which requires $\mathcal{O}(\sum_{\ell=1}^L  n_\ell^2) = \mathcal{O}(n^2)$ operations. Then, we compute $\{\langle W_{i}^{t_j^{(i)}}, \mathbf{z} - p \rangle\}_{j=1}^{M_i}$ using~\eqref{eq:statistic numerator expansion}, which requires $\mathcal{O}(n_\ell M_i)$ operations for each $i\in\mathcal{C}_\ell$. Overall, the computational complexity of step~\ref{step:Clopper-Pearson} in Algorithm~\ref{alg:local two-sample testing by aRWSS} is therefore
\begin{equation}
    \mathcal{O}(n^2) + \mathcal{O}(\sum_{\ell=1}^L \sum_{i\in\mathcal{C}_\ell} n_\ell M_i) = \mathcal{O}(n^2 \cdot \max_{i=1,\ldots,n} M_i),
\end{equation}
which is clearly lesser than the computational complexity of step~\ref{step:compute time steps}, given by~\eqref{eq:algorithm1 computational complexity}. 

Overall, if $1 - \lambda_{<1} = \mathcal{O}(1/n)$ then the computational complexity of Algorithm~\ref{alg:local two-sample testing by aRWSS} is dominated by the Sinkhorn-Knopp algorithm for computing $\widetilde{W}$, which is $\mathcal{O}(n^2/(1-\lambda_{<1}))$. Otherwise, the computational complexity of Algorithm~\ref{alg:local two-sample testing by aRWSS} is dominated by Algorithm~\ref{alg:evaluating diffusion time steps}, whose computational complexity in this case is at most $\mathcal{O}(n^3 (\log n)^2)$. Therefore, the computational complexity of Algorithm~\ref{alg:local two-sample testing by aRWSS} is $\mathcal{O}(n^3 (\log n)^2)$ if $1 - \lambda_{<1} = \Omega(1/n)$.

\section{Tables for $\varepsilon$ and $\hat{h}_{n,\alpha,\lambda_{<1}}(\varepsilon)$} \label{appendix:tables}
\begin{table}[H]
\begin{tabular}{|l|l|llllllllll}
\cline{1-11}
                          $\mathbf{\alpha}$   &           \backslashbox{$\mathbf{n}$}{$\mathbf{1-\lambda_{<1}}$}  &  \multicolumn{1}{l|}{\textbf{$\mathbf{0.1}$}} & \multicolumn{1}{l|}{\textbf{$\mathbf{0.01}$}} & \multicolumn{1}{l|}{\textbf{$\mathbf{10^{-3}}$}} & \multicolumn{1}{l|}{\textbf{$\mathbf{10^{-4}}$}} & \multicolumn{1}{l|}{\textbf{$\mathbf{10^{-5}}$}} & \multicolumn{1}{l|}{\textbf{$\mathbf{10^{-6}}$}} & \multicolumn{1}{l|}{\textbf{$\mathbf{10^{-9}}$}} & \multicolumn{1}{l|}{\textbf{$\mathbf{10^{-12}}$}} & \multicolumn{1}{l|}{\textbf{$\mathbf{10^{-16}}$}} & \\ \cline{1-11}
\multirow{4}{*}{\textbf{$\mathbf{0.1}$}}     & \textbf{$\mathbf{10^3}$} & \multicolumn{1}{l|}{2.647}          & \multicolumn{1}{l|}{2.860}           & \multicolumn{1}{l|}{3.056}              & \multicolumn{1}{l|}{3.239}              & \multicolumn{1}{l|}{3.413}              & \multicolumn{1}{l|}{3.455}              & \multicolumn{1}{l|}{3.455}              & \multicolumn{1}{l|}{3.455}               & \multicolumn{1}{l|}{3.455}               &                                                 \\ \cline{2-11}
                                    & \textbf{$\mathbf{10^4}$} & \multicolumn{1}{l|}{2.872}          & \multicolumn{1}{l|}{3.070}           & \multicolumn{1}{l|}{3.253}              & \multicolumn{1}{l|}{3.425}              & \multicolumn{1}{l|}{3.590}              & \multicolumn{1}{l|}{3.747}              & \multicolumn{1}{l|}{3.806}              & \multicolumn{1}{l|}{3.806}               & \multicolumn{1}{l|}{3.806}               &                                                 \\ \cline{2-11}
                                    & \textbf{$\mathbf{10^5}$} & \multicolumn{1}{l|}{3.078}          & \multicolumn{1}{l|}{3.264}           & \multicolumn{1}{l|}{3.436}              & \multicolumn{1}{l|}{3.600}              & \multicolumn{1}{l|}{3.757}              & \multicolumn{1}{l|}{3.907}              & \multicolumn{1}{l|}{4.121}              & \multicolumn{1}{l|}{4.121}               & \multicolumn{1}{l|}{4.121}               &                                                 \\ \cline{2-11}
                                    & \textbf{$\mathbf{10^6}$} & \multicolumn{1}{l|}{3.270}          & \multicolumn{1}{l|}{3.445}           & \multicolumn{1}{l|}{3.609}              & \multicolumn{1}{l|}{3.765}              & \multicolumn{1}{l|}{3.915}              & \multicolumn{1}{l|}{4.059}              & \multicolumn{1}{l|}{4.409}              & \multicolumn{1}{l|}{4.409}               & \multicolumn{1}{l|}{4.409}               &                                                 \\ \cline{1-11}
\multirow{4}{*}{\textbf{$\mathbf{0.01}$}}    & \textbf{$\mathbf{10^3}$} & \multicolumn{1}{l|}{2.856}          & \multicolumn{1}{l|}{3.056}           & \multicolumn{1}{l|}{3.239}              & \multicolumn{1}{l|}{3.413}              & \multicolumn{1}{l|}{3.578}              & \multicolumn{1}{l|}{3.625}              & \multicolumn{1}{l|}{3.625}              & \multicolumn{1}{l|}{3.625}               & \multicolumn{1}{l|}{3.625}               &                                                 \\ \cline{2-11}
                                    & \textbf{$\mathbf{10^4}$} & \multicolumn{1}{l|}{3.066}          & \multicolumn{1}{l|}{3.252}           & \multicolumn{1}{l|}{3.425}              & \multicolumn{1}{l|}{3.590}              & \multicolumn{1}{l|}{3.747}              & \multicolumn{1}{l|}{3.898}              & \multicolumn{1}{l|}{3.959}              & \multicolumn{1}{l|}{3.959}               & \multicolumn{1}{l|}{3.959}               &                                                 \\ \cline{2-11}
                                    & \textbf{$\mathbf{10^5}$} & \multicolumn{1}{l|}{3.260}          & \multicolumn{1}{l|}{3.436}           & \multicolumn{1}{l|}{3.600}              & \multicolumn{1}{l|}{3.757}              & \multicolumn{1}{l|}{3.907}              & \multicolumn{1}{l|}{4.052}              & \multicolumn{1}{l|}{4.262}              & \multicolumn{1}{l|}{4.262}               & \multicolumn{1}{l|}{4.262}               &                                                 \\ \cline{2-11}
                                    & \textbf{$\mathbf{10^6}$} & \multicolumn{1}{l|}{3.442}          & \multicolumn{1}{l|}{3.608}           & \multicolumn{1}{l|}{3.765}              & \multicolumn{1}{l|}{3.915}              & \multicolumn{1}{l|}{4.059}              & \multicolumn{1}{l|}{4.199}              & \multicolumn{1}{l|}{4.541}              & \multicolumn{1}{l|}{4.541}               & \multicolumn{1}{l|}{4.541}               &                                                 \\ \cline{1-11}
\multirow{4}{*}{\textbf{$\mathbf{10^{-3}}$}} & \textbf{$\mathbf{10^3}$} & \multicolumn{1}{l|}{3.052}          & \multicolumn{1}{l|}{3.239}           & \multicolumn{1}{l|}{3.413}              & \multicolumn{1}{l|}{3.578}              & \multicolumn{1}{l|}{3.735}              & \multicolumn{1}{l|}{3.786}              & \multicolumn{1}{l|}{3.786}              & \multicolumn{1}{l|}{3.786}               & \multicolumn{1}{l|}{3.786}               &                                                 \\ \cline{2-11}
                                    & \textbf{$\mathbf{10^4}$} & \multicolumn{1}{l|}{3.249}          & \multicolumn{1}{l|}{3.425}           & \multicolumn{1}{l|}{3.590}              & \multicolumn{1}{l|}{3.747}              & \multicolumn{1}{l|}{3.898}              & \multicolumn{1}{l|}{4.043}              & \multicolumn{1}{l|}{4.107}              & \multicolumn{1}{l|}{4.107}               & \multicolumn{1}{l|}{4.107}               &                                                 \\ \cline{2-11}
                                    & \textbf{$\mathbf{10^5}$} & \multicolumn{1}{l|}{3.432}          & \multicolumn{1}{l|}{3.600}           & \multicolumn{1}{l|}{3.757}              & \multicolumn{1}{l|}{3.907}              & \multicolumn{1}{l|}{4.052}              & \multicolumn{1}{l|}{4.192}              & \multicolumn{1}{l|}{4.399}              & \multicolumn{1}{l|}{4.399}               & \multicolumn{1}{l|}{4.399}               &                                                 \\ \cline{2-11}
                                    & \textbf{$\mathbf{10^6}$} & \multicolumn{1}{l|}{3.605}          & \multicolumn{1}{l|}{3.765}           & \multicolumn{1}{l|}{3.915}              & \multicolumn{1}{l|}{4.059}              & \multicolumn{1}{l|}{4.199}              & \multicolumn{1}{l|}{4.334}              & \multicolumn{1}{l|}{4.670}              & \multicolumn{1}{l|}{4.670}               & \multicolumn{1}{l|}{4.670}               &                                                 \\ \cline{1-11}

\end{tabular}
\caption{Values of $\min_\varepsilon \hat{h}_{n,\alpha,\lambda_{<1}}(\varepsilon)$ (recall that $\hat{h}_{n,\alpha,\lambda_{<1}}(\varepsilon) \geq h(\varepsilon)$), where $ \hat{h}_{n,\alpha,\lambda_{<1}}(\varepsilon)$ is from~\eqref{eq:beta upper bound}.} \label{table:values of beta_hat}
\end{table}

\begin{table}[H]
\begin{tabular}{|l|l|l|l|l|l|l|l|l|l|l|}
\hline
$\mathbf{\alpha}$                        & \backslashbox{$\mathbf{n}$}{$\mathbf{1-\lambda_{<1}}$}       & $\mathbf{0.1}$ & $\mathbf{0.01}$ & $\mathbf{10^{-3}}$ & $\mathbf{10^{-4}}$ & $\mathbf{10^{-5}}$ & $\mathbf{10^{-6}}$ & $\mathbf{10^{-9}}$ & $\mathbf{10^{-12}}$ & $\mathbf{10^{-16}}$ \\ \hline
\multirow{4}{*}{$\mathbf{0.1}$}     & $\mathbf{10^3}$ & 0.0083         & 0.0073          & 0.0069             & 0.0065             & 0.0061             & 0.1514             & 0.1514             & 0.1514              & 0.1514              \\ \cline{2-11} 
                                    & $\mathbf{10^4}$ & 0.0060         & 0.0056          & 0.0053             & 0.0051             & 0.0048             & 0.0046             & 0.1362             & 0.1362              & 0.1362              \\ \cline{2-11} 
                                    & $\mathbf{10^5}$ & 0.0048         & 0.0045          & 0.0043             & 0.0041             & 0.0039             & 0.0038             & 0.1252             & 0.1252              & 0.1252              \\ \cline{2-11} 
                                    & $\mathbf{10^6}$ & 0.0038         & 0.0037          & 0.0036             & 0.0034             & 0.0033             & 0.0032             & 0.1165             & 0.1165              & 0.1165              \\ \hline
\multirow{4}{*}{$\mathbf{0.01}$}    & $\mathbf{10^3}$ & 0.0075         & 0.0067          & 0.0065             & 0.0061             & 0.0058             & 0.1437             & 0.1437             & 0.1437              & 0.1437              \\ \cline{2-11} 
                                    & $\mathbf{10^4}$ & 0.0054         & 0.0053          & 0.0052             & 0.0048             & 0.0046             & 0.0044             & 0.1306             & 0.1306              & 0.1306              \\ \cline{2-11} 
                                    & $\mathbf{10^5}$ & 0.0043         & 0.0043          & 0.0041             & 0.0039             & 0.0038             & 0.0036             & 0.1207             & 0.1207              & 0.1207              \\ \cline{2-11} 
                                    & $\mathbf{10^6}$ & 0.0038         & 0.0036          & 0.0034             & 0.0033             & 0.0032             & 0.0030             & 0.1129             & 0.1129              & 0.1129              \\ \hline
\multirow{4}{*}{$\mathbf{10^{-3}}$} & $\mathbf{10^3}$ & 0.0068         & 0.0067          & 0.0061             & 0.0058             & 0.0055             & 0.1370             & 0.1370             & 0.1370              & 0.1370              \\ \cline{2-11} 
                                    & $\mathbf{10^4}$ & 0.0054         & 0.0050          & 0.0048             & 0.0046             & 0.0044             & 0.0042             & 0.1256             & 0.1256              & 0.1256              \\ \cline{2-11} 
                                    & $\mathbf{10^5}$ & 0.0043         & 0.0041          & 0.0039             & 0.0038             & 0.0036             & 0.0035             & 0.1168             & 0.1168              & 0.1168              \\ \cline{2-11} 
                                    & $\mathbf{10^6}$ & 0.0034         & 0.0033          & 0.0032             & 0.0032             & 0.0030             & 0.0029             & 0.1097             & 0.1097              & 0.1097              \\ \hline
\end{tabular}
\caption{Values of $\varepsilon$ that minimize $\hat{h}_{n,\alpha,\lambda_{<1}}(\varepsilon)$ from~\eqref{eq:beta upper bound}.} \label{table:values of vareps}
\end{table}

\section{Proof of Lemma~\ref{lem:relation between local and global risks}} \label{appendix:relation between local and global risks}
First, we note that $Q_\mathbf{z} = 1$ if and only if $\hat{\mathcal{G}}_{\mathbf{z}} \cap \mathcal{F} \neq \emptyset$. Therefore, we can write
\begin{equation}
    \sup_{f_1,f_0\in H_0}\operatorname{Pr}\{Q_{\mathbf{z}}=1  \mid  f_0, f_1 \} 
    = \sup_{f_1,f_0\in H_0(\mathcal{F})}\operatorname{Pr} \{ \hat{\mathcal{G}}_{\mathbf{z}} \cap \mathcal{F} \neq \emptyset   \mid  f_0, f_1 \}
    \leq \sup_{\mathcal{G} \subseteq \mathcal{F}}\sup_{f_1,f_0\in H_0(\mathcal{G})}\operatorname{Pr} \{ \hat{\mathcal{G}}_{\mathbf{z}} \cap \mathcal{G} \neq \emptyset   \mid  f_0, f_1 \}. \label{eq:type I error inequality between global and local}
\end{equation}
Second, recall that by definition $\inf_{\hat{\mathbf{w}} \in \hat{\mathcal{G}}_{\mathbf{z}}} \mathcal{E}_{\operatorname{TV}} (\hat{\mathbf{w}}, \mathbf{w}) = 1$ if $\hat{\mathcal{G}}_{\mathbf{z}}$ is an empty set, which holds if  $Q_\mathbf{z} = 0$. Hence,
\begin{align}
    &\sup_{\mathbf{w}\in\mathcal{F}} \sup_{f_1,f_0\in H_1(\mathbf{w},\gamma)}  \mathbb{E}[ \inf_{\hat{\mathbf{w}} \in \hat{\mathcal{G}}_{\mathbf{z}}} \mathcal{E}_{\operatorname{TV}} (\hat{\mathbf{w}}, \mathbf{w})   \mid  f_0,f_1 ] 
    \geq \sup_{\mathbf{w}\in\mathcal{F}} \sup_{f_1,f_0\in H_1(\mathbf{w},\gamma)}  \operatorname{Pr}\{ \inf_{\hat{\mathbf{w}} \in \hat{\mathcal{G}}_{\mathbf{z}}} \mathcal{E}_{\operatorname{TV}} (\hat{\mathbf{w}}, \mathbf{w}) = 1  \mid  f_0,f_1 \} \cdot 1 \nonumber\\
    & \geq \sup_{\mathbf{w}\in\mathcal{F}} \sup_{f_1,f_0\in H_1(\mathbf{w},\gamma)}  \operatorname{Pr}\{ Q_\mathbf{z} = 0  \mid  f_0,f_1 \}.
\end{align}

\section{Proof of Proposition~\ref{prop:random walk distribution properties}} \label{appendix:proof of proposition on random-walk distribution properties}
 We start with the eigen-decomposition of $W^t$, which is given by
 \begin{equation}
     W_{i,j}^t = \sum_{k=1}^{n_\ell} \psi_k^{(\ell)}[i] (\lambda_{k}^{(\ell)})^t \psi_k^{(\ell)}[j], \label{eq:eigen-decomposition of W^t}
 \end{equation}
 and using the orthonormality of the eigenvectors of $W$ we have
  \begin{equation}
     \Vert W_i^t \Vert_2^2 = \sum_{k=1}^{n_\ell} (\psi_k^{(\ell)}[i])^2 (\lambda_k^{(\ell)})^{2t}. \label{eq:W_i_t norm expression in proof of proposition}
 \end{equation}
 Property~\ref{prop:random-walk distribution properties 2} in Proposition~\ref{prop:random walk distribution properties} follows directly from~\eqref{eq:W_i_t norm expression in proof of proposition},
 which is monotonically decreasing in $t$ since $0 \leq \lambda_k^{(\ell)} \leq 1$ for all $k$ and $\ell$. 
 To prove property~\ref{prop:random-walk distribution properties 3}, we use~\eqref{eq:eigen-decomposition of W^t}, the orthonormality of the eigenvectors, and fact that $(\lambda_k^{(\ell)})^{t} \geq (\lambda_k^{(\ell)})^{\tau}$ for $t\leq \tau$, which gives
 \begin{align}
     \Vert W_i^{t} - W_i^{\tau} \Vert^2_2 &= \sum_{k=1}^{n_\ell} (\psi_k^{(\ell)}[i])^2 \left((\lambda_k^{(\ell)})^{t} - (\lambda_k^{(\ell)})^{\tau} \right)^{2}  = \sum_{k=1}^{n_\ell} (\psi_k^{(\ell)}[i])^2 \left((\lambda_k^{(\ell)})^{2t} + (\lambda_k^{(\ell)})^{2\tau} - 2(\lambda_k^{(\ell)})^{t} (\lambda_k^{(\ell)})^{\tau}\right) \nonumber \\
     &\leq \sum_{k=1}^{n_\ell} (\psi_k^{(\ell)}[i])^2 \left((\lambda_k^{(\ell)})^{2t} - (\lambda_k^{(\ell)})^{2\tau}\right) = \Vert W_i^{t} \Vert_2^2 - \Vert W_i^{\tau} \Vert^2_2 
     \leq \Vert W_i^{t^{'}} \Vert_2^2 - \Vert W_i^{\tau^{'}} \Vert^2_2,
 \end{align}
where we also used property~\ref{prop:random-walk distribution properties 2} from Proposition~\ref{prop:random walk distribution properties} in the last inequality.
The first part of property~\ref{prop:random-walk distribution properties 1} follows immediately from the eigen-decomposition of $W^t$ and property~\ref{prop:spctral properties 2} from Proposition~\ref{prop:W spectral properties}. To prove the second part of property~\ref{prop:random-walk distribution properties 1}, we use~\eqref{eq:W_i_t norm expression in proof of proposition}
and Proposition~\ref{prop:W spectral properties} to write
 \begin{align}
     \frac{1}{n_\ell} = (\psi_1^{(\ell)}[i])^2 = (\psi_1^{(\ell)}[i])^2 (\lambda_1^{(\ell)})^{2t} \leq   \sum_{k=1}^{n_\ell} (\psi_k^{(\ell)}[i])^2 (\lambda_k^{(\ell)})^{2t} 
     &\leq (\psi_1^{(\ell)}[i])^2 (\lambda_1^{(\ell)})^{2t} + (\lambda_2^{(\ell)})^{2t} \sum_{k=2}^{n_\ell} (\psi_k^{(\ell)}[i])^2 \nonumber \\
     &\leq \frac{1}{n_\ell} + (\lambda_2^{(\ell)})^{2t}.
 \end{align}

\section{Proof of Lemma~\ref{lem:specific w bound}} \label{appendix:proof of specific w bound}
\begin{proof}
We first state Hoeffding's inequality for sums of bounded and independent random variables.
\begin{thm}[Hoeffding~\cite{hoeffding1994probability}]
Let $u_1,\ldots,u_n$ be independent random variables satisfying $u_i\in [a_i,b_i]$. Then,
\begin{align}
    &\operatorname{Pr}\left\{ \sum_{i=1}^n u_i - \mathbb{E}\left[\sum_{i=1}^n u_i\right] > t\right\} \leq \operatorname{exp}\left(-\frac{2 t^2}{\sum_{i=1}^n (b_i-a_i)^2} \right),  \label{eq:Hoeffding bound} \\
    &\operatorname{Pr}\left\{ \sum_{i=1}^n u_i - \mathbb{E}\left[\sum_{i=1}^n u_i\right] < -t\right\} \leq \operatorname{exp}\left(-\frac{2 t^2}{\sum_{i=1}^n (b_i-a_i)^2} \right). \label{eq:reverse Hoeffding bound}
\end{align}
\end{thm}
Assuming all quantities are conditioned on $x_1,\ldots,x_n$, using~\eqref{eq:expected value of y_i is s(x_i)} we have
\begin{equation}
    \mathbb{E}\left[\sum_{i=1}^n w_i y_i \right] = \sum_{i=1}^n w_i \mathbb{E}[y_i] = \sum_{i=1}^n w_i s_i.
\end{equation}
Defining $u_i = w_i y_i$, we have that $u_i\in [-p w_i, (1-p) w_i]$, and Hoeffding's inequality asserts that
\begin{align}
    \operatorname{Pr}\left\{ \sum_{i=1}^n w_i y_i - \sum_{i=1}^n w_i s_i > t \right\} 
    &= \operatorname{Pr}\left\{ \sum_{i=1}^n w_i y_i - \mathbb{E}\left[\sum_{i=1}^n w_i y_i \right] > t \right\} 
    \leq \operatorname{exp}\left(-\frac{2 t^2}{\sum_{i=1}^n w_i^2} \right).
\end{align}
Taking $ t = \Vert \mathbf{w}\Vert_2 \sqrt{0.5 \log (1/\alpha)}$ gives 
\begin{equation}
    \operatorname{Pr}\left\{ \sum_{i=1}^n w_i y_i - \sum_{i=1}^n w_i s_i >  \Vert \mathbf{w}\Vert_2 \sqrt{0.5 \log (1/\alpha)} \right\} \leq \alpha,
\end{equation}
which establishes~\eqref{eq:specific w upper bound}.
The proof of~\eqref{eq:specific w lower bound} is analogous to the proof of~\eqref{eq:specific w upper bound} when using~\eqref{eq:reverse Hoeffding bound} instead of~\eqref{eq:Hoeffding bound}, and we omit it for the sake of brevity.  
\end{proof}

\section{Proof of Lemma~\ref{lem:epsilon net}} \label{appendix:proof of epsilon net lemma}
\subsection{Proof of~\eqref{eq:total variation dist bound eps-net}}
According to the definition of $T_\ell$ from~\eqref{eq:T def}, and applying property~\ref{prop:random-walk distribution properties 1} in Proposition~\ref{prop:random walk distribution properties}, we have
\begin{align}
     \Vert W_i^{T_\ell} \Vert^2_2
     &\leq \frac{1}{n_\ell} + (\lambda_{2}^{(\ell)})^{2T_\ell} 
     \leq \frac{1}{n_\ell} + \operatorname{exp} \left( \log (\lambda_{2}^{(\ell)}) \frac{2\log(n_\ell/\varepsilon)}{\log((\lambda_{2}^{(\ell)})^{-1})} \right) = \frac{1}{n_\ell} +  \frac{\varepsilon^2}{n_\ell^2}. \label{eq:W_i_Tl bound}
\end{align}
Consequently, by properties~\ref{prop:random-walk distribution properties 3} and~\ref{prop:random-walk distribution properties 1} in Proposition~\ref{prop:random walk distribution properties}, we have for all $i\in\mathcal{C}_\ell$ and $t\geq T_\ell$ 
\begin{equation}
    \mathcal{E}_{\operatorname{TV}}(W_i^t,W_i^{\pi(t)}) 
    = \frac{1}{2}\Vert {W_i^t} - {W_i^{T_\ell}}\Vert_1 
    \leq \frac{\sqrt{n_\ell}}{2} \Vert {W_i^t} - {W_i^{T_\ell}}\Vert_2 
    \leq \frac{1}{2}\sqrt{n_\ell \left(\Vert {W_i^{T_\ell}} \Vert_2^2 - \Vert {W_i^t}\Vert_2^2 \right) } 
    \leq \frac{\varepsilon}{2\sqrt{n_\ell}} \leq \frac{\varepsilon \Vert W_i^t \Vert_2}{2}, \label{eq:T_l distance from stationary distribution}
\end{equation}
where we used~\eqref{eq:W_i_Tl bound} and the fact that $ \Vert {W_i^t}\Vert_2^2 \geq 1/n_\ell$ for $i\in \mathcal{C}_\ell$ in the last two inequalities.
Last, if $t_{j-1}^{(i)} < t \leq t_{j}^{(i)} < T_\ell$  for some $j$ and $\ell$, then according to~\eqref{eq:choosing t_j iterative inequality} and properties~\ref{prop:random-walk distribution properties 3} and~\ref{prop:random-walk distribution properties 1} in Proposition~\ref{prop:random walk distribution properties}, we can write 
\begin{align}
    \mathcal{E}_{\operatorname{TV}}(W_i^t,W_i^{\pi(t)}) 
    &= \frac{1}{2} \Vert W_i^t - W_i^{t_{j}^{(i)}} \Vert_1 
    \leq \frac{\sqrt{n_\ell}}{2} \Vert W_i^t - W_i^{t_{j+1}^{(i)}} \Vert_2 
    \leq \frac{1}{2}\sqrt{n_\ell \left( \Vert W_i^{t_{j-1}^{(i)}+1} \Vert^2_2 - \Vert W_i^{t_{j}^{(i)}} \Vert^2_2  \right)} \nonumber \\
    &\leq \frac{\varepsilon}{2} \Vert W_i^{t_{j}^{(i)}} \Vert_2 \leq \frac{\varepsilon}{2} \Vert W_i^{t} \Vert_2.
\end{align}

\subsection{Proof of~\eqref{eq:eps-net thm bound}}
Recall that for $i\in\mathcal{C}_\ell$ we have $W_{i,j}^t = 0$ for all $j\notin \mathcal{C}_\ell$ and positive integers $t$ (see Proposition~\ref{prop:random walk distribution properties}). Therefore, using the fact that $-1 \leq 1-p \leq y_i \leq p \leq 1$ together with the Cauchy-Shwarz inequality, we can write
\begin{align}
    &\left\vert \mathcal{S}(W_i^t) - \mathcal{S}(W_i^{\pi(t)})\right\vert = \left\vert \frac{\langle W_i^t , \mathbf{y}\rangle}{\Vert W_i^t \Vert_2} - \frac{\langle W_i^{\pi(t)} , \mathbf{y}\rangle}{\Vert W_i^{\pi(t)} \Vert_2} \right\vert 
    = \left\vert \left\langle \frac{ W_i^t }{\Vert W_i^t \Vert_2} - \frac{ W_i^{\pi(t)} }{\Vert W_i^{\pi(t)} \Vert_2} , \mathbf{y}\right\rangle \right\vert \nonumber \\
    &=  \left\vert \left\langle \frac{ W_i^t }{\Vert W_i^t \Vert_2} - \frac{ W_i^{\pi(t)} }{\Vert W_i^{\pi(t)} \Vert_2} , \mathbbm{1}_{\mathcal{C}_\ell} \odot \mathbf{y}\right\rangle \right\vert 
    \leq \left\Vert \frac{ W_i^t }{\Vert W_i^t \Vert_2} - \frac{ W_i^{\pi(t)} }{\Vert W_i^{\pi(t)} \Vert_2} \right\Vert_2 \Vert \mathbbm{1}_{\mathcal{C}_\ell} \odot \mathbf{y} \Vert_2 
    \leq \sqrt{n_\ell} \left\Vert \frac{ W_i^t }{\Vert W_i^t \Vert_2} - \frac{ W_i^{\pi(t)} }{\Vert W_i^{\pi(t)} \Vert_2} \right\Vert_2, \label{eq:eps net inner product bound}
\end{align}
where $\mathbbm{1}_{\mathcal{C}_\ell}$ is the indicator vector on $\mathcal{C}_\ell$, and $\odot$ is the Hadamard (element-wise) product. 

Next, we establish the following lemma.
\begin{lem} \label{lem:u-v normalized norm bound}
Let $u,v\in\mathbb{R}^n$ be such that $\Vert u \Vert_2 = 1$ and $\Vert v \Vert_2 \geq 1$. Then,
\begin{equation}
    \left\Vert {u} - \frac{v}{\Vert v \Vert_2}\right\Vert_2 \leq \Vert u-v\Vert_2.
\end{equation}
\end{lem}
\begin{proof}
Define $v_p = v/\Vert v \Vert_2$, and $v_\perp = a v$ where $a$ is the scalar that minimizes $\Vert u - a v\Vert_2$. That is, $v_\perp$ is the projection of $u$ onto $v$, and hence $\langle v_\perp, u - v_\perp \rangle = 0$. In particular, we have that $a = {\langle u,v\rangle}/{\Vert v \Vert_2^2}$.
Overall, we have that $\Vert v \Vert_2 \geq 1$, $\Vert v_p\Vert_2 = 1$, and $\Vert v_\perp \Vert_2 \leq 1$, where the last inequality is due to
\begin{equation}
    \Vert v_\perp \Vert_2 = \frac{ \mid \langle u,v\rangle  \mid  \Vert v \Vert_2}{\Vert v \Vert_2^2}  =  \frac{ \mid \langle u,v\rangle  \mid  }{\Vert v \Vert_2} \leq \frac{\Vert u \Vert_2 \Vert v \Vert_2 }{\Vert v \Vert_2} = 1,
\end{equation}
using the Cauchy-Shwarz inequality and the fact that $\Vert u\Vert_2 = 1$. Since $v,v_p,v_\perp$ are just different scalings of $v$, it follows that $\Vert v - v_\perp \Vert_2 \geq \Vert v_p - v_\perp\Vert_2$.
Therefore, we can write
\begin{equation}
    \Vert u-v_p\Vert_2^2 = \Vert u-v_\perp\Vert_2^2 +  \Vert v_\perp-v_p\Vert_2^2 \leq \Vert u-v_\perp\Vert_2^2 + \Vert v - v_\perp \Vert_2^2 = \Vert u - v \Vert_2^2,
\end{equation}
where we used the fact $\langle u-v_\perp, v \rangle = 0$ (and thus also $\langle u-v_\perp, v_\perp - v_p \rangle = \langle u-v_\perp, v - v_\perp \rangle = 0$).
\end{proof}

The following is a corollary of Lemma~\ref{lem:u-v normalized norm bound} which is useful for our purposes.
\begin{cor} \label{cor:u,v normalized error bound}
Let $u,v\in \mathbb{R}^n$. Then
\begin{equation}
    \left\Vert \frac{u}{\Vert u \Vert_2} - \frac{v}{\Vert v \Vert_2}\right\Vert_2 \leq \frac{\Vert u-v\Vert_2}{\min \{ \Vert u \Vert_2, \Vert v \Vert_2 \} }.
\end{equation}
\end{cor}
\begin{proof}
If $\Vert u \Vert_2 \leq \Vert v \Vert_2$ we define $\widetilde{u} = u/\Vert u \Vert_2$, $\widetilde{v} = v/\Vert u \Vert_2$. Then we have $\Vert \widetilde{u} \Vert_2 = 1$ and $\Vert \widetilde{v} \Vert_2 \geq 1$, and by Lemma~\ref{lem:u-v normalized norm bound}
\begin{equation}
    \left\Vert \frac{u}{\Vert u \Vert_2} - \frac{v}{\Vert v \Vert_2}\right\Vert_2 = \left\Vert \widetilde{u} - \frac{\widetilde{v}}{\Vert \widetilde{v} \Vert_2} \right\Vert_2 \leq \Vert \widetilde{u} - \widetilde{v} \Vert_2 = \frac{\Vert u - v \Vert_2}{\Vert u \Vert_2}. \label{eq:u,v cor}
\end{equation}
If on the other hand $\Vert u \Vert_2 \geq \Vert v \Vert_2$, we define $\widetilde{u} = v/\Vert v \Vert_2$, $\widetilde{v} = u/\Vert v \Vert_2$, and the proof follows from applying Lemma~\ref{lem:u-v normalized norm bound} as in~\eqref{eq:u,v cor}.
\end{proof}
Now, let us write
\begin{multline}
\max_{1\leq i \leq n} \sup_{t = 1,\ldots,\infty} \left\Vert \frac{ W_i^t}{\Vert W_{i}^{t}\Vert_2} - \frac{W_{i}^{\pi(t)}}{\Vert W_{i}^{\pi(t)}\Vert_2}\right\Vert_2 \\
= \max_{1\leq i \leq n} \max \left\{ \max_{1 \leq t \leq t_{M_i}^{(i)}} \left\Vert \frac{ W_i^t}{\Vert W_{i}^{t}\Vert_2} - \frac{W_{i}^{\pi(t)}}{\Vert W_{i}^{\pi(t)}\Vert_2}\right\Vert_2 , \sup_{t > t_{M_i}^{(i)}} \left\Vert \frac{ W_i^t}{\Vert W_{i}^{t}\Vert_2} - \frac{W_{i}^{\pi(t)}}{\Vert W_{i}^{\pi(t)}\Vert_2}\right\Vert_2 \right\}. \label{eq:eps net sup t norm}
\end{multline}
Since $\Vert W_i^t \Vert_2 \leq \Vert W_i^{t_{M_i}^{(i)}} \Vert_2$ for all $t>t_{M_i}^{(i)} = T_\ell$ (by Proposition~\ref{prop:random walk distribution properties} property~\ref{prop:random-walk distribution properties 2}), using Corollary~\ref{cor:u,v normalized error bound} we have
\begin{align}
    &\sup_{t > t_{M_i}^{(i)}} \left\Vert \frac{ W_i^t}{\Vert W_{i}^{t}\Vert_2} - \frac{W_{i}^{\pi(t)}}{\Vert W_{i}^{\pi(t)}\Vert_2}\right\Vert_2 
    = \sup_{t > T_\ell}\left\Vert \frac{ W_i^t }{\Vert W_i^t \Vert_2} - \frac{ W_i^{T_\ell} }{\Vert W_i^{T_\ell} \Vert_2} \right\Vert_2  
    \leq \sup_{t > T_\ell} \frac{\Vert W_i^t - W_i^{T_\ell} \Vert_2}{\Vert W_i^t \Vert_2} \nonumber \\
    &\leq \sup_{t > T_\ell}  \sqrt{n_\ell} \Vert W_i^t - W_i^{T_\ell} \Vert_2
    \leq \sup_{t > T_\ell} \sqrt{n_\ell \left( \Vert W_i^{T_\ell} \Vert^2_2 - \Vert W_i^{t} \Vert^2_2 \right) } \leq \frac{\varepsilon}{\sqrt{n_\ell}}, \label{eq:sup t large eps net}
\end{align}
where we also used~\eqref{eq:pi map def}, property~\ref{prop:random-walk distribution properties 1} in Proposition~\ref{prop:random walk distribution properties}, and~\eqref{eq:T_l distance from stationary distribution}.

Next, if $M_i>1$, using the fact that $t_{1}^{(i)} = 1$ and~\eqref{eq:pi map def}, we have
\begin{align}
    \max_{1 \leq t \leq t_{M_i}^{(i)}} \left\Vert \frac{ W_i^t}{\Vert W_{i}^{t}\Vert_2} - \frac{W_{i}^{\pi(t)}}{\Vert W_{i}^{\pi(t)}\Vert_2}\right\Vert_2 
    &= \max_{2\leq j \leq M_{i}} \max_{t^{(i)}_{j-1} < t \leq t_j^{(i)}} \left\Vert \frac{ W_i^t}{\Vert W_{i}^{t}\Vert_2} - \frac{W_{i}^{\pi(t)}}{\Vert W_{i}^{\pi(t)}\Vert_2}\right\Vert_2 \nonumber \\
    &= \max_{2\leq j \leq M_{i}} \max_{t^{(i)}_{j-1} < t \leq t_j^{(i)}} \left\Vert \frac{ W_i^t}{\Vert W_{i}^{t}\Vert_2} - \frac{W_{i}^{t_j^{(i)}}}{\Vert W_{i}^{t_j^{(i)}}\Vert_2}\right\Vert_2. \label{eq:normalized maximal error expression in proof}
\end{align}
Note that if $M_i=1$, then the left-most term in the above expression is $0$. Since $\Vert W_{i}^{t}\Vert_2 \geq \Vert W_{i}^{t_j^{(i)}}\Vert_2$ for all integer $t \in (t^{(i)}_{j-1},t_j^{(i)}]$ (according to property~\ref{prop:random-walk distribution properties 2} in Proposition~\ref{prop:random walk distribution properties}), applying Corollary~\ref{cor:u,v normalized error bound} to~\eqref{eq:normalized maximal error expression in proof} gives
\begin{align}
    &\max_{1 \leq t \leq t_{M_i}^{(i)}} \left\Vert \frac{ W_i^t}{\Vert W_{i}^{t}\Vert_2} - \frac{W_{i}^{\pi(t)}}{\Vert W_{i}^{\pi(t)}\Vert_2}\right\Vert_2 
    \leq 
    \max_{2\leq j \leq M_{i}} \max_{t^{(i)}_{j-1} < t \leq t_j^{(i)}} \frac{ \Vert W_i^t - W_{i}^{t_j^{(i)}} \Vert_2 }{\Vert W_{i}^{t_j^{(i)}}\Vert_2} \nonumber \\
    &= \max_{2\leq j \leq M_{i}} \max_{t^{(i)}_{j-1} +1 \leq t \leq t_j^{(i)}} \frac{ \Vert W_i^t - W_{i}^{t_j^{(i)}} \Vert_2 }{\Vert W_{i}^{t_j^{(i)}}\Vert_2} 
\leq  \max_{2\leq j \leq M_i} \sqrt{\frac{ \Vert W_i^{t_{j-1}^{(i)}+1} \Vert_2^2 - \Vert W_i^{t_{j}^{(i)}} \Vert_2^2 }{\Vert W_i^{t_{j}^{(i)}} \Vert_2^2}} \leq \frac{\varepsilon}{\sqrt{n_\ell}}, \label{eq:normalized maximal error expression 2 in proof}
\end{align}
where we also made use of property~\ref{prop:random-walk distribution properties 3} from Proposition~\ref{prop:random walk distribution properties} and~\eqref{eq:choosing t_j iterative inequality}) in the last two inequalities.
Substituting~\eqref{eq:normalized maximal error expression 2 in proof} and~\eqref{eq:sup t large eps net} into~\eqref{eq:eps net sup t norm}, for $i\in \mathcal{C}_\ell$, gives
\begin{equation}
    \sup_{t =1,\ldots,\infty} \left\Vert \frac{ W_i^t}{\Vert W_{i}^{t}\Vert_2} - \frac{W_{i}^{\pi(t)}}{\Vert W_{i}^{\pi(t)}\Vert_2}\right\Vert_2 \leq \frac{\varepsilon}{\sqrt{n_\ell}}, 
\end{equation}
which together with~\eqref{eq:eps net inner product bound} provides the required result.

\section{Proof of Lemma~\ref{lem:t_1 and M_i bounds}} \label{appendix:proof of lemma t_1 and M_i bounds}
According to Algorithm~\ref{alg:evaluating diffusion time steps}, $\{t_j^{(i)}\}_{j=1}^{M_i}$ must be distinct positive integers, and therefore $M_i \leq t_{M_i}^{(i)} = T_\ell$ for all $i\in\mathcal{C}_\ell$. Next, recall that $t_{j-1}^{(i)}$ is the smallest integer $t \in [1, t^{(i)}_j-1]$ such that~\eqref{eq:choosing t_j iterative inequality} holds. Hence, since $\Vert W_i^t \Vert$ is monotonically decreasing in $t$, it is evident that either $t_{j-1}^{(i)}  = t_{1}^{(i)} = 1$, or 
\begin{equation}
    \Vert W_i^{t_{j-1}^{(i)}} \Vert_2^2 > \Vert W_i^{t_{j}^{(i)}} \Vert_2^2 (1+\frac{\varepsilon^2}{n_\ell}). \label{eq:bound 2 recurrence}
\end{equation}
Equation~\eqref{eq:bound 2 recurrence} is a recurrence relation in the index $j$, and expanding it for $j = M_i, M_i-1, \ldots , 2$ gives
\begin{align}
    1 \geq \Vert W_i^{t_{1}^{(i)}} \Vert_2^2 
    &> \Vert W_i^{t_{M_i}^{(i)}} \Vert_2^2 (1+\frac{\varepsilon^2}{n_\ell})^{M_i-1} \geq \frac{1}{n_\ell}(1+\frac{\varepsilon^2}{n_\ell})^{M_i-1},
\end{align}
where we used property~\ref{prop:random-walk distribution properties 1} in Proposition~\ref{prop:random walk distribution properties}.
Therefore, we have
\begin{equation}
    M_i < 1 + \frac{\log(n_\ell)}{\log({1+{\varepsilon^2}/{n_\ell})}}.
\end{equation}
Since $M_i$ is an integer, the above inequality actually implies
\begin{equation}
    M_i \leq \left\lceil \frac{\log(n_\ell)}{\log({1+{\varepsilon^2}/{n_\ell})}} \right\rceil.
\end{equation}

\section{Proof of Theorem~\ref{thm:test power}} \label{appendix:proof of testing power and accuracy}
Let $\widetilde{\alpha} \in (0,1)$, suppose that $\gamma \geq h(\varepsilon) + \sqrt{0.5\log ({1}/\widetilde{\alpha} )}$ where $h(\varepsilon)$ is from~\eqref{eq:beta def}, and denote $\mathbf{w} = W_i^t$. Then, according to Corollary~\ref{cor:H_1(alpha) bound},~\eqref{eq:beta def}, and Lemma~\ref{lem:t_1 and M_i bounds}, with probability at least $1-\widetilde{\alpha}$ we have
\begin{align}
    \mathcal{S}(W_i^{\pi(t)}) 
    > \gamma - \varepsilon - \sqrt{0.5\log ({1}/\widetilde{\alpha} )}
    \geq  h(\varepsilon) - \varepsilon
    \geq \sqrt{0.5 \log(\sum_{i=1}^n M_i/{\alpha})}.
\end{align} 
Therefore, if $\gamma \geq h(\varepsilon) + \sqrt{0.5\log ({1}/\widetilde{\alpha} )}$, according to~\eqref{eq:G_hat def} we have that 
\begin{equation}
    \sup_{f_1,f_0 \in H_1(\mathbf{w},\gamma)}\operatorname{Pr}\{ Q_\mathbf{z} = 1  \mid  f_0, f_1 \} \geq \sup_{f_1,f_0 \in H_1(\mathbf{w},\gamma)} \operatorname{Pr}\{ W_i^{\pi (t)} \in \hat{\mathcal{G}}_\mathbf{z}  \mid f_0,f_1 \} \geq 1-\widetilde{\alpha}. \label{eq: power lower bound in proof}
\end{equation}
Clearly, if $\gamma > h(\varepsilon)$, the smallest $\widetilde{\alpha}$ that satisfies $\gamma \geq h(\varepsilon) + \sqrt{0.5\log ({1}/\widetilde{\alpha} )}$ is obtained by choosing $\widetilde{\alpha}$ according to $\sqrt{0.5\log ({1}/\widetilde{\alpha} )} = \gamma - h(\varepsilon)$. Manipulating this expression to extract $\widetilde{\alpha}$ gives $\widetilde{\alpha} = \operatorname{exp}\left[-2(\gamma - h(\varepsilon))^2\right]$, which together with~\eqref{eq: power lower bound in proof} implies~\eqref{eq: power lower bound}. To prove~\eqref{eq: test accuracy upper bound}, we can write for all $f_1,f_0 \in H_1(W_i^t,\gamma)$
\begin{multline}
    \mathbb{E} [ \inf_{\hat{\mathbf{w}}\in \hat{\mathcal{G}}_\mathbf{z}} \mathcal{E}_{\operatorname{TV}} (\hat{\mathbf{w}},W_i^t)  \mid  f_0, f_1] \\
    \leq \mathcal{E}_{\operatorname{TV}} (W_i^{\pi (t)},W_i^t) \cdot \operatorname{Pr}\{ W_i^{\pi (t)} \in \hat{\mathcal{G}}_\mathbf{z}  \mid  f_0,f_1 \}
    + 1 \cdot \operatorname{Pr}\{ W_i^{\pi (t)} \notin \hat{\mathcal{G}}_\mathbf{z}  \mid  f_0,f_1 \},
\end{multline}
where we use the fact that the total variation distance is upper bounded by $1$. Applying Lemma~\ref{lem:epsilon net} (and particularly~\eqref{eq:total variation dist bound eps-net}) to bound $\mathcal{E}_{\operatorname{TV}} ({W_i^{\pi(t)}},W_i^t)$, and using~\eqref{eq: power lower bound in proof} with $\widetilde{\alpha} = \operatorname{exp}\left[-2(\gamma - h(\varepsilon))^2\right]$, gives
\begin{equation}
    \sup_{f_1,f_0 \in H_1(W_i^t,\gamma)} \mathbb{E} [ \inf_{\hat{\mathbf{w}}\in \hat{\mathcal{G}}_\mathbf{z}} \mathcal{E}_{\operatorname{TV}} (\hat{\mathbf{w}},W_i^t)  \mid  f_0,f_1] \leq \frac{\varepsilon \Vert W_i^t \Vert_2}{2} + \operatorname{exp}\left[-2(\gamma - h(\varepsilon))^2\right]. \label{eq:test accuracy bound in proof}
\end{equation}
Observe that by the definition of $H_1(\mathbf{w},\gamma)$ in~\eqref{eq:H_1 specific def}, the alternative $H_1(W_i^t,\gamma)$ implies that $\Vert W_i^t \Vert_2 \leq \langle W_i^t,\mathbf{s} \rangle / \gamma \leq (1-p)/\gamma$. Applying this observation to~\eqref{eq:test accuracy bound in proof}, together with the fact that $\Vert W_i^t \Vert_2 \leq 1$, gives~\eqref{eq: test accuracy upper bound}.

\section{Proof of Theorem~\ref{thm:consistency rate}} \label{appendix:proof of onsistency}
Taking $\alpha=1/\log n$, and since $\varepsilon$ is kept constant, we have
\begin{align}
    \limsup_{n\rightarrow\infty}\frac{\hat{h}_{n,\alpha,\lambda_{<1}}(\varepsilon) }{\sqrt{\log n}} 
    &= \limsup_{n\rightarrow\infty} \sqrt{\frac{0.5\log\left({n}{\log n} \cdot \min \left\{ \left\lceil \frac{\log(n/\varepsilon)}{\log(\lambda_{<1}^{-1})} \right\rceil,  \left\lceil \frac{\log(n)}{\log(1+\varepsilon^2/n)} \right\rceil \right\} \right)}{\log n}}, \label{eq:limsup beta_hat bound}
\end{align}
where $\hat{h}_{n,\alpha,\lambda_{<1}}$ is from~\eqref{eq:beta upper bound}. 
We begin by proving part~\ref{cor_part:cor consistency part 1} of Theorem~\ref{thm:consistency rate}. 
Notice that we can write
\begin{align}
    &\limsup_{n\rightarrow\infty} \sqrt{\frac{0.5\log\left({n}{\log n} \cdot \min \left\{ \left\lceil \frac{\log(n/\varepsilon)}{\log(\lambda_{<1}^{-1})} \right\rceil,  \left\lceil \frac{\log(n)}{\log(1+\varepsilon^2/n)} \right\rceil \right\} \right)}{\log n}} 
    \leq \lim_{n\rightarrow\infty} \sqrt{\frac{0.5\log\left({n}{\log n} \cdot \left\lceil \frac{\log(n)}{\log(1+\varepsilon^2/n)} \right\rceil \right)}{\log n}} \nonumber \\
    &= \lim_{n\rightarrow\infty} \sqrt{\frac{0.5\log\left(\frac{n^2 (\log n)^2}{\varepsilon^2} \right)}{\log n}} = \lim_{n\rightarrow\infty} \sqrt{\frac{\log n + \log(\log n) - \log \varepsilon }{\log n}} = 1,
\end{align}
where we used~\eqref{eq:log(1+n)/log(1+n/varEps^2) asymptotics}. Therefore, if $\liminf_{n\rightarrow\infty} \gamma_n / \sqrt{\log n} > 1$, then $\gamma_n > \hat{h}_{n,\alpha,\lambda_{<1}}(\varepsilon) \geq h(\varepsilon) $ for sufficiently large $n$. 
Furthermore, we have that $\gamma_n - \hat{h}_{n,\alpha,\lambda_{<1}}(\varepsilon) \underset{n\rightarrow \infty}{\longrightarrow} \infty$, since
\begin{equation}
    \liminf_{n\rightarrow\infty} \frac{\gamma_n - \hat{h}_{n,\alpha,\lambda_{<1}}(\varepsilon)}{\sqrt{\log n}} \geq \liminf_{n\rightarrow\infty} \frac{\gamma_n}{\sqrt{\log n}} - \limsup_{n\rightarrow\infty} \frac{\hat{h}_{n,\alpha,\lambda_{<1}}(\varepsilon)}{\sqrt{\log n}} > 0. \label{eq:consistency proof part 1 reasoning 2}
\end{equation}
Applying Theorem~\ref{thm:test power}, we have
\begin{equation}
    \sup_{\mathbf{w} \in \mathcal{F}}\sup_{f_1,f_0 \in H_1(\mathbf{w},\gamma)}\mathbb{E} [ \inf_{\hat{\mathbf{w}}\in \hat{\mathcal{G}}_\mathbf{z}} \mathcal{E}_{\operatorname{TV}} (\hat{\mathbf{w}},\mathbf{w})  \mid  f_1,f_0] 
    \leq \frac{\varepsilon (1-p)}{2\gamma_n} + e^{-2(\gamma_n - h(\varepsilon))^2} \leq \frac{\varepsilon (1-p)}{2\gamma_n} + e^{-2(\gamma_n - \hat{h}_{n,\alpha,\lambda_{<1}}(\varepsilon))^2}, \label{eq:consistency proof part 1 reasoning 1}
\end{equation}
for sufficiently large $n$. 
Consequently, combining~\eqref{eq:consistency proof part 1 reasoning 2}, ~\eqref{eq:hypothesis testing local risk def}, ~\eqref{eq:type I error bound in local test}, and the fact that $\lim_{n\rightarrow \infty} \gamma_n = \infty$, we obtain 
\begin{equation}
    r_{\mathcal{F}}^{(n)}(\hat{\mathcal{G}}_\mathbf{z},\gamma_n) \leq \frac{1}{\log n} + \frac{\varepsilon (1-p)}{2\gamma_n} + e^{-2(\gamma_n - \hat{h}_{n,\alpha,\lambda_{<1}}(\varepsilon))^2} \underset{n\rightarrow \infty}{\longrightarrow} 0.
\end{equation}

Next, we prove parts~\ref{cor_part:cor consistency part 2} and~\ref{cor_part:cor consistency part 3} of Theorem~\ref{thm:consistency rate}.
From~\eqref{eq:limsup beta_hat bound}, we have
\begin{align}
    \limsup_{n\rightarrow\infty}\frac{\hat{h}_{n,\alpha,\lambda_{<1}}(\varepsilon) }{\sqrt{\log n}} 
    &\leq \limsup_{n\rightarrow\infty} \sqrt{\frac{0.5\log\left({n}{\log n} \cdot \left\lceil \frac{\log(n/\varepsilon)}{\log(\lambda_{<1}^{-1})} \right\rceil \right)}{\log n}} 
    = \limsup_{n\rightarrow\infty} \sqrt{\frac{0.5\log\left( \frac{ n \cdot \log n \cdot \log(n/\varepsilon) }{\log(\lambda_{<1}^{-1}) }\right)}{\log n}} \nonumber \\
    &= \limsup_{n\rightarrow\infty} \sqrt{\frac{0.5 (\log n + \log(\log n) + \log(\log (n/\varepsilon)) - \log(\log(\lambda_{<1}^{-1}))) }{\log n}} \nonumber \\
    &= \limsup_{n\rightarrow\infty} \sqrt{0.5 -  \frac{\log(\log(\lambda_{<1}^{-1}))}{\log n}}.
\end{align}
If $\lambda_{<1}$ is bounded away from $1$, then it follows that
\begin{equation}
    \limsup_{n\rightarrow\infty}\frac{\hat{h}_{n,\alpha,\lambda_{<1}}(\varepsilon) }{\sqrt{\log n}} 
    \leq \limsup_{n\rightarrow\infty} \sqrt{0.5 -  \frac{\log(\log(\lambda_{<1}^{-1}))}{\log n}} = \sqrt{0.5}.
\end{equation}
Thus, if $\liminf_{n\rightarrow\infty} \gamma_n / \sqrt{\log n} > \sqrt{0.5}$, then $\gamma_n > \hat{h}_{n,\alpha,\lambda_{<1}}(\varepsilon) \geq h(\varepsilon) $ for all sufficiently large $n$, and $\hat{\mathcal{G}}_\mathbf{z}$ is locally consistent following~\eqref{eq:consistency proof part 1 reasoning 1}--\eqref{eq:consistency proof part 1 reasoning 2}.
On the other hand, if $\lim_{n\rightarrow \infty} \lambda_{<1} = 1$, we have
\begin{equation}
    \log(\lambda_{<1}^{-1}) \underset{n\rightarrow \infty}{\sim} 1-\lambda_{<1}, \label{eq:log lambda_inv asymptotics}
\end{equation}
and therefore, if in addition $\lim_{n\rightarrow \infty} (1-\lambda_{<1}) n^\gamma > 0$ for some $0 < \gamma \leq 0.5$, we obtain
\begin{align}
    &\limsup_{n\rightarrow\infty}\frac{\hat{h}_{n,\alpha,\lambda_{<1}}(\varepsilon) }{\sqrt{\log n}}  
    \leq \lim_{n\rightarrow\infty} \sqrt{0.5 - \lim_{n\rightarrow\infty} \frac{\log(\log(\lambda_{<1}^{-1}))}{\log n}} \nonumber \\
    & = \lim_{n\rightarrow\infty}\sqrt{0.5 - \frac{\log(1-\lambda_{<1})}{\log n}} 
    = \lim_{n\rightarrow\infty}\sqrt{0.5 -  \frac{\log((1-\lambda_{<1})n^{\gamma}) - \gamma \log n}{\log n}} \leq \sqrt{0.5 + \gamma}.
\end{align}
Consequently, if $\liminf_{n\rightarrow\infty} \gamma_n / \sqrt{\log n} > \sqrt{0.5+\gamma}$, then $\gamma_n > \hat{h}_{n,\alpha,\lambda_{<1}}(\varepsilon) \geq h(\varepsilon)$ for all sufficiently large $n$, and again $\hat{\mathcal{G}}_\mathbf{z}$ is locally consistent following~\eqref{eq:consistency proof part 1 reasoning 1}--\eqref{eq:consistency proof part 1 reasoning 2}.

\section{Proof of Theorem~\ref{thm:minimax risk}} \label{appendix:Proof of minimax risk}
Let $\{m_n\}$ be a sequence of integers, and for each index $n$ consider a matrix $W\in\mathbb{R}^{n \times n}$ such that the corresponding graph $G$ has $L_n = \lceil \frac{n}{m_n} \rceil$ connected components $\{ \mathcal{C}_\ell \}_{\ell=1}^{L_n}$, where $\vert \mathcal{C}_1 \vert = \vert \mathcal{C}_2 \vert = \ldots = \vert \mathcal{C}_{L_n-1} \vert = m_n$, and $\vert \mathcal{C}_{L_n} \vert = n - m_n (L_n - 1) $. That is, $W$ can be permuted (symmetrically) into a block-diagonal form with $L_n$ blocks, where the first $L_n-1$ blocks are of size $m_n\times m_n$, and the last block is of the appropriate size to match the dimensions of $W$. In addition, we take the values of $W$ so that $W_i$ is the uniform distribution over $\mathcal{C}_\ell$ for each $i\in\mathcal{C}_\ell$, $\ell = 1,\ldots,L_n$. Specifically, $W_{i,j} = 1/\vert \mathcal{C}_\ell  \vert$ for all $i,j \in \mathcal{C}_\ell$, and $W_{i,j} = 0$ otherwise. Clearly, $W$ is nonnegative, symmetric, stochastic, and PSD, thereby satisfying Assumption~\ref{assump:W properties}. In this case, $\mathcal{F}$ has only $L_n$ distinct distributions, which are the uniform distributions over the connected components of $G$, i.e., $\mathcal{F} = \{ \mathbbm{1}_{\mathcal{C}_1}/m_n, \ldots, \mathbbm{1}_{\mathcal{C}_{L_n-1}}/m_n, \mathbbm{1}_{\mathcal{C}_{L_n}}/(n - m_n (L_n - 1)) \}$, where $\mathbbm{1}_{\mathcal{C}_\ell}$ is the indicator vector over $\mathcal{C}_\ell$. Continuing, let us define the hypothesis $H_0^{'}$ as
\begin{equation}
    H_0^{'}: \;\; f_1(x_i) = f_0(x_i) \quad \forall i=1,\ldots,n.
\end{equation}
Notably, $H_0^{'}$ implies that $s(x_i) = 0$ for all $i=1,\ldots,n$, and is therefore a subset of $H_0$.
Additionally, we define $H_1^{'}(\ell)$, for $\ell=1,\ldots,L_n-1$, as alternative hypotheses to $H_0$ (and $H_0^{'}$), as
\begin{equation}
    H_1^{'}(\ell): \;\; 
    \begin{dcases}
    f_0(x_i) = 0, &\forall i\in \mathcal{C}_{\ell}, \\
    f_1(x_i) = f_0(x_i), &\forall i\notin \mathcal{C}_{\ell}.
    \end{dcases} \label{eq:H_1_prime def}
\end{equation}
Observe that under $H_1^{'}(\ell)$, we have that $s(x_i) = 1-p$ for all $i\in \mathcal{C}_\ell$, and hence
\begin{equation}
    \langle {\mathbbm{1}_{\mathcal{C}_\ell}}/{m_n}, \mathbf{s} \rangle = 1-p = (1-p) \sqrt{m_n} \Vert {\mathbbm{1}_{\mathcal{C}_\ell}}/{m_n} \Vert_2.
\end{equation}
Taking 
\begin{equation}
    m_n = \left\lceil \left( \frac{\gamma_n}{1-p} \right)^2 \right\rceil, \label{eq:m_n choice}
\end{equation}
guarantees that $\langle {\mathbbm{1}_{\mathcal{C}_\ell}}/{m_n}, \mathbf{s} \rangle \geq \gamma_n \Vert {\mathbbm{1}_{\mathcal{C}_\ell}}/{m_n} \Vert_2$.
Therefore, each hypothesis $H_1^{'}(\ell)$  is a subset of the alternative $H_1(\mathbbm{1}_{\mathcal{C}_\ell}/m_n,\gamma_n)$, for $\ell=1,\ldots,L_n-1$.
Consequently, by our definitions of $H_0^{'}$ and $\{H_1^{'}(\ell)\}_{\ell=1}^{L_n-1}$, we have that
\begin{align}
    \widetilde{R}^{(n)}(\gamma_n) &= \min_{Q_\mathbf{z}} \sup_{W^{'}\in \mathcal{W}} R_{\mathcal{F}}^{(n)}(Q_{\mathbf{z}},\gamma_n) \geq \min_{Q_\mathbf{z}} R_{\mathcal{F}}^{(n)}(Q_{\mathbf{z}},\gamma_n) \nonumber \\
    &\geq \min_{Q_\mathbf{z}} \left\{ \sup_{f_1,f_0\in H_0^{'}}\operatorname{Pr}\{Q_{\mathbf{z}}=1  \mid  f_0, f_1 \} + \max_{\ell=1,\ldots,L_n-1} \sup_{f_1,f_0 \in H_1^{'}(\ell)}\operatorname{Pr}\{Q_{\mathbf{z}}=0  \mid  f_0, f_1  \}\right\}.  \label{eq:R_tilde lower bound}
\end{align}
Additionally, according to~\eqref{eq:m_n choice} and~\eqref{eq:alpha_n lower rate bound}, we have
\begin{equation}
    \liminf_{n\rightarrow \infty} \frac{m_n}{\log n} < \frac{1}{\log (p^{-1})}. \label{eq:m_n lower rate bound}
\end{equation}

We now derive a lower bound on the right-hand side of~\eqref{eq:R_tilde lower bound} under the condition~\eqref{eq:m_n lower rate bound}.
Notice that under $H_0^{'}$, according to~\eqref{eq:z_i distribution}, the labels $z_i$ are sampled independently from $(Z \mid X=x_i) \sim \operatorname{Bernoulli}(p)$. On the other hand, under $H_1^{'}(\ell)$, the labels $z_i$ are always equal to $1$ for all $i\in\mathcal{C}_\ell$, and are sampled independently from $\operatorname{Bernoulli}(p)$ for all $i\notin \mathcal{C}_\ell$.  Let $Q_{\mathbf{z}}^\star$ be a test that outputs $1$ if $[\mathbf{z}]_{\mathcal{C}_\ell} = \mathbf{1}$ for some $\ell \in \{1,\ldots,L_n-1\}$, and $0$ otherwise, where $[\mathbf{z}]_{\Omega}$ is the restriction of the entries of $\mathbf{z}$ to the index set $\Omega$, and $\mathbf{1}$ is a corresponding vector of ones. Note that $\operatorname{Pr}\{Q_{\mathbf{z}}^\star=0  \mid  H_1^{'}(\ell)  \} = 0$ for all $\ell=1,\ldots,L_n-1$. In addition, we have
\begin{equation}
    \operatorname{Pr}\{Q_{\mathbf{z}}^\star=1  \mid  H_0^{'} \} = 1-(1-p^{m_n})^{L_n-1} > 1 - \operatorname{exp}({-p^{m_n}(L_n-1)}), \label{eq:Pr Q_star=1 under the null expression}
\end{equation}
where we used the inequality $\log (1-x) < -x$ for $0<x<1$.
Let us write
\begin{equation}
    \log\left( p^{m_n}(L_n-1) \right) = { \log (L_n-1) - m_n \log (p^{-1}) } \geq \log \left({n} - m_n\right) - \log \left( m_n \right) - m_n \log (p^{-1}), \label{eq:log L_n-1 -m_n log p_inv lowe bound}
\end{equation}
where we used $L_n = \lceil \frac{n}{m_n} \rceil$.
Combining~\eqref{eq:log L_n-1 -m_n log p_inv lowe bound} and~\eqref{eq:m_n lower rate bound}, we obtain
\begin{equation}
    \liminf_{n\rightarrow \infty} \frac{\log \left( p^{m_n}(L_n-1) \right)}{\log n} > \liminf_{n\rightarrow \infty} \frac{\log \left({n} - m_n\right) - \log \left( m_n \right) }{\log n} - 1 = 0,
\end{equation}
which implies that $\lim_{n\rightarrow\infty} p^{m_n}(L_n-1) = \infty$, and therefore $\lim_{n\rightarrow\infty} \operatorname{Pr}\{Q_{\mathbf{z}}^\star=1  \mid  H_0^{'} \} = 1$. 

Last, we prove that $Q_{\mathbf{z}}^\star$ is a test which achieves the minimum in~\eqref{eq:R_tilde lower bound}. If $Q_{\mathbf{z}}$ is any other test, then one of the following two statements must be true (or both).
\begin{enumerate}
    \item $Q_{\mathbf{z}}$ outputs $0$ for some $\mathbf{z}$ satisfying $[\mathbf{z}]_{\mathcal{C}_\ell} = \mathbf{1}$ for some $\ell\in\{1,\ldots,L_n-1\}$.
    \item $Q_{\mathbf{z}}$ outputs $1$ for some $\mathbf{z}$ satisfying $[\mathbf{z}]_{\mathcal{C}_\ell} \neq \mathbf{1}$ for all $\ell=1,\ldots,L_n-1$.
\end{enumerate}
If the first statement is true, then $\operatorname{Pr}\{Q_{\mathbf{z}}=0  \mid  H_1^{'}(\ell) \} = 1$ for some $\ell\in\{1,\ldots,L_n-1\}$, and therefore the expression inside the minimization in~\eqref{eq:R_tilde lower bound} is lower bounded by $1$ for the test $Q_{\mathbf{z}}$. The analogous expression for the test $Q_{\mathbf{z}}^\star$ is upper bounded (with a strict inequality) by $1$ according to~\eqref{eq:Pr Q_star=1 under the null expression}, and is therefore always smaller than for $Q_{\mathbf{z}}$.
On the other hand, if the second statement is true, then $Q_{\mathbf{z}}$ cannot be a test which achieves the minimum in~\eqref{eq:R_tilde lower bound}, since there is no alternative $H_1^{'}(\ell)$ for which $[\mathbf{z}]_{\mathcal{C}_\ell} \neq \mathbf{1}$ for all $\ell=1,\ldots,L_n-1$. That is, the expression inside the minimization in~\eqref{eq:R_tilde lower bound} can always be made smaller for the test $Q_{\mathbf{z}}$ if we enforce that $Q_{\mathbf{z}}=0$ for all $\mathbf{z}$ that satisfy $[\mathbf{z}]_{\mathcal{C}_\ell} \neq \mathbf{1}$ for all $\ell=1,\ldots,L_n-1$, while retaining all other aspects of $Q_{\mathbf{z}}$ unchanged. 

\section{Proof of Proposition~\ref{prop:doubly stochastic KL-divergence interpretation}} \label{appendix:doubly stochastic KL-divergence interpretation proof}
If $K_{i,j} = 0$, then $D_{\text{KL}}(H_i \mid  \mid K_i) = \sum_{i,j} H_{i,j} \log (H_{i,j}/K_{i,j})$ is defined through the limit $K_{i,j}\rightarrow 0$. Therefore, to keep the objective function finite, we must have $H_{i,j} = 0$. We know that there exists $H$ that satisfies this property (that $H_{i,j}=0$ if $K_{i,j}=0$) and that is simultaneously symmetric and stochastic, due to the assumption in the theorem (that there exists $d_1,\ldots,d_n$ such that $\widetilde{W}$ from~\eqref{eq:doubly stochastic normalization} is symmetric and stochastic). Consequently, $H_{i,j}=0$ for all indices $i,j$ for which $K_{i,j}=0$. Next, the Lagrangian function associated with~\eqref{eq:KL-divergence optim} is given by
\begin{equation}
    \mathcal{L}(H,\{\lambda_i\}_{i=1}^n,\{\mu_{i,j}\}_{i,j=1}^n) = \sum_{i,j} H_{i,j} \log \left( \frac{H_{i,j}}{K_{i,j}}\right) + \sum_{i=1}^n \lambda_i (\sum_{j=1}^n H_{i,j} - 1) + \sum_{i,j=1}^n \mu_{i,j} (H_{i,j} - H_{j,i}).
\end{equation}
Taking the derivative of $\mathcal{L}$ with respect to $H_{i,j}$ and equating to zero, gives
\begin{equation}
    \log ( {H_{i,j}}) = \log(K_{i,j}) - 1 -\lambda_i - \mu_{i,j} + \mu_{j,i}, \label{eq:log W_{i,j} expression}
\end{equation}
for all $i,j$ for which $K_{i,j} \neq 0$.
Applying the symmetry constraints $H_{i,j} = H_{j,i}$ to the above, and after some manipulation, we get that
\begin{equation}
    \mu_{i,j} - \mu_{j,i} = (-\lambda_i + \lambda_j + \log (K_{i,j}) - \log (K_{j,i}))/2.
\end{equation}
Substituting this back in~\eqref{eq:log W_{i,j} expression}, we have
\begin{equation}
    \log (H_{i,j}) = -1 - (\lambda_i + \lambda_j)/2 + \log (\sqrt{K_{i,j} K_{j,i}}).
\end{equation}
Taking the exponential of both hand sides of the above equation, and defining $d_i = \operatorname{exp}(-(\lambda_i+1)/2)$, completes the proof.

\section{Analysis supporting the example in Section~\ref{sec:toy example circle}} \label{appendix:analysis for unit circle example}
We now approximate $\gamma$ from~\eqref{eq:alpha def explicit} using~\eqref{eq:W_i^t analytical}. According to~\eqref{eq:W_i^t analytical}, and for sufficiently large $n$, we can write
\begin{align}
    \gamma &= \max_{1\leq i \leq n} \sup_{t=1,\ldots,\infty} \frac{\langle W_i^t, \mathbf{s}\rangle}{\Vert W_i^t \Vert} = \max_{1\leq i \leq n} \sup_{t=1,\ldots,\infty} \frac{\sum_{j=1}^n W_{i,j}^t s(\theta_j) }{\sqrt{\sum_{j=1}^n W_{i,j}^{2t}}} \nonumber \\
    &\approx \max_{\varphi \in [0,2\pi)} \sup_{\tau>0} \frac{\sum_{j=1}^n G_\tau(\varphi,\theta_j) s(\theta_j)}{\sqrt{\sum_{j=1}^n G_\tau^2 (\varphi,\theta_j)}} = \max_{\varphi \in [0,2\pi)} \sup_{\tau>0} \frac{\frac{1}{n}\sum_{j=1}^n G_\tau(\varphi,\theta_j) s(\theta_j)}{\frac{1}{\sqrt{n}}\sqrt{\frac{1}{n}\sum_{j=1}^n G_\tau^2 (\varphi,\theta_j)}}.
\end{align}
Since $G_\tau$ and $s(\theta)$ are bounded, and $\theta_1,\ldots,\theta_n$ are i.i.d, we have
\begin{align}
    \frac{1}{n}\sum_{j=1}^n G_\tau(\varphi,\theta_j) s(\theta_j) &= \frac{1}{2\pi}\int_{-\pi}^\pi G_\tau (\varphi,\theta) s(\theta) d\theta + \mathcal{O}(\frac{1}{\sqrt{n}}), \\
    \frac{1}{\sqrt{n}}\sqrt{\frac{1}{n}\sum_{j=1}^n G_\tau^2 (\varphi,\theta_j)} &= \frac{1}{\sqrt{n}}\sqrt{\frac{1}{2\pi}\int_{-\pi}^\pi G^2_\tau (\varphi,\theta) d\theta + \mathcal{O}(\frac{1}{\sqrt{n}})}.
\end{align}
Therefore, for sufficiently large $n$ and sufficiently small $\sigma$, it follows that
\begin{equation}
    \gamma \approx \max_{\varphi \in [0,2\pi)} \sup_{\tau>0} \frac{\frac{1}{2\pi}\int_{-\pi}^\pi G_\tau (\varphi,\theta) s(\theta) d\theta}{\frac{1}{\sqrt{n}}\sqrt{\frac{1}{2\pi}\int_{-\pi}^\pi G^2_\tau (\varphi,\theta) d\theta}} = 
    \sup_{\tau>0} \sqrt{\frac{n}{2\pi}} \frac{\int_{-\pi}^\pi G_\tau (0,\theta) s(\theta) d\theta}{\sqrt{\int_{-\pi}^\pi G^2_\tau (0,\theta) d\theta}}.
\end{equation}
Next, using the expression for $s(\theta)$ from~\eqref{eq:unit circle example f and s def} we can write
\begin{align}
    &\int_{-\pi}^\pi G_\tau (0,\theta) s(\theta) d\theta = \int_{-\min \{\pi,3\pi/2\omega\}}^{\min \{\pi,3\pi/2\omega\}} e^{-\theta^2/\tau} 2p(1-p)b\cos(\omega \theta) d\theta \nonumber \\ 
    &= 2p(1-p)b \left( \int_{-\infty}^\infty e^{-\theta^2/\tau} \cos(\omega \theta) d\theta + E_1 \right) 
    = 2p(1-p)b \left( \sqrt{\pi \tau} e^{-\omega^2 \tau /4} +  E_1 \right),
\end{align}
where $E_1$ is an error term bounded by
\begin{align}
     \mid E_1 \mid  &= 2 \left\vert \int_{\min \{\pi,3\pi/2\omega\}}^\infty 
     e^{-\theta^2/\tau} \cos(\omega \theta) d\theta \right\vert \leq 2 \int_{\min \{\pi,3\pi/2\omega\}}^\infty 
     e^{-\theta^2/\tau} d\theta \nonumber \\
     &\leq 2 \int_{\min \{\pi,3\pi/2\omega\}}^\infty 
     \frac{\theta}{\tau} e^{-\theta^2/\tau} d\theta \leq e^{-(\min \{\pi,3\pi/2\omega\})^2/\tau},
\end{align}
for $\tau \leq \min \{\pi,3\pi/2\omega\}$.
Additionally, 
\begin{align}
    &{\int_{-\pi}^\pi G^2_\tau (0,\theta) d\theta} = 
    {\int_{-\pi}^{\pi} e^{-2\theta^2/\tau} d\theta}
    = {\int_{-\infty}^\infty e^{-2\theta^2/\tau} d\theta + E_2} \nonumber \\
    &= { \sqrt{\frac{\pi \tau}{2}} + E_2},
\end{align}
where $E_2$ is an error term bounded by
\begin{align}
     \mid E_2 \mid  &= 2  \int_{\pi}^\infty 
     e^{-2\theta^2/\tau} d\theta \leq  2 \int_{\pi}^\infty 
     \frac{\theta}{\tau} e^{-2\theta^2/\tau} d\theta \leq \frac{1}{2}e^{-2\pi^2/\tau},
\end{align}
for $\tau \leq \pi$.
By neglecting $E_1$ and $E_2$ (which will be justified shortly), we get
\begin{align}
    \gamma \approx ({8}/{\pi})^{1/4} b \sqrt{n} p (1-p) \sup_{\tau>0} \left\{ \tau^{1/4}  e^{-\omega^2\tau/4} \right\},
\end{align}
and it is easy to verify that the supremum in the above expression is achieved at $\tau = 1/\omega^2$, hence
\begin{equation}
    \gamma \approx  ({8}/{\pi e})^{1/4} b p (1-p) \sqrt{{n}/{\omega}}.
\end{equation}
Note that the conditions $\tau \leq \min \{\pi,3\pi/2\omega\}$ and $\tau \leq \pi$ (required for the bounds on $E_1$ and $E_2$) hold for $\tau = 1/\omega^2$, and one can verify that $E_1$ and $E_2$ are indeed negligible if $\omega$ is not too small.

\end{appendices}

\bibliographystyle{plain}
\bibliography{mybib}

\end{document}